%% file: ym.tex
\documentclass[10pt,reqno]{amsart}
\usepackage{libertine}
\usepackage{amsmath}
\usepackage{amssymb}
\usepackage{amsthm}
\usepackage{mathtools}
\usepackage{amsmath}
\usepackage{tikz-cd}
\usepackage{float}
\usepackage{wrapfig}
\usepackage{caption}
\usepackage{subcaption}
\usepackage{aligned-overset}
\usepackage[utf8]{inputenc}
\usepackage[
	colorlinks=true,
	pdfstartview=FitV,
	linkcolor=blue,
	citecolor=blue,
	urlcolor=blue
]{hyperref}
\numberwithin{equation}{section}
\input{tex2/symbols.tex}

\usepackage[normalem]{ulem}
\usepackage{amsaddr}
\author[]{\small Jonas Cassel$^{1,2}$, Fabio Schlindwein$^{1}$, Peter Albers$^{1,2}$, Christoph Schn\"{o}rr$^{1,2}$}
\address{$^1$Institute for Mathematics, Heidelberg University, Germany \\
$^2$Research Station Geometry and Dynamics, Heidelberg University, Germany}
\email{cassel@math.uni-heidelberg.de}
\date{\today}

\thanks{This work is funded by the Deutsche Forschungsgemeinschaft (DFG) under Germany's Excellence Strategy EXC-2181/1 - 390900948 (the Heidelberg STRUCTURES Excellence Cluster). This work was funded by the Deutsche Forschungsgemeinschaft (DFG), grant SCHN 457/17-2, within the priority programme
SPP 2298: Theoretical Foundations of Deep Learning.}

\makeatletter
\@namedef{subjclassname@2020}{\textup{2020} Mathematics Subject Classification}
\makeatother


\title{Yang-Mills Meets Data}
\begin{document}
\maketitle

\begin{abstract}
\textit{Gauge symmetric methods} for data representation and analysis utilize tools from the differential geometry of \textit{vector bundles} in order to achieve consistent data processing architectures with respect to local symmetry and equivariance.
In this work, we elaborate concepts of geometric gauge theory for data science. Motivated by \textit{lattice gauge theory}, we focus on discrete descriptions of vector bundles for data representation and analysis, with clear relations to the established mathematical bundle formalism. Our approach unifies various existing approaches to data processing via vector bundles, within the framework of gauge theory.
We provide geometric insights into \textit{gauge symmetric heat kernel operators} that are closely related to graph connection Laplacians, and into their data transformation properties in terms of the non-trivial nullspace of the corresponding \textit{gauged Laplacians}. In particular, we utilize a \textit{discrete Yang-Mills energy} in order to characterize significant properties of the heat kernel in terms of associated \textit{synchronization problems}.
\end{abstract}

\tableofcontents

\input{tex2/introduction.tex}
\input{tex2/dvb.tex}
\input{tex2/heatkers.tex}
\input{tex2/yangmills.tex}

\section{Conclusions}
In this paper, we extended of the available formalism of gauge invariant data representation for machine learning \cite{weilerCoordinateIndependentConvolutional2021,gerkenGeometricDeepLearning2021} using methods from lattice gauge theory. Specifically, we introduced a gauge invariant heat kernel formalism and elaborated the connection to the discrete Yang-Mills energy from lattice gauge theory.
In addition to the connection to discrete gauge theories, our results reinforce from a common viewpoint the links between equivariant deep learning \cite{weilerCoordinateIndependentConvolutional2021,gerkenGeometricDeepLearning2021}, classical data science of vector bundles \cite{Singer:2012aa,batardHeatEquationsVector2011,batardClassGeneralizedLaplacians2014,batardHighOrderDIPVBTVImage2025} and the discrete vector bundle/cellular sheaf formalism from \cite{hansenSpectralTheoryCellular2019a,bodnarNeuralSheafDiffusion2023,bambergerBundleNeuralNetworks2024}.

Motivated by other approaches to discrete fiber bundles \cite{leokDiscreteTheoryConnections2005,Gattringer:2010zz,Singer:2012aa,gaoGeometrySynchronizationProblems2019}, we emphasized a framework of discrete associated vector bundles which provides a proper representation of vector bundle concepts from differential geometry that is amenable to computational realizations. The relations of our discrete approach and the continuous vector bundle formalism are discussed in some detail in Appendix \ref{sec:app}. 

Our discrete approach has the advantage of being formulated entirely in terms of concepts from finite graph theory, that are convenient for numerical applications. It also removes some part of the theoretical overhead incurred by established mathematical vector bundle descriptions, while preserving the main features of gauge symmetry properties.

We equipped our discrete model with suitable Laplacians and associated heat kernels and analyzed them using gauge theoretic methods. We showed that our gauged Laplacian is distinguished by the property of being the only gauge invariant operator which satisfies a \textit{local Laplacian} and a \textit{ tree blow-up lifting} property, formulated as Theorem \ref{thm:axioms}.
Furthermore, we analyzed the interplay of our discrete bundle methods and the \textit{Yang-Mills formalism} from high energy physics, showing that Yang-Mills analysis provides a useful tool for \textit{synchronization analysis}, as shown by Theorem \ref{thm:eym}.

\vspace{0.25cm}
Future directions for research may further extend our approach towards machine learning, discrete gauge theory and topological graph theory.
As explained in the Section \ref{sec:hks}, the heat kernels that we studied in this work appeared in several works on machine learning on graphs \cite{bodnarNeuralSheafDiffusion2023,bambergerBundleNeuralNetworks2024,cassel2025bundle}.
In such scenarios, one is often confronted with loosely structured data such that one cannot rely on structures such as graph voltages as part of given data. Thus, a natural research task is to \textit{learn the required structures from data}. This approach amounts to interpreting the geometric structures of discrete vector bundles (in the present case, graph voltages $\omega$) as a parameter of a machine learning architecture, to be estimated by minimizing a suitable loss function.
The resulting approaches are promising candidates for removing persistent limitations of graph neural networks, e.g.~in connection with \textit{node classification on heterophilic graphs}.

In particular, the recent work \cite{bodnarNeuralSheafDiffusion2023} showed that the kernel of the associated Laplacian is a crucial quantity for understanding the characteristics of vector bundle heat kernel methods in the context of machine learning.
Our work shows that the \textit{Yang-Mills energy} could serve as a \textit{gauge invariant regularizing term} in learning scenarios, that preserves gauge invariance but also controls the kernel of the relevant Laplacian operator. Exploring these connections will also be part of our future work.

Furthermore, we next aim to analyze how discrete vector bundle concepts may profit from insights from topological graph theory about voltage graphs. In particular, we wish to connect our present analysis of the synchronizability of graph voltages to the theory of \textit{derived graphs}, which is an established concept of topological graph theory \cite{grossTopologicalGraphTheory1987}.

\appendix
\input{tex2/general.tex}
\bibliographystyle{alpha}
\bibliography{ym.bib}
\end{document}

%% file: tex2/symbols.tex
\newtheoremstyle{bfnote}%
  {}{}
  {\itshape}{}
  {\bfseries}{.}
  { }{\thmname{#1}\thmnumber{ #2}\thmnote{ (#3)}}

\newtheoremstyle{bfdef}%
  {}{}
  {}{}
  {\bfseries}{.}
  { }{\thmname{#1}\thmnumber{ #2}\thmnote{ (#3)}}

\theoremstyle{bfnote}
\newtheorem{theorem}{Theorem}[section]
\newtheorem{lemma}[theorem]{Lemma}
\newtheorem{proposition}[theorem]{Proposition}
\newtheorem{corollary}[theorem]{Corollary}

\theoremstyle{bfdef}
\newtheorem{definition}[theorem]{Definition}
\newtheorem{assumption}[theorem]{Assumption}
\newtheorem{remark}[theorem]{Remark}
\newtheorem{example}[theorem]{Example}

\newcommand{\mc}[1]{\mathcal{#1}}

\newcommand{\mbf}[1]{\mathbf{#1}}

\newcommand{\mfk}[1]{\mathfrak{#1}}

\newcommand{\la}{\langle}
\newcommand{\ra}{\rangle}
\newcommand{\lla}{\langle\langle}
\newcommand{\rra}{\rangle\rangle}

\newcommand{\R}{\mathbb{R}}

\newcommand{\veps}{\varepsilon}
\newcommand{\n}[1]{\|#1\|}

\newcommand{\II}{\mathbf{I}_{d \times d}}

\renewcommand\o\operatorname
\renewcommand\r\mathrm
\renewcommand\c\mathcal

\newcommand{\yyy}{}

\DeclareMathOperator{\dd}{d}

\DeclareMathOperator{\U}{U}
\DeclareMathOperator{\SO}{SO}
\DeclareMathOperator{\GL}{GL}

\DeclareMathOperator{\tr}{tr}

\usepackage{wasysym}

\providecommand\x; %
\renewcommand{\x}{x}

\providecommand\X; %
\renewcommand{\X}{X}

\providecommand\gt; %
\renewcommand{\gt}{\xi}

\providecommand\a; %
\renewcommand{\a}{\alpha}

\providecommand\cc; %
\renewcommand{\cc}{voltage }

\providecommand\ccs; %
\renewcommand{\ccs}{voltages }

\providecommand\XX; %
\renewcommand{\XX}{\Gamma}

\providecommand\gtt; %
\renewcommand{\gtt}{\tau}

\providecommand\E; %
\renewcommand{\E}{E}

\providecommand\Hom; %
\renewcommand{\Hom}{H}

\providecommand\Hol; %
\renewcommand{\Hol}{\mc{H}}

\providecommand\Lps; %
\renewcommand{\Lps}{\mc{L}}

\providecommand\Lpl; %
\renewcommand{\Lpl}{\Delta}

\providecommand\hgen; %
\renewcommand{\hgen}{L}

\providecommand\vg; %
\renewcommand{\vg}{\chi}

\providecommand\hol; %
\renewcommand{\hol}{\mathfrak{h}}

%% file: tex2/introduction.tex
\section{Introduction}
\subsection{Node Features on Graphs, Gauge Invariance and Heat Kernels}\label{sec:intro1}

		In this work, we analyze \textit{node feature data on graphs} using \textit{vector bundle gauge theory}. Gauge theory is a widely employed modelling paradigm in theoretical physics \cite{peskinIntroductionQuantumField1995a,hamiltonMathematicalGaugeTheory2017} and differential geometry \cite{Berline:1996aa,Jost:2017aa}.
		It describes interactions of \textit{vector bundle sections} and \textit{gauge fields}, which are geometrically represented by \textit{Ehresmann connections} and organized through \textit{gauge symmetry} principles. Figure \ref{fig:intro} illustrates the data model used in this paper.

	Node feature data on graphs form a large class of relevant data like, for instance, images \cite{Aubert:2006aa,Handbook:2023aa}, tokenized text sequences \cite{vaswani2017attention,Siino:2024aa,Jia:2025aa} or community/network data \cite{Battiston:2020aa,Boguna:2021aa}.
	Corresponding processing architectures comprise \textit{node feature transformations} acting on node features, such that the class of the underlying data whose definition depends on the architecture, is preserved. Examples include linear inverse models \cite{Tarantola:2005aa,groetschLinearInverseProblems2011,Arridge:2019wu}, transformers \cite{vaswani2017attention} and graph neural networks \cite{kipfSemiSupervisedClassificationGraph2017a,bronstein2017geometric}. 

	We apply and interpret concepts from gauge theory in the context of data science in this paper, which requires discrete formulations of gauge theory. Specifically, node feature data on graphs are represented as \textit{sections of discrete vector bundles} while \textit{discretized gauge fields} are identified with \textit{graph voltages} \cite{grossTopologicalGraphTheory1987}.
	Our work is motivated by \textit{lattice/discrete gauge theory} \cite{baezSpinNetworkStates1996,Gattringer:2010zz,jiangGaugeTheoryGraphs2023} that provides the discretization formalism for gauge theories used in theoretical physics.
	
	In addition, we propose \textit{gauge invariant heat kernel transformations} as distinguished node feature transformations that respect discrete vector bundle structures.
	 Heat kernel transformations have a long tradition in data science and image processing \cite{Lindeberg:1994aa,Lindeberg:2022aa}. From this viewpoint, we devise advanced heat kernel methods rooted in discrete gauge theory and closely related to, but different from, \textit{graph connection Laplacians} \cite{Singer:2012aa} and \textit{Laplacians of cellular sheaves} \cite{hansenSpectralTheoryCellular2019a}.

	The gauged heat kernel approach establishes links between different approaches to discrete geometric modelling in data science, like the discrete bundle approaches of \cite{hansenOpinionDynamicsDiscourse2021,bodnarNeuralSheafDiffusion2023,bambergerBundleNeuralNetworks2024}, the gauge invariant methods of \cite{weilerCoordinateIndependentConvolutional2021,gerkenGeometricDeepLearning2021} and the vector bundle PDEs in \cite{batardHeatEquationsVector2011,batardClassGeneralizedLaplacians2014,batardHighOrderDIPVBTVImage2025}.
	Thus, gauge theory appears to provide a unifying principle for gauge invariance and vector bundle methods in data science, 
	with a clear path from smooth bundle frameworks in mathematics to discrete counterparts in data science.

	Consequently, along these lines, we furthermore transfer analysis methods for gauge fields to data analysis, based on \textit{discrete Yang-Mills energies} for graph voltages that represent discrete gauge fields. Yang-Mills theory \cite{atiyah1983yang,peskinIntroductionQuantumField1995a,baumEichfeldtheorieEinfuehrungDifferentialgeometrie2014,hamiltonMathematicalGaugeTheory2017,gaoGeometrySynchronizationProblems2019} is an analysis tool from vector bundle gauge theory for characterizing gauge fields in terms of \textit{curvature}, which historically generalized notions of Maxwell's theory of electromagnetism \cite{Garrity:2015aa}.

	The relevance for data science is again established in terms of heat kernels: We show that Yang-Mills energies help to understand the data transformation characteristics of gauge invariant heat kernel transformations. In particular, this highlights how \textit{topological} properties of underling graphs shape properties of the heat kernel transformation.

		\begin{figure}
		\centering
		\includegraphics[width=1.0\textwidth]{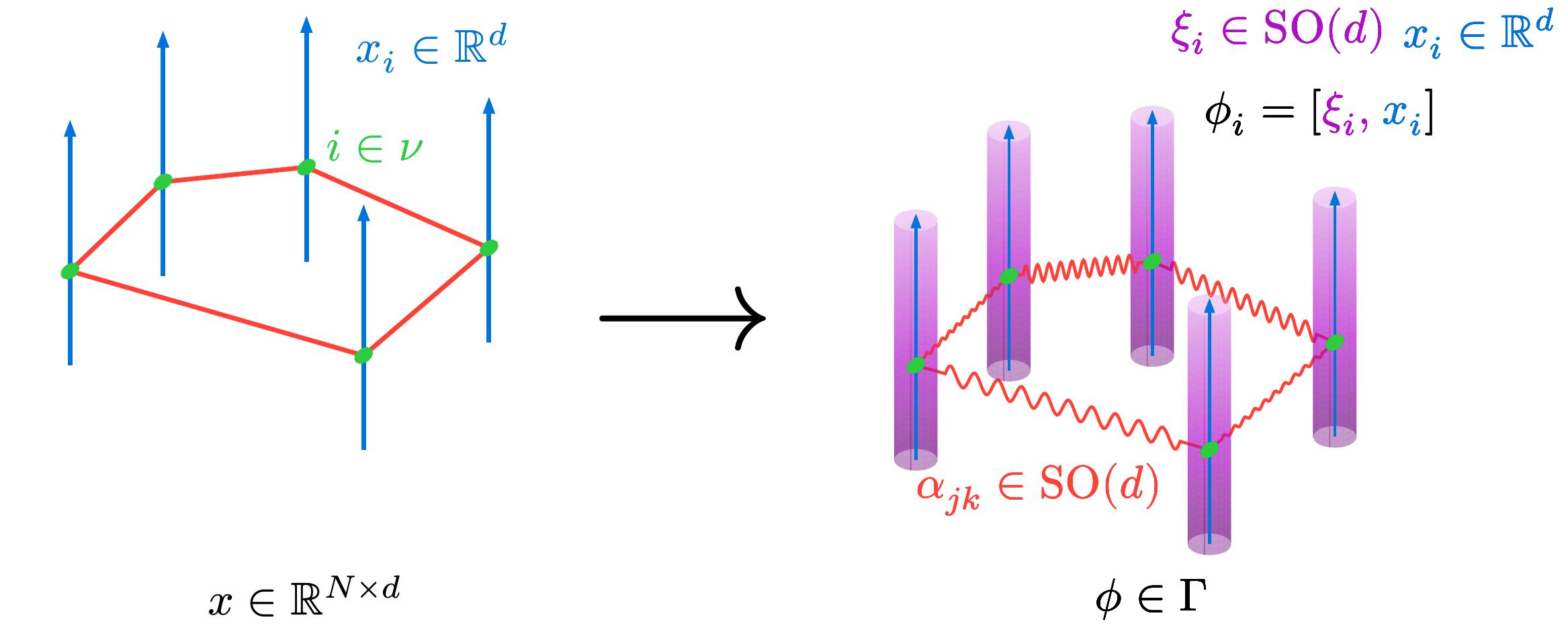}
		\caption{
			Illustration of the gauge invariant feature data model employed in this paper.
\textsf{Left:} Visualization of node feature data $x \in X =\R^{N \times d}$ over a graph $\eta$ with vertex set $\nu$ of cardinality $|\nu| = N$. Features are represented by matrices $x\in X$, with rows $x_{i}$ indexed by the vertices $i\in\nu$. \textsf{Right:} The scenario on the left changes when \textit{gauge invariant feature vectors} $\phi$ over $\eta$ are considered. The feature space $X$ is extended to the space of \textit{discrete bundle sections} $\Gamma$. A section $\phi\in\Gamma$ assigns to every node $i\in\nu$ a gauge invariant feature vector $\phi_i=[\xi_{i},x_{i}]$, where  $x_{i} \in \R^{d}$ is a regular feature vector, and $\xi_{i}$ is an element of $\SO(d)$; see Section \ref{sec:bundles} for details. Gauge invariant feature vectors $\phi_{i}$ interact with each other along the edges of the underlying graph, where the interaction is mediated in terms of \textit{graph voltages} $\alpha \in \mc{A}$ (a matrix $\alpha_{jk} \in \SO(d)$ for every edge $j\sim k$ of $\eta$) and associated \textit{gauged Laplacians} $\hgen_{\vg}$, as defined and studied in Section \ref{sec:hks}.
} 
		 \label{fig:intro}
	\end{figure}
	
\subsection{Related Work}
	\textit{Fiber bundle gauge theory:}
	This work is based on \textit{fiber bundle methods} from differential geometry \cite{steenrod1999topology,Hatcher:2002aa,hatcher2003vector,Jost:2017aa}. In particular, we make use of \textit{gauge theory}, which studies \textit{principal fiber bundles} and \textit{associated vector bundles} in terms of \textit{Ehresmann connections} \cite{baumEichfeldtheorieEinfuehrungDifferentialgeometrie2014,hamiltonMathematicalGaugeTheory2017,nakaharaGeometryTopologyPhysics2003}.

	\textit{Gauge symmetry and vector bundles in data science:}
	Smooth vector bundle methods and the concepts of gauge symmetry were introduced to data science in \cite{batardHeatEquationsVector2011,batardClassGeneralizedLaplacians2014,batardHighOrderDIPVBTVImage2025}, and \cite{weilerCoordinateIndependentConvolutional2021,gerkenGeometricDeepLearning2021}.

	\textit{Smooth and discrete physical gauge theory:}
	Physics based perspectives on geometric gauge theory can be found, e.g., in \cite{peskinIntroductionQuantumField1995a,hamiltonMathematicalGaugeTheory2017,tong2018gauge,wipfSymmetrienPhysikGruppen2023}. \cite{baezSpinNetworkStates1996,Gattringer:2010zz,rotheLatticeGaugeTheories2012,jiangGaugeTheoryGraphs2023} describe discrete notions of gauge theory which operate in terms discrete versions of connections or gauge fields.

	\textit{Discrete vector bundles in data science:}
	Discrete methods for connections on fiber bundles for discrete geometry modelling are discussed in, e.g., \cite{leokDiscreteTheoryConnections2005}, and connected to questions of data science, for instance, in works on geometric approaches to \textit{synchronization problems} \cite{Abbe:2018ab,gaoGeometrySynchronizationProblems2019,Carletti:2023aa} and \textit{vector diffusion maps} \cite{Singer:2012aa,gaoDiffusionGeometryFibre2021}.
	These notions were also extended by works on cellular sheaves in the context of \textit{opinion dynamics} \cite{hansenOpinionDynamicsDiscourse2021} and \textit{node classification} on graphs \cite{bodnarNeuralSheafDiffusion2023,bambergerBundleNeuralNetworks2024}.

	\textit{Heat kernel methods in data science: }
	The heat kernel methods that we refer to are inspired from many sources. For general heat kernel PDE frameworks for data science, see \cite{Alvarez:1993aa,Lindeberg:1994aa,Weickert:1998aa,Lindeberg:2022aa}. For heat kernels in differential geometry we refer to \cite{chavelEigenvaluesRiemannianGeometry1984,Davies:1989aa,Berline:1996aa,Grigoryan:2009aa,Avramidi:2015aa}.
	Notions of \textit{twisted/generalized Laplacians} on vector bundles as heat kernel generators can be found in \cite{Berline:1996aa,hamiltonMathematicalGaugeTheory2017}.
	Heat kernels on vector bundles have appeared in connection to image processing in \cite{batardHeatEquationsVector2011,batardClassGeneralizedLaplacians2014,batardHighOrderDIPVBTVImage2025}.
	Heat kernels in conjunction with \textit{discrete} vector bundle models have appeared in the context of \textit{vector diffusion maps} \cite{Singer:2012aa}, in the form of \textit{graph connection Laplacians}.
	These perspectives have extended into the realm of deep learning in \cite{bodnarNeuralSheafDiffusion2023,bambergerBundleNeuralNetworks2024}, where graph connection kernels are used as a tool for deep learning on graphs and node classification.
	Another perspective on graph connection Laplacians is provided by the \textit{twisted Hodge theory} of associated fiber bundles as presented in \cite{gaoGeometrySynchronizationProblems2019,gaoDiffusionGeometryFibre2021}.

	\textit{Yang-Mills theory:}
	Yang-Mills theory was developed in the context of non-abelian gauge theories in theoretical physics \cite{peskinIntroductionQuantumField1995a,baumEichfeldtheorieEinfuehrungDifferentialgeometrie2014,hamiltonMathematicalGaugeTheory2017}. Discrete notions have been introduced primarily on regular lattices \cite{Gattringer:2010zz}, but more general discrete notions have also been developed \cite{jiangGaugeTheoryGraphs2023}. Recently, other approaches to discrete Yang-Mills theory have been developed, diverging from the lattice gauge theory approach; see \cite{sushch2DDiscreteYang2024} and references therein. 

		Additional references on vector bundle methods in applied mathematics can be found here \cite{batardCliffordBundlesCommon2010,ohsawaSymmetryReductionOptimal2013,batardCovariantDerivativesTheir2014,budninskiyParallelTransportUnfolding2019,huroyanSolvingJigsawPuzzles2020}. Also relevant in the context of voltage graphs are generalizations of \textit{signed graphs} in network theory, see e.g.~\cite{Zaslavsky:1982aa,cloningerRandomWalksConductance2024,tianMatrixweightedNetworksModeling2024,tianStructuralBalanceRandom2024,evansSignedNetworksGroup2025}. See also \cite{favoniLatticeGaugeEquivariant2022,aronssonHomogeneousVectorBundles2022,aronssonGeometricalAspectsLattice2023a}, where machine learning frameworks for lattice gauge theory problems have been developed.

\subsection{Contribution}
	Our work extends the perspectives on gauge symmetry for data science adopted by \cite{weilerCoordinateIndependentConvolutional2021} and \cite{gerkenGeometricDeepLearning2021}, from a complementary point of view, in that we work out the details and implications of gauge invariant methods when applied to node features over a finite graph.
	In addition, we compare our discrete theory to smooth approaches to gauge invariance and specify explicitly \textit{how} smooth and discretized gauge theory interact within our formalism.
	Our approach removes technical requirements of fiber bundle theory and differential geometry from the gauge invariant frameworks of \cite{weilerCoordinateIndependentConvolutional2021,gerkenGeometricDeepLearning2021,batardHeatEquationsVector2011,batardClassGeneralizedLaplacians2014,gaoGeometrySynchronizationProblems2019}, by formulating results solely in terms of finite graph theory, which should make the framework accessible to a wider audience.

	We identify heat kernel methods as the \textit{key connection} between filter/kernel based approaches for vector bundle data analysis as in \cite{gerkenGeometricDeepLearning2021}, \cite{weilerCoordinateIndependentConvolutional2021} and the approaches in \cite{batardHeatEquationsVector2011,batardClassGeneralizedLaplacians2014} motivated by PDE methods.
	This helps to interpret the \textit{steerable kernel methods} of \cite{weilerCoordinateIndependentConvolutional2021,jennerSteerablePartialDifferential2022} in the context of differential operators on vector bundles, so that we can connect the vector bundle PDE methods of \cite{batardClassGeneralizedLaplacians2014} and the perspective on deep learning adopted in \cite{weilerCoordinateIndependentConvolutional2021}.
	Furthermore, our discrete gauge theoretic approach provides a perspective under which other discrete vector bundle data processing approaches may be related to the framework of \cite{weilerCoordinateIndependentConvolutional2021}, such as the synchronization bundle methods of \cite{gaoGeometrySynchronizationProblems2019} and the cellular sheaf methods of \cite{hansenOpinionDynamicsDiscourse2021,bodnarNeuralSheafDiffusion2023,bambergerBundleNeuralNetworks2024}.
	Thus our formalism `interpolates' between discrete and smooth vector bundle approaches to data science and establishes gauge theory as the organizing first principle.

	Based on these relations between data science and gauge theory, we extend the \textit{synchronization analysis} of \cite{gaoGeometrySynchronizationProblems2019} to a discrete Yang-Mills analysis and show the connection to heat kernel approaches.
	Additionally, we apply gauge theoretic arguments in order to analyze the data transformation properties of our heat kernel formalism. We show how the underlying graph structure shapes the nullspace of the generator of our heat kernel, in terms of a \textit{spanning tree isomorphism}. In particular, we state axioms which uniquely identifies our heat kernel operators among all gauge invariant data transformations.

	To summarize, we complement the existing apparatus for gauge symmetry in data science with concepts from discrete and smooth gauge theory. In particular, we show how tools from Yang-Mills theory integrate into data science.

\subsection{Organization}
	Section \ref{sec:bundles} introduces gauge symmetry for discrete data and its implications for data processing.

	Section \ref{sec:hks} defines gauge symmetric heat kernel methods for discrete vector bundles generated by gauged Laplacian operators, which are induced by graph voltages. These gauged Laplacians are closely related to graph connection Laplacians \cite{Singer:2012aa} and Laplacians of cellular sheaves \cite{hansenSpectralTheoryCellular2019a} and discrete twisted Hodge Laplacians \cite{gaoGeometrySynchronizationProblems2019}.
	We show furthermore that these specific Laplacians arise naturally when gauge invariance conditions are imposed.

	In Section \ref{sec:transform}, we analyze the data transformation properties of the gauged heat kernel, which shows that the structure of the kernel of the gauged Laplacian  closely relates to the synchronizability of the underlying voltage graph.

	We then proceed and connect these transformation properties of the heat kernel to Yang-Mills analysis of voltage graphs in Section \ref{sec:ym}. We show that the Yang-Mills energy reveals whether a given voltage is synchronizable in terms of scalar quantities, which is relevant for the data transformation characteristics of the gauged heat kernel.

	A discussion of the relationship between smooth and discrete gauge theory in Appendix \ref{sec:app} complements our paper.

%% file: tex2/dvb.tex
\section{Gauge Invariant Representations of Discrete Data}\label{sec:bundles}
	This section first recalls the concept of \textit{gauge invariance} for data science as introduced in \cite{batardHeatEquationsVector2011,weilerCoordinateIndependentConvolutional2021,gerkenGeometricDeepLearning2021}. Then we extend this perspective and provide additional interpretations for the formalism of associated vector bundles, in order to highlight its significance for data science.

	We focus on gauge invariance concepts for \textit{discrete data}. That is, we do not impose smoothness or manifold assumptions on our data. We relate our approach to smooth vector bundle theory in Appendix \ref{sec:app}.

\subsection{Notation}
	
	Let $\eta$ be a finite, \textit{connected}, simple, directed graph with vertex set $\nu$ and directed edge relation $<$, i.e. we write $i < j$ if there is a directed edge from node $i$ to node $j$. We also denote by $\veps$ the edge set of $\eta$ and by $\sim$ the \textit{symmetrization} of $<$, that is $i \sim j$ if $i < j$ or $j < i$ for $i,j \in \nu$. We write also $i>j \in \veps$ and $i\sim j \in \veps$ to address the edges of a graph $\eta$.

	The symbol $N \coloneqq | \nu |$ denotes the number of graph nodes such that set $\nu$ can be identified with the set $\{1,\dots , N\}$ after fixing a total order of the vertices. 
	
	We consider any fixed dimension $d > 0$ and define the set
	\begin{equation}\label{eq:def-data-space-X}
		\X = \{ \x : \nu \to \R^{d} \} = (\R^{d})^{N},
	\end{equation}
	of $d$-dimensional vectors indexed by the nodes of the graph $\eta$. We equip $\R^{d}$ with the Euclidean scalar product $\la v,w\ra = v^\top w$, such that $\X$ also becomes an inner product vector space.  Elements $\x \in \X$ are represented by tuples $x = (x_1, \ldots, x_N)$ with $\x_i \in \R^d$ for $i \in \nu$. We identify $X$ with the space of matrices $\R^{N\times d}$ by identifying the tuple $x = (x_i)_{i \in \nu}$ with the $N \times d$ matrix whose rows are given by $\x_i$. 

	We assume furthermore that the graph $\eta$ comes along with a symmetric weight function
	\begin{equation}
		\kappa : \nu \times \nu \to \R_{\geq 0}, \quad \kappa_{ij} = \kappa_{ji}, \quad \kappa_{ij} = 0 \quad \text{if} \quad i \not\sim j,
	\end{equation}
	that is, $\kappa$ is a positive weight function on the symmetrized (undirected) edges of the graph.

\subsection{Gauge Symmetry for Discrete Data}
	This subsection introduces coordinate independence for data modelling as in \cite{weilerCoordinateIndependentConvolutional2021,gerkenGeometricDeepLearning2021}, here adapted to the case of \textit{discrete} data. Here `discrete' refers to a graph $\eta$ that supports the data, rather than to the range of the data that are real-valued. Suppose a data space $X$ of the format \eqref{eq:def-data-space-X} is given. We want to transform data samples $\x \in \X$ in terms of a function $f : \X \to \X$.

	\textit{Gauge invariance} entails a change of perspective on the feature vector space $X=\R^{N \times d}$. We postulate that there exists an abstract vector space $\XX$, which is the natural space for representing given data and that $\XX \cong \X$ after choosing a basis of $\XX$.
	A notable point here is that there are many choices for isomorphisms $\XX \cong \X$ but no canonical one among them. Rather, the space $\XX$ represents all the bases of $\X$ simultaneously.
	Elements $\phi \in \XX$ are called \textit{gauge invariant} or \textit{basis independent representatives} of samples $\x \in \X$. We call mappings $\mc{F} : \XX \to \XX$ \textit{gauge invariant transformations}. They represent \textit{basis independent transformations} of the space $\XX$.

	The following assumption is adopted throughout.
	\begin{assumption}[gauge invariance]
		Data $\x \in \X$ must be represented in terms of gauge invariant representatives $\phi \in \XX$ and are transformed by gauge invariant mappings $\mc{F} : \XX \to \XX$.
	\end{assumption}
	The concept of gauge invariance and its role in data processing has been motivated and discussed in the literature \cite{batardHeatEquationsVector2011,weilerCoordinateIndependentConvolutional2021,gerkenGeometricDeepLearning2021}.
	We point out that our analysis of gauge invariance is not aimed at facilitating numerical computations. Rather, we are interested in the \textit{consistency conditions} resulting from assumptions about gauge invariance and their significance in the context of data science.

	Abstract vector spaces $\XX$ are not directly amenable for practical applications, since no representation suggests itself that is suitable for numerical computation. However, since the basis independence assumption stipulates that the abstract space $\XX$ is the correct domain for analyzing data, we are forced to find a way of working with $\XX$.
	
	A solution is provided by \textit{gauge theory} known from differential geometry and theoretical physics \cite{hamiltonMathematicalGaugeTheory2017}, due to the observation that the space of all possible choices of bases $\X$ is given by the general linear group $\GL(X)$. Gauge theory considers a \textit{concrete model} for the space $\XX$ of the form
	\begin{equation}\label{eq:smashed-prod}
		\begin{aligned}
			\XX
			&= (\GL(X) \times X) / \GL(X) \\
			&= \{ [g,\x ] \subset \GL(X) \times \X \mid [gh,\x ] = [g,hx] ,\; h \in \GL(X) \}.
		\end{aligned}
	\end{equation}
	We note here that the group $\GL(\X)$ acts naturally on $\X$, since $\GL(X)$ consists of linear isomorphisms of $X$. This action can be understood more concretely if we identify $X=\R^{N\times d}$ with $\R^{Nd}$ via the row stacking isomorphism $\o{vec}$. The isomorphism $\o{vec}$ also identifies $\GL(X)$ with $\GL_{Nd}(\R)$, which we denote by $g \in \GL(\X) \mapsto \tilde{g} \in \GL_{Nd}(\R)$. The action of $\GL(\X)$ on $\X$, denoted by $\GL(\X) \curvearrowright \X$, is then given by
	\begin{equation}\label{eq:action-glx}
		\GL(\X) \times \X \to \X, \quad (g,x) \mapsto gx \coloneqq \o{vec}^{-1}(\tilde{g}\o{vec}(x)).
	\end{equation}

	The set $\XX$ forms a vector space \textit{isomorphic} but \textit{not equal} to $\X$, with respect to the operations
	\begin{equation}\label{eq:Gamma-operations}
		[g,x] + [g,x'] = [g,x + x'], \quad
		r [g,x] = [g,rx], \quad \forall [g,x],[g,x'] \in \XX ,\; r \in \R.
	\end{equation}
	Modelling data in terms of the space $\Gamma$ imposes the constraint that vectors $\x \in \X$ must come along with an explicit choice of basis $g \in \GL(\X)$ in order to be well-defined. The space $\XX$ is canonically isomorphic to $X$ in terms of the isomorphism
	\begin{equation}
		\XX \to X, \quad [g,x] \mapsto gx. 
	\end{equation}
	\begin{remark}[$\XX$ versus $\X$]
		Despite the existence of this canonical isomorphism $\Gamma \cong \X$, we do not identify the two spaces. The goal of gauge invariant data modelling entails to keep track of \textit{how} to identify $\XX$ and $\X$. In particular, \textit{two distinct roles} of $\GL(X)$  in $\Gamma = (\GL(\X) \times X) / \GL(X)$ are considered; it appears as a group and as a \textit{principal homogeneous space}. 
	\end{remark}
	Additionally, we would like to point out that $\Gamma$ represents a special case of a more general modelling scheme; similar constructions can be made when replacing the group action $\GL(X) \curvearrowright X$, see also Remark \ref{rem:general-groups}. 

	The concrete model \eqref{eq:smashed-prod} of $\XX$ now allows to describe gauge invariant representatives $\phi \in \XX$ explicitly.
	The elements $[g,\x ] \in \XX$ are equivalence classes of bases and vectors, identified by the relation $[gh,h^{-1}\x ] = [g,\x ]$, which means that if we transform the basis $g$ along with the vector $\x$, then we stay in the same equivalence class.
	A data sample $\x$ can be promoted to a gauge invariant representation by choosing a basis $g \in \GL(X)$ and setting
	\begin{equation}
		\phi \coloneqq [g,\x].
	\end{equation}

	By means of the concrete model \eqref{eq:smashed-prod} of $\XX$, we can give a complete characterization of transformations $\mc{F} : \XX \to \XX$.
	\begin{proposition}[gauge invariant transformations]\label{prop:g-inv-trafos}
		The gauge invariant transformations $\mc{F} : \XX \to \XX$ correspond one-to-one to the functions
		\begin{equation}
			F : \GL(X) \times \X \to \X
		\end{equation}
		with equivariance property
		\begin{equation}\label{eq:consis}
			F(gh^{-1},hx) = h F(g,\x ), \quad \forall h \in \GL(\X).
		\end{equation}
	\end{proposition}

	\begin{proof}
		\renewcommand\pluto\Sigma
		We denote in this proof by $\pluto_{g}$ the isomorphism $\Gamma \to X, [h,x] \mapsto g^{-1}hx$, for any fixed $g \in \GL(\X)$. This implies $\Sigma_g([g,x]) = x$ and $\Sigma_{gh^{-1}} = h\Sigma_{g}$, since
		\begin{equation}\label{eq:sigma-gh}
			\Sigma_{gh^{-1}}([h',x]) = h g^{-1} h' x = h \Sigma_{g}([h',x])  
		\end{equation}
		for all $g,h \in \GL(X)$ and $[h',x] \in \XX$.

		Let $\mc{F}$ be a given gauge invariant transformation $\Gamma \to \Gamma$. Then define $F$ by
		\begin{equation}
			F : \GL(\X) \times \X \to \X; \quad F(g,x) = \pluto_g(\mc{F}([g,\x])).
		\end{equation}
		$F$ defined by this rule satisfies
		\begin{subequations}
		\begin{align}
			F(gh^{-1},hx)
			&= \pluto_{gh^{-1}}(\mc{F}[gh^{-1},hx]) \\ 
			\overset{\eqref{eq:smashed-prod}}&{=} \pluto_{gh^{-1}}(\mc{F}([g,x])) \\ 
			\overset{\eqref{eq:sigma-gh}}&{=} h \pluto_{g}(\mc{F}([g,x])) \\ 
			&= h F(g,x)
		\end{align}
		\end{subequations}
		This implies relation \eqref{eq:consis}.

		As for the converse direction, assume that $F : \GL(\X) \to \X$ satisfying \eqref{eq:consis} is given. Then define
		\begin{equation}
			\mc{F} : \Gamma \to \Gamma, \quad [x,g] \mapsto [g,F(g,x)].
		\end{equation}
		We establish that $\mc{F}$ is well-defined,
		\begin{subequations}
		\begin{align}
			\mc{F}([gh,h^{-1}x]) &=
			[gh,F(gh,h^{-1}x)] \overset{\eqref{eq:consis}}{=}
			[gh,h^{-1}F(g,x)] \overset{\eqref{eq:smashed-prod}}{=}
			[g,F(g,x)] \\ &= \mc{F}([g,x]).
		\end{align}
		\end{subequations}

		These two constructions are clearly inverse to each other.
	\end{proof}
	
	\begin{corollary}[characterization of gauge invariant transformations]\label{cor:ginv-trafo}
		The set of gauge invariant transformations $\mc{F} : \Gamma \to \Gamma$ corresponds one-to-one to the space of mappings $\X \to \X/\GL(X) $, where $\X / \GL(X) \coloneqq \big\{[x] \in \X \mid [x] = [gx] ,\; \forall g \in \GL(X)\big\}$ is the space of orbits of the group action $\GL(\X) \curvearrowright \X$.
	\end{corollary}

	\begin{proof}
		The gauge invariant transformations $\mc{F}$ correspond one-to-one to equivariant functions $F : \GL(\X) \times \X \to \X$, due to Proposition \ref{prop:g-inv-trafos}. Equivariant functions are fully characterized by their induced functions $(\GL(\X) \times \X) / \GL(\X) \to \X / \GL(X)$ on the corresponding orbit spaces.

		Consequently, since the space $\Gamma = (\GL(\X) \times \X) / \GL(\X)$ is isomorphic to $\X$ as argued in connection with \eqref{eq:Gamma-operations}, a bijection exists between the spaces of gauge invariant transformations $\Gamma \to \Gamma$ and functions $\X \to \X / \GL(\X)$.
	\end{proof}

	Corollary \ref{cor:ginv-trafo} says that $\GL(\X)$-gauge invariant transformations $\mc{F} : \Gamma \to \Gamma$ correspond to functions $f : \X \to \X / \GL(\X)$. The space of orbits  $\X / \GL(\X)$ contains only two elements, however: the orbit of 0 and the orbit of nonzero elements. Consequently,  the transformations $\mc{F}$ can be understood as binary classification functions $f : X \to \{0,1\}$.

	This reveals that $\GL(\X)$-gauge invariance imposes conditions on the space of admissible feature transformations which restrict the applicability of $\GL(\X)$-gauge invariant modelling to binary classification tasks.

	\begin{remark}[subgroup-gauge invariance]\label{rem:sub-ginv}
		We can achieve broader applicability of the basis independence assumption by considering gauge invariance under suitable \textit{subgroups} $\Xi \subset \GL(\X)$. Subgroup-gauge invariance changes the structure of the gauge-invariant feature vector representation, see \eqref{eq:def-Gamma}. Yet, the statements of Proposition \ref{prop:g-inv-trafos} and Corollary \ref{cor:ginv-trafo} remain valid after replacing $\GL(\X)$ by $\Xi$.
		
		For instance, if we consider the subgroup $\Xi = \SO(X) \subset \GL(X)$, we can still construct the space $\Gamma = (\Xi \times X) / \Xi$, but the mappings $\mc{F} : \Gamma \to \Gamma$ are now parametrized by the orbit space $X / \Xi = \R_{\geq 0}$, which projects to $X$ via $x \mapsto \|x\|$.
	\end{remark}

	The following Definition \eqref{def:product-metric-structures} and Assumption \ref{as:sod-ginv} below prepare our choice of the subgroup $\Xi = \SO(d)^N$ -- cf.~\eqref{eq:choice-Xi} -- for the remainder of this paper.

\subsection{Discrete Vector Bundles}
	So far, we considered the feature space $\X$ as a vector space. We point out additional structures carried by $\X$.
	\begin{enumerate}
		\item $\X = \R^{N \times d} = (\R^d)^N$ is a product of $N$ copies of $\R^d$.

		\item $\X$ is an inner product space through
			\begin{equation}
				\lla x,x' \rra = \o{Tr}(x^{\top}x'),
			\end{equation}
			compatible with the product structure, that is $\lla \cdot,\cdot \rra$ induces an inner product on the factors of $\X$.
	\end{enumerate}
	Now we define gauge invariant feature spaces that respect these structures of $\X$.

	\begin{definition}[discrete vector bundle]\label{def:product-metric-structures}
		A \textit{discrete vector bundle} is a vector space $\Gamma$ with a decomposition
		\begin{equation}
			\Gamma = \bigtimes_{i=1}^N \Gamma_0
		\end{equation}
		where $N \in \mathbb{N}$ and $\Gamma_0$ is another vector space. For an element $\phi \in \Gamma$, we write $\phi_i \in \Gamma_0$ for the projection to the $i$-th factor of $\phi$. 

		A \textit{compatible metric} on $\Gamma$ is a scalar product $\langle \langle \cdot, \cdot \rangle\rangle$ that respects the product decomposition. This means that we can write
		\begin{equation}
			\langle \langle \phi,\phi' \rangle \rangle
			= \sum_{i=1}^N \langle \phi_i , \phi_i' \rangle,
		\end{equation}
		where $\langle \cdot , \cdot \rangle$ is a scalar product on $\Gamma_{0}$.

		A \textit{compatible isometry} between two discrete vector bundles $\Gamma,\Gamma'$ and compatible metrics is a vector space isometry $\Gamma \cong \Gamma'$ that induces an isometry on each factor.
	\end{definition}
	
	Recall that $X = (\R^d)^N$ is itself a discrete vector bundle with compatible metric, with factors $\R^d$. These notions lead to a modified version of the gauge invariant modelling assumption, if we impose gauge invariance only for the subgroup $\Xi \subset \GL(\X)$ of automorphisms of $X$ compatible with the discrete vector bundle and metric structure of $\X$.
	\begin{assumption}[$\Xi$-gauge invariance]\label{as:sod-ginv}
		We postulate that there exists an abstract vector space $\Gamma$ which admits a discrete vector bundle structure and a compatible metric, and is isomorphic to $\X$ via a compatible isometry. Data $\x \in \X$ must be represented in terms of gauge invariant representatives $\phi \in \Gamma$ and transformed by transformations $\mc{F} : \Gamma \to \Gamma$, as characterized by Proposition \ref{prop:g-inv-trafos}.
	\end{assumption}
	This modified assumption expresses that data can be modeled in a basis invariant way, via an abstract vector space $\Gamma$ that respects the product decomposition of $\X$ and its metric structure. In view of Remark \ref{rem:sub-ginv}, transformations $\mc{F} : \Gamma \to \Gamma$ are called \textit{$\Xi$-gauge invariant transformations}.

	Similar to Equation \eqref{eq:smashed-prod}, we specify a concrete model for $\Gamma$ by noticing that the group of factor preserving isometries of $\X$ is given by
	\begin{equation}\label{eq:choice-Xi}
		\Xi \coloneqq \SO(d)^N \subset \GL(\X),
	\end{equation}
	which acts on $\X = (\R^d)^N$ by applying factorwise rotations. The group $\Xi$ can be embedded into $\GL(\X)$, after identification with $\GL_{Nd}(\R)$, via block matrices with $d \times d$ rotation matrices on the main diagonal and matrices 0 elsewhere. $\Xi$ represents the orthonormal bases of $\X$ that respect the product structure, and there is a group action $\Xi \curvearrowright \X$ analogous to \eqref{eq:action-glx}.

	\begin{remark}[local symmetry properties]
		This property of gauge invariant transformations is often interpreted as a \textit{equivariance under local symmetry transformations} \cite{batardClassGeneralizedLaplacians2014,hamiltonMathematicalGaugeTheory2017,weilerCoordinateIndependentConvolutional2021,wipfSymmetrienPhysikGruppen2023,gerkenGeometricDeepLearning2021}.
		`Local' here refers to the fact that the group $\Xi = \SO(d)^{N}$ acts with rotation matrices $\xi_{i} \in \SO(d)$ \textit{independently} at each graph node $i \in \nu$. In other words, a group element $\xi \in \Xi$ `localizes' at each node as $\xi_i$, such that the overall decomposition of $\X$ into factors is preserved. See Figure \ref{fig:rotation} for an example and a visualization of such a group action.

		In the context of theoretical physics, this observation leads to \textit{gauge theory} as an extension of Maxwell's theory of electrodynamics \cite{peskinIntroductionQuantumField1995a,Garrity:2015aa,hamiltonMathematicalGaugeTheory2017}.

		On the other hand, these symmetry properties were also central parts of the works \cite{batardClassGeneralizedLaplacians2014,weilerCoordinateIndependentConvolutional2021,gerkenGeometricDeepLearning2021} where they served as a motivation to introduce gauge invariance and equivariance in data science. In particular, it was recognized that local symmetries are important when we consider transformation properties of data sets, and it was demonstrated that local symmetry-aware modelling architectures can improve performance in various applications.
	\end{remark}

	\begin{figure}[t]
		\begin{minipage}{0.4\textwidth}
			\includegraphics[width=\textwidth]{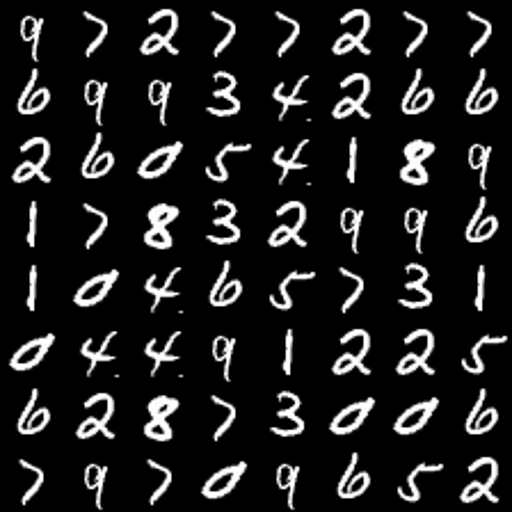}
		\end{minipage}
		\hspace{1cm}
		\begin{minipage}{0.4\textwidth}
			\includegraphics[width=\textwidth]{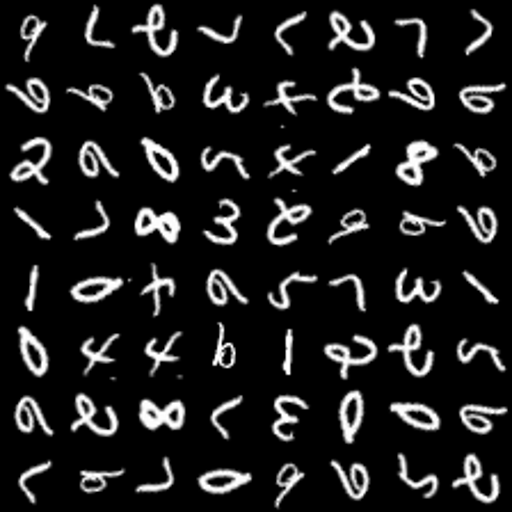}
		\end{minipage}
		\caption{
			Local rotations applied to an array of MNIST digits. The relevant graph here is a 8 $\times$ 8 grid graph with node feature vectors in $\R^{784}$ (visualized by MNIST images). Thus, the relevant feature space is given by $X = \R^{64 \times 784}$ and the local rotation is performed by acting with the group $\SO(784)^{64}$, or more precisely, via a $\U(1)^{64}$ subgroup of $\SO(784)^{64}$.
		}
		\label{fig:rotation}
	\end{figure}

	\begin{definition}[discrete bundle sections\yyy]\label{def:discrete-bundles}
		We define the space of \textit{discrete bundle sections} as
		\begin{equation}\label{eq:def-Gamma}
			\begin{aligned}
				\Gamma
				&= (\Xi \times X) / \Xi \\
				&= \{ [\gt,\x ] \in \Xi \times \X \mid [\gt \gtt,\x ] = [\gt,\gtt \x ] ,\; \forall \gtt \in \Xi\}.
			\end{aligned}
		\end{equation}
	\end{definition}

	\begin{remark}[connection to the fiber bundle formalism]
		The above definition of the abstract vector space $\Gamma$ connects to the general framework of \textit{associated vector bundles} (cf., e.g., \cite[pp.~87-88]{Jost:2017aa}).
		Given a principal $G$-bundle $P\to M$ with smooth manifolds $P$ and $M$ and structure group $G$, and a vector space $V$ on which $G$ acts linearly from the left, the \textit{associated vector bundle} $E\to M$ with fiber $V$ is constructed as follows. One has a free action of $G$ on $P\times V$ from the right,
		\begin{equation}
			P\times V\times G \to P\times V,\qquad
			(p,v)\cdot g = (p\cdot g, g^{-1} v).
		\end{equation}
		Dividing out the $G$-action by identifying $(p,v)$ and $(p,v)\cdot g$, the fibers of the bundle $(P\times V)/G \to M$ become vector spaces isomorphic to $V$, and
		\begin{equation}
			E \coloneqq (P\times V)/G \to M
		\end{equation}
		is the associated vector bundle with fiber $(G\times V)/G = V$.

		The relationship between the discrete bundle sections $\Gamma$ and the continuous framework is explained in Appendix \ref{sec:app}, where we show that $\Gamma$ can be thought of as a discretization of the space of sections $\Gamma(E)$.
		We point out that we do not require the graph structure so far. The constructed objects only rely on the finite set $\nu$. In the next section, however, the graph structure will be necessary in order to define graph voltages and Laplacians.
	\end{remark}

	We can write $\Gamma$ also in the form
	\begin{equation}\label{eq:gamma}
		\Gamma = \big((\SO(d) \times \R^{d}) / \SO(d)\big)^N
	\end{equation}
	which clarifies that $\Gamma$ has indeed a discrete vector bundle structure. A compatible metric can be defined as
	\begin{equation}\label{eq:def-norm-gamma}
		\lla \phi , \phi ' \rra =
		\sum_{i=1}^N \la \gt_i \x_i, \gtt_i \x '_i \ra
		\quad \text{where} \quad
		\phi = [\gt,\x ],\; \phi' = [\gtt,\x '].
	\end{equation}
	Importantly, this choice is well-defined and does not depend of the representatives $\x ,\x '$.

	By virtue of Proposition \ref{prop:g-inv-trafos}, Corollary \ref{cor:ginv-trafo}, the choice \eqref{eq:choice-Xi} and the product-compatible form \eqref{eq:gamma}, we can completely characterize $\Xi$-gauge invariant mappings $\mc{F}:\Gamma \to \Gamma$ as functions $f : \X \to \X / \Xi$. In contrast to the $\GL(\X)$ case, the orbit space $\X / \Xi$ now is bijective to $\R_{\geq 0}^{N}$. This means that the class of $\Xi$-gauge invariant transformations $\mc{F}$ effectively represents the class of functions $f : \X \to \R_{\geq 0}^{N}$.
	
	Thus, by restriction to the subgroup $\Xi \subset \GL(\X)$, we significantly increase the class of transformations that are representable in a gauge invariant manner, from $\X \to \{0,1\}$ to $\X \to \R_{\geq 0}^{N}$.
	
	\begin{remark}[general group representations]\label{rem:general-groups}
		We present here only a snapshot of the theory of discrete associated vector bundles, where one would consider general subgroups $\Xi \subset \GL(\X)$  induced by group representations 
		\begin{equation}\label{eq:rho-general}
		\rho: G \to \GL_{d}(\R).
		\end{equation}
		Such a representation induces an abstract vector space
		\begin{equation}
			\Gamma = (G^{N} \times X) / G^{N},
		\end{equation}
		which also admits a product structure compatible with the product structure of $\X$.

		We focus in this work on the special group representation $\SO(d) \to \GL_{d}(\R)$, since we aim for a simplified exposition in terms of more concrete arguments. Generally, the smaller the subgroup $G$, the larger is the space of data transformations compatible with gauge invariance. For a more general treatment, we refer to \cite{hamiltonMathematicalGaugeTheory2017,weilerCoordinateIndependentConvolutional2021,gerkenGeometricDeepLearning2021}.

		The theory that we develop in the remainder of this paper, however, holds more generally for compact, connected Lie groups $G$ with $d$-dimensional orthogonal representations.
	\end{remark}

	In the next section, we define \textit{heat kernel methods} compatible with the notion of gauge invariance. This will allow us to connect the gauge invariant formalism to the theory of discrete connections \cite{leokDiscreteTheoryConnections2005,Gattringer:2010zz} and associated Laplacians \cite{Singer:2012aa}.

%% file: tex2/heatkers.tex
\section{Gauge Invariant Heat Kernel Operators for Data Science}\label{sec:hks}

We define in this section heat kernel methods for data transformation based on the gauge invariant modelling method presented in the preceding section. The general literature on heat kernels is vast. We refer, e.g., to \cite{Davies:1989aa,Berline:1996aa,Grigoryan:2009aa,Avramidi:2015aa} for general expositions.

\subsection{Heat Kernel Approaches to Data Transformation}
	We first collect few basic concepts about generalized heat kernels and give examples where such transformations appear in data science.
	In this subsection, we assume that data samples to be transformed are elements of a Hilbert space $(V,\la \cdot, \cdot \ra)$, for which the node feature space $\X$ is an example.

	\begin{definition}[generalized heat kernel]
		A \textit{generalized heat kernel} $K^t : V \to V$ is a one-parameter family of continuous linear operators, indexed by $t \in \R_+$, with the following properties.
		\begin{enumerate}
			\item There exists a continuous operator $\hgen : V \to V$ such $K^t = e^{-t \hgen}$ holds. $\hgen$ is called the \textit{generator}.

			\item The generator $\hgen$ is positive semi-definite and self-adjoint.
		\end{enumerate}
	\end{definition}
	The superscript $t$ on $K$ is suppressed in the following.

	The importance of such heat kernel transformations for data science is evident.
	An example is provided by the graph heat kernel $K_{\eta,\kappa}$ associated to a weighted graph $(\eta,\kappa)$ and acting on the space of node features $\X$ over the graph \cite{Grigoryan:2011aa}. The generator of $K_{\eta,\kappa}$ is the weighted graph Laplacian matrix $\Delta_{\eta,\kappa}$. The graph heat kernel is of great interest for node labelling tasks on graphs, such as spectral clustering methods \cite{Chung:1997aa,Luxburg:2007aa} and graph neural networks \cite{kipfSemiSupervisedClassificationGraph2017a,bodnarNeuralSheafDiffusion2023,bambergerBundleNeuralNetworks2024}.

\subsection{Gauged Heat Kernel (GHK)}
	We now introduce heat kernel operators $K_\chi$ derived from discrete gauge theory \cite{Gattringer:2010zz,jiangGaugeTheoryGraphs2023}. An important property of these operators is that they are consistent with $\Xi$-gauge symmetry while acting on the abstract space $\Gamma$.
	In order to define these heat kernels we need to equip $\eta$ with more structure which is required to generate the relevant operators.

	\subsubsection{Graph Voltages}
		Recall that we assume a graph $\eta$ to be given with vertex set $\nu = \{1,\ldots,N\}$ and directed edge relation $<$.

		\begin{definition}[graph voltages \cite{grossTopologicalGraphTheory1987}]
			A \textit{$\SO(d)$ graph voltage $\a$} assigns to each positively oriented edge $i > j$ of $\eta$ an element $\a_{ij} \in \SO(d)$ and the inverse matrices $\a_{ij}^{-1}$ to negatively oriented edges $i < j$. This results in the function
			\begin{equation}
				\a : \nu \times \nu \to \SO(d), \qquad
				\begin{dcases}
					\a_{ij} = \II, &\text{if $i \not \sim j$,} \\
					\a_{ij} = \a_{ji}^{-1}, \quad &\text{else}.
				\end{dcases}
			\end{equation}
			We denote the space of graph voltages by $\mc{A}$. We denote by $\mbf{1}_{\mc{A}}$ the voltage assigning the identity matrix to all edges.
		\end{definition}

		We define additionally a group action of $\Xi$ on the space of voltages in order to define gauge invariant voltages.
		\begin{definition}[gauge invariant voltages]
			Given a voltage $\a \in \mc{A}$ and a gauge transformation $\gt \in \Xi$, we can define a new voltage $\gt \a \in \mc{A}$ defined by
			\begin{equation}\label{def-action-Xi-A}
				(\xi\alpha)_{ij}=\; \gt_i \a_{ij} \gt_j^{-1}, \qquad \forall i,j \in \nu.
			\end{equation}
			This defines a group action of $\Xi$ on $\mc{A}$ and allows us to define the space of \textit{gauge invariant voltage representations}
			\begin{equation}
				\Omega \coloneqq (\Xi \times \mc{A}) / \Xi = \{[\gt,\a] \subset \Xi \times \mc{A} \mid [\gt \gtt ,\a] = [\gt,\gtt \a]\}
			\end{equation}
			A triple
			\begin{equation}
				\chi = (\eta,\kappa,\omega)
			\end{equation}
			consisting of a weighted graph $(\eta,\kappa)$, with a positive weight function $\kappa$ supported on the edges of $\eta$, and a gauge invariant voltage $\omega\in\Omega$, is called a \textit{gauge invariant voltage graph}.
		\end{definition}

		\begin{remark}[graph voltages in related work]\label{rem:graph-voltage-related}
			Graph voltages naturally appeared in different areas of mathematics and physics. While our exposition is inspired from algebraic graph theory \cite{grossTopologicalGraphTheory1987}, we list here few other occurrences.

			One perspective on graph voltages is given by \textit{cellular sheaves} \cite{hansenSpectralTheoryCellular2019a,hansenOpinionDynamicsDiscourse2021,bodnarNeuralSheafDiffusion2023}, where one encodes sheaf structures over graphs in terms edge matrices, similar to voltages. This approach favors \textit{flat voltages} as specified in Definition \ref{def:ym-flat}.

			Graph voltages also can be obtained from \textit{discrete gauge theory} \cite{baezSpinNetworkStates1996,Gattringer:2010zz,jiangGaugeTheoryGraphs2023}, where similar edge matrices on graphs arise from \textit{gauge fields on a manifold} after discretization \cite{leokDiscreteTheoryConnections2005}. This connection is further explained in Appendix \ref{sec:app}.

			Furthermore, voltages naturally arise as a modelling scheme for \textit{synchronization problems} \cite{gaoGeometrySynchronizationProblems2019}, which are encountered for instance in cryo electron microscopy \cite{singerViewingAngleClassification2011,shkolnisky2012viewing}.
		\end{remark}

		In the sequel, we will assume that we are given a gauge invariant graph voltage $\omega$, as a basis for defining gauge invariant heat kernels.

	\subsubsection{Gauged Heat Kernel}
		We are now in the position to define gauge invariant heat kernels $K_{\vg}$. Consider $\Gamma$, c.f.~\eqref{eq:def-Gamma}, as a inner product space with product
		\begin{equation}\label{eq:def-bracket-z}
			\lla \phi,\phi' \rra_{Z} = \sum_{i \in \nu} Z_i \la \x_i, x'_i \ra
			\quad \text{where} \quad
			\phi = [\gt,\x ],\; \phi' = [\gt,x'].
		\end{equation}
		This definition is well-posed and does not depend on the choice of $\gt$. Here, $Z_i$ denotes here the weight degree of vertex $i$
		\begin{equation}\label{eq:def-zi}
			Z_{i} = \sum_{j \in \nu} \kappa_{i j},
		\end{equation}
		used below for normalization.

		We next define some basic linear operators.

		\begin{definition}[gauged heat kernel and Laplacian \yyy]\label{def:ginv-k}
			Let $\chi = (\eta,\kappa,\omega)$ denote a gauge invariant voltage graph. We define the \textit{gauged Laplacian}, denoted by  $\hgen_{\vg}$, as the linear operator
			\begin{equation}\label{eq:def-gauged-Laplacian}
				\hgen_{\vg} : \Gamma \to \Gamma; \quad \phi = [\gt,\x ] \mapsto [\gt,\Lpl_{\eta,\kappa,\alpha}x], \qquad \omega = [\gt,\a],
			\end{equation}
			with
			\begin{equation}\label{eq:def-lalpha}
				(\Lpl_{\eta,\kappa,\alpha}x)_i = -\sum_{j \in \nu} \frac{\kappa_{ij}}{Z_i}(\a_{ij} \x_j - \x_i) \in \R^d, \qquad i \in \nu.
			\end{equation}
			Define the \textit{gauged heat kernel} associated to $\chi $ as
			\begin{equation}\label{eq:def-gauged-hk}
				K_{\vg}^t = e^{- t \hgen_{\vg}} : \Gamma \to \Gamma, \qquad \chi =  (\eta,\kappa,\omega).
			\end{equation}
		\end{definition}
		The notation $\phi = [\xi,x]$, $\omega = [\xi,\alpha]$ used in the definition means that both, the section $\phi$ and the voltage $\omega$ must be represented in terms of a \textit{matching} gauge $\xi$, in order for equation \eqref{eq:def-gauged-Laplacian} to hold.
		
		Despite the summation over all graph nodes $\nu$, $\hgen_{\vg}$ is a \textit{local} operator in the sense that, for a given vertex $i$, the sum only visits the first-order neighbors $j\sim i$, as $\kappa$ is only supported on adjacent pairs of vertices.

		\begin{remark}[relation to connection and sheaf Laplacians]\label{rem:laplacian-relations}
			The gauge invariant Laplacian, or rather its representation $\Delta_{\eta,\kappa,\alpha}$, are known concepts in the context of gauge theory and data science.

			From the point of view of gauge theory, this Laplacian is derived from the Lagrangian of the scalar field after the minimal coupling procedure, as discretized differential operator which governs the kinetic part of the field \cite{hamiltonMathematicalGaugeTheory2017}. This connection is explained in \cite[Section 4]{tong2018gauge} and \cite{rotheLatticeGaugeTheories2012}.

			Another way to formulate this relationship is to identify the gauge theoretic Laplacian as a \textit{generalized/twisted Laplacian} on a vector bundle \cite{Berline:1996aa} and treat the operator $L_{\chi}$ as its discretization. This perspective is adopted in \cite{gaoGeometrySynchronizationProblems2019}.

			Furthermore, this Laplacian was introduced as the \textit{discrete connection Laplacian} in the context of \textit{vector diffusion maps} \cite{Singer:2012aa}.

			Yet another perspective is gained when looking at the problem through the lens of \textit{cellular sheaf theory}, where similar operators are constructed to capture $H^{0}$-spaces of sheaf cohomology \cite{hansenSpectralTheoryCellular2019a,bodnarNeuralSheafDiffusion2023}.
		\end{remark}

		\begin{proposition}
			The operators $\hgen_{\vg}$ and $K_{\vg}$ are well-defined.
		\end{proposition}

		\begin{proof}
			We need to check that the definition of $\hgen_{\vg}$ does not depend on the choice of the representative $[\xi,\x ]$. To this end, we calculate $\hgen_{\vg}[\gt \gtt ^{-1}, \gtt \x ]$ and verify that it is equal to $\hgen_{\vg}[\gt,\x ]$.
			In order to evaluate $\hgen_{\vg}[\gt \gtt ^{-1},\gtt \x ]$, we transform the gauge for $\omega$ via
			\begin{equation}
				[\gt,\a] = [\gt \gtt^{-1},\gtt \a] = \omega.
			\end{equation}
			Thus, we have
			\begin{subequations}
			\begin{align}
			\hgen_{\vg}[\gt \gtt^{-1},\gtt \x ]
			\overset{\eqref{eq:def-gauged-Laplacian}}&{=} [\gt \gtt^{-1}, \Lpl_{\eta,\kappa,\gtt \a}(\gtt \x)]
			\overset{\substack{\eqref{def-action-Xi-A}\\ \eqref{eq:def-lalpha}}}{=}
			[\gt \gtt^{-1},\gtt \Lpl_{\eta,\kappa,\alpha}x] = [\gt,\Lpl_{\eta,\kappa,\alpha}x]
			\\
			\overset{\eqref{eq:def-Gamma}}&{=}
			\hgen_{\vg}[\gt,\x ].
			\end{align}
			\end{subequations}
		\end{proof}

		The following lemma prepares the study of properties of the gauged Laplacian and heat kernel from Definition \ref{def:ginv-k}. Similar versions appear in \cite[Proposition 2.4]{gaoGeometrySynchronizationProblems2019}. We give a full proof to make this paper self-contained.
		\begin{lemma}[discrete integration by parts]\label{lem:dip}
			We have the following relation
			\begin{equation}\label{eq:discrete-int-by-parts}
				\lla \phi', \hgen_{\vg} \phi \rra_{Z} = \frac{1}{2} \sum_{i,j \in \nu} \kappa_{ij} \la \a_{ij}y_{j} - y_{i} ,\a_{ij}\x_j - \x_i \ra,
			\end{equation}
			for all $\phi = [\gt,\x ], \phi' = [\gt,y] \in \Gamma$ and $\omega = [\gt,\a] \in \Omega$.
		\end{lemma}

		\begin{proof}
		First write
		\begin{subequations}\label{eq:symmetry-lhs}
			\begin{align}
				\lla \phi', \hgen_{\vg} \phi \rra
				&=
				\sum_{i \in \nu} Z_i \la y_i, (\Delta_{\eta,\kappa,\alpha}x)_{i}\ra \\
				\overset{\eqref{eq:def-lalpha}}&{=}
				-\sum_{i,j \in \nu} \kappa_{ij} \la y_i , \a_{ij} \x_j - \x_i \ra \\
				&=
				-\sum_{i,j \in \nu}  \kappa_{ij} \Big(\la y_i, \a_{ij} \x_j \ra - \la y_i, \x_i \ra \Big) 
				\label{eq:symmetry-lhs-final}.
			\end{align}
		\end{subequations}
		The relation $\sum_{j \in \nu} \kappa_{ij} = Z_i$ from the definition \eqref{eq:def-zi} then implies
		\begin{equation}
			\sum_{i,j \in \nu} \kappa_{ij} \la y_i,x_i \ra =
			\sum_{i \in \nu} Z_i \la y_i , x_i \ra.
		\end{equation}
		
		Furthermore, the relations $\a_{ij}^{-1} = \a_{ij}^{\top} = \a_{ji}$ and $\kappa_{ij} = \kappa_{ji}$ imply 
		\begin{equation}
			  \sum_{i,j \in \nu} \kappa_{ij} \la y_i, \alpha_{ij} x_j \ra
			\overset{i\leftrightarrow j}{=} \sum_{i,j \in \nu} \kappa_{ji} \la y_{j},\alpha_{ji} x_i\ra
			= \sum_{i,j \in \nu} \kappa_{ij} \la \alpha_{ij}y_{j},x_{i}\ra
			= \sum_{i,j \in \nu} \kappa_{ij} \la x_i , \alpha_{ij} y_{j} \ra,
		\end{equation}
		such that we have
		\begin{equation}
			\frac{1}{2} \sum_{i,j\in \nu} \kappa_{ij}
			(\la y_i, \a_{ij} \x_j\ra + \la \x_i, \a_{ij} y_j\ra )
			=
			\sum_{i, j \in \nu} \kappa_{ij} \la y_i, \a_{ij} \x_j\ra.
		\end{equation}
		This last simplification can also be derived from the fact that for any matrix $O \in \R^{N \times N}, N = | \nu |$, one has
		\begin{equation}\label{eq:symmetrization}
			\sum_{i,j \leq N} O_{ij} = \frac{1}{2} \sum_{i,j \leq N} (O_{ij} + O_{ji}),
		\end{equation}
		thus only the symmetric part of $O$ contributes to the sum of all entries.

		These simplifications allow us to rewrite \eqref{eq:symmetry-lhs-final} as
		\begin{equation}
				\lla \phi', \hgen_{\vg} \phi \rra
				=
				- \frac{1}{2} \sum_{i,j \in \nu} \kappa_{ij}
				(\la y_i, \a_{ij} \x_j\ra + \la \x_i, \a_{ij}y_{j} \ra)
				+ \sum_{i \in \nu} Z_i \la y_i,x_i \ra.
		\end{equation}

		Putting together, we obtain for the right-hand side of \eqref{eq:discrete-int-by-parts}
		\begin{subequations}\label{eq:symmetry-rhs}
			\begin{align}
				&\quad \frac{1}{2} \sum_{i,j \in \nu} \kappa_{ij} \la \a_{ij} y_{j} - y_{i}, \a_{ij} \x_j - \x_i \ra \\
				&=
				- \frac{1}{2}\sum_{i, j \in \nu} \kappa_{ij} (\la \x_i,\a_{ij}y_{j} \ra + \la y_i, \a_{ij} \x_{j}\ra) \\ &\qquad
				+ \frac{1}{2}\sum_{i,j\in \nu}  \kappa_{ij} (\la \a_{ij}x_j, \a_{ij} y_{j} \ra + \la y_{i},x_i \ra)  \\
				\overset{\alpha_{ij}^{\top}=\alpha_{ij}^{-1}}&{=}
				- \frac{1}{2}\sum_{i,j \in \nu}  \kappa_{ij} (\la \x_i,\a_{ij} y_{j} \ra + \la y_i, \a_{ij} \x_{j}\ra) \\ &\qquad
				 + \frac{1}{2}\sum_{i,j \in \nu}  \kappa_{ij} (\la y_{j}, \x_j \ra + \la y_{i},x_i \ra)  \\
				&=
				- \frac{1}{2}\sum_{i,j \in \nu}  \kappa_{ij} (\la \x_i,\a_{ij}y_{j} \ra + \la y_i, \a_{ij} \x_{j}\ra)
				+ \sum_{i \in \nu} Z_i \la y_{i},x_i \ra,
			\end{align}
		\end{subequations}
		where the last step follows again from \eqref{eq:symmetrization}. Thus we obtained \eqref{eq:symmetry-lhs-final} which proves Equation \eqref{eq:discrete-int-by-parts}.
		\end{proof}

		\begin{proposition}[$K_{\vg}$ is a heat kernel]
			The operator $K_{\vg}$ is a heat kernel, i.e. $\hgen_{\vg}$ is self-adjoint and positive semi-definite.
		\end{proposition}

		\begin{proof}
			Lemma \ref{lem:dip} shows that
			\begin{equation}
				\lla \phi , \hgen_{\vg} \phi' \rra_{Z}
				=
				\sum_{i,j \in \nu} \kappa_{ij} \la \a_{ij} y_{j} - y_{i},\a_{ij}x_j - \x_i \ra
				=
				\lla \hgen_{\vg} \phi, \phi' \rra_{Z}
			\end{equation}
			for all $\phi= [\gt,\x]$ and $\phi'=[\gt,y]$ in $\Gamma$ and
			\begin{equation}
				\lla \phi , \hgen_{\vg} \phi \rra_{Z} = \sum_{i,j \in \nu} \kappa_{ij} \|\a_{ij}x_j - \x_i \|^2 \geq 0.
			\end{equation}
		\end{proof}

		We provide bounds on the spectrum of the heat kernel.
		\begin{proposition}[spectrum of the GHK]
			The spectra of $\hgen_{\vg}$ and $K_{\vg}^t$ lie in $[0,2]$ and $[e^{-2t},1]$ respectively.
		\end{proposition}

		\begin{proof}
			We use an argument of Gershgorin type. Assume $\lambda, \phi=[\xi,\x]$ is an eigenpair of $\hgen_{\vg}$. Then we can write
			\begin{equation}
				\hgen_{\vg}\phi = [\gt,\lambda \x ] = [\gt,\Lpl_{\eta,\kappa,\alpha}x], \quad \omega = [\gt,\a]
			\end{equation}
			which implies by \eqref{eq:def-lalpha}
			\begin{equation}
				\lambda \x_i = -\sum_{j \in \nu} \frac{\kappa_{ij}}{Z_i} (\a_{ij} \x_j - \x_{i}), \quad \forall i \in \nu.
			\end{equation}
			Taking the scalar product with $\x_i$ on both side and rearranging yields
			\begin{equation}
				(\lambda - 1) \|x_i\|^2 = -\sum_{j \in \nu} \frac{\kappa_{ij}}{Z_i} \la \x_i, \a_{ij} \x_j \ra, \quad \forall i \in \nu.
			\end{equation}
			We pick $i$ such that $\|x_i\|^2$ is maximal and take absolute values.
			\begin{equation}
			\begin{aligned}
				|\lambda - 1|
				&=
				\bigg| \sum_{j \in \nu} \frac{\kappa_{ij}}{Z_i \|x_i\|^2} \la \x_i, \a_{ij} \x_j \ra \bigg| \\
				&\leq
				\sum_{j \in \nu} \frac{\kappa_{ij}}{Z_i} \left| \frac{\la \x_i, \a_{ij} \x_j \ra}{\|x_i\|^2} \right|
				\\
				&\leq
				\sum_{j \in \nu} \frac{\kappa_{ij}}{Z_i} = 1
			\end{aligned}
			\end{equation}
			due to our choice of the index $i$ and $\alpha_{ij} \in \SO(d)$.			Thus $| \lambda - 1 | \leq 1$. Furthermore, we know that $\lambda$ is real, since $\hgen_{\vg}$ is self-adjoint. Hence $\lambda \in [0,2]$ and $e^{-\lambda t} \in [e^{-2t},1]$, which proves the claim.
		\end{proof}

\subsection{Axiomatic Description of the Gauged Laplacian}
	In this subsection, we give an additional description of the gauged Laplacian $\hgen_{\vg}$ which distinguishes it among all other gauge invariant operators. We first introduce further notation required to state our result.

	Let $\mc{F} : \Gamma \to \Gamma$ denote a gauge invariant data transformation. For any gauge $\xi \in \Xi$, we write $\mc{F}([\xi,x]) = [\xi,y(x)]$ for some $y(x) \in X$ which depends on both $\xi$ and $x$. We denote by $y(x)$ the value of the function $F: \Xi \times X \to X$ for the argument $(\xi,x)$ from Proposition \ref{prop:g-inv-trafos},
	\begin{equation}\label{eq:y-by-F}
		y(x) = F(\xi,x).
	\end{equation}

	We say that $\mc{F}$ is \textit{locally Laplacian} if, for every $i \in \nu$, there is gauge $\xi^{i} \in \Xi$ such that, for every $\phi \in \Gamma$, one has
	\begin{equation}\label{eq:def-yi}
		\mc{F}([\xi^{i},x]) = [\xi^{i},y^{i}(x)], \quad y^{i}(x) \coloneq F(\xi^{i},x), \quad \phi = [\xi^{i},x]
	\end{equation}
	and
	\begin{equation}\label{eq:def-y-diag}
		y^{i}(x)_{i} = -\sum_{j \in \nu} \frac{\kappa_{ij}}{Z_{i}}(x_{j} - x_{i}) \coloneq (\Delta_{\eta,\kappa} x)_{i} \in \R^{d}.
	\end{equation}
	The notation uses the linear operator $\Delta_{\eta,\kappa}: \R^{N \times d} \to \R^{N \times d}$, which can be identified as the \textit{random walk normalized graph Laplacian} \cite{hein2007graph} associated to the weighted graph $(\eta,\kappa)$.

	We set
	\begin{equation}\label{eq:ass-g}
		\xi^{ij} \coloneq (\xi^{i})^{-1} \xi^{j} \in \Xi
	\end{equation}
	for any pair $i,j \in \nu$, not necessarily adjacent in the graph $\eta$, with the property $\xi^{ij} = (\xi^{ji})^{-1}$. We call $\xi^{ij}$ the \textit{associated gauges} for a given locally Laplacian transformation $\mc{F}$, which satisfy the relation
	\begin{equation}
		\xi^{ji}\xi^{ik} = \xi^{jk}, \qquad \forall i,j,k \in \nu.
	\end{equation}

	These notions allow us to characterize the gauged Laplacians in terms of locally Laplacian transformations that satisfy a certain \textit{integrability condition} imposed on the associated gauges.

	\begin{proposition}[consistent locally Laplacian operators are gauged Laplacians]\label{prop:loc-lpl}
		Let $\mc{F} : \Gamma \to \Gamma$ be a locally Laplacian gauge invariant data transformation, such that the associated gauges $\xi^{ij}$ satisfy the relations
		\begin{equation}\label{eq:intergrability}
			\xi^{ij}_{i} = \xi^{ij}_{j}, \qquad \forall i,j \in \nu.
		\end{equation}
		Then there is a gauge invariant graph voltage $\omega \in \Omega$ such that
		\begin{equation}\label{eq:mcF=Lchi}
			\mc{F} = L_{\chi}, \qquad \chi = (\eta,\kappa,\omega),
		\end{equation}
		and for all $i,j,k \in \nu, j > k$, we have
		\begin{equation}\label{eq:def-alphai}
			\omega = [\xi^{i},\alpha^{i}], \qquad
			\alpha^{i}_{jk} = \xi^{ij}_{j} \xi^{ji}_{k} \in \SO(d).
		\end{equation}
	\end{proposition}

	\begin{proof}
		We show that $\omega$ defines a gauge invariant voltage, which amounts to verifying that the following consistency conditions hold
		\begin{equation}
			\alpha_{jk}^{i} = (\alpha_{kj}^{i})^{-1}, \quad \text{and} \quad
			\alpha_{jk}^{m} = \xi^{mi}_{j} \alpha^{i}_{jk} \xi^{im}_{k}
		\end{equation}
		for all $i,j,k,m \in \nu$. The second equation one holds since
		\begin{equation}
			\xi^{mi}_{j} \alpha^{i}_{jk} \xi^{im}_{k}
			=
			\underbrace{\xi^{mi}_{j} \xi^{ij}_{j}}_{\overset{\eqref{eq:ass-g}}{=}\xi^{mj}_{j}}
			\underbrace{\xi^{ji}_{k} \xi^{im}_{k}}_{\overset{\eqref{eq:ass-g}}{=}\xi^{jm}_{k}}
			= \xi^{mj}_{j} \xi^{jm}_{k}
			\overset{\eqref{eq:def-alphai}}{=} \alpha^{m}_{jk}.
		\end{equation}
		For the first condition, we compute
		\begin{equation}
			\alpha^{i}_{jk}\alpha^{i}_{kj}
			=
			\xi^{ij}_{j}
			\underbrace{
				\xi^{ji}_{k}\xi^{ik}_{k}
			}_{\overset{\eqref{eq:ass-g}}{=} \xi^{jk}_{k}}
			\xi^{ki}_{j}
			=
			\xi^{ij}_{j}
			\underbrace{
				\xi^{jk}_{k}
			}_{\overset{\eqref{eq:intergrability}}{=} \xi^{jk}_{j}}
			\xi^{ki}_{j}
			=
			\xi^{ij}_{j}\xi^{jk}_{j}\xi^{ki}_{j}
			\overset{\eqref{eq:ass-g}}{=}
			\II.
		\end{equation}
		This shows that the definition \eqref{eq:def-alphai} is consistent and defines a graph voltage.

		We proceed in order to prove \eqref{eq:mcF=Lchi}.
		Since both transformations are gauge invariant, we can check equality in any gauge, which we choose here to be $\xi^{i}$ for a fixed choice of $i \in \nu$. Evaluating both operators on $\phi = [\xi^{i},x]$ yields
		\begin{equation}
			\mc{F}([\xi^{i},x]) = [\xi^{i},y^{i}(x)], \qquad
			L_{\chi}([\xi^{i},x]) = [\xi^{i},\Delta_{\eta,\kappa,\alpha^{i}}x],
		\end{equation}
	 	where $y^{i}(x)$ is defined by \eqref{eq:def-yi}.
		We thus have to verify that
		\begin{equation}\label{eq:proof-show-Delta-y-i}\Delta_{\eta,\kappa,\alpha^{i}}x = y^{i}(x).
		\end{equation}
		We separately check equality at all nodes $j \in \nu$ and first calculate $y^{i}(x)_j$. Due to gauge invariance, the equation
		\begin{equation}
			y^{i}(x) = F(\xi^{i},x)
			\overset{
				\substack{
					\eqref{eq:consis} \\ \eqref{eq:ass-g}
				}
			}{=}
			\xi^{ij} F(\xi^{j},\xi^{ji}x)
			\overset{
				\eqref{eq:def-yi}
			}{=}
			\xi^{ij} y^{j}(\xi^{ji}x).
		\end{equation}
		holds. In particular, we have
		\begin{equation}
			y^{i}(x)_{j} = \xi^{ij}_{j} y^{j}(\xi^{ji}x)_{j}, \quad \forall j \in \nu.
		\end{equation}
		On the other hand, relation \eqref{eq:def-y-diag} yields
		\begin{equation}
			y^{j}(\xi^{ji} x)_{j} = (\Delta_{\eta,\kappa} \xi^{ji}x)_{j}.
		\end{equation}
		By combining both equations, we obtain
		\begin{equation}\label{eq:proof-F-L-aux-1}
			y^{i}(x)_{j}
			= \xi^{ij}_{j}y^{j}(\xi^{ji}x)_j
			= \xi^{ij}_{j}(\Delta_{\eta,\kappa} \xi^{ji}x )_{j}
		\end{equation}
		which expands to
		\begin{equation}
			\xi^{ij}_{j}(\Delta_{\eta,\kappa} \xi^{ji}x )_{j}
			\overset{\eqref{eq:def-y-diag}}{=}
			- \xi^{ij}_{j} \sum_{k\in \nu} \frac{\kappa_{jk}}{Z_{j}}(\xi^{ji}_{k} x_{k} - \xi^{ji}_{j} x_{j})
			=
			- \sum_{k \in \nu} \frac{\kappa_{jk}}{Z_{j}}(\xi^{ij}_{j}\xi^{ji}_{k}x_{k} - x_{j})
		\end{equation}
		using $\xi^{ij}_{j}\xi^{ji}_{j} = \II$ in the last step. Applying the definition \eqref{eq:def-alphai} to the last expression gives
		\begin{equation}\label{eq:proof-F-L-aux-2}
			\xi^{ij}_{j}(\Delta_{\eta,\kappa} \xi^{ji}x)_{j}
			= -\sum_{k \in \nu} \frac{\kappa_{jk}}{Z_{j}}(\alpha_{jk}^{i} x_{k} - x_j)
			\overset{\eqref{eq:def-lalpha}}{=} (\Delta_{\eta,\kappa,\alpha^{i}}x)_{j}.
		\end{equation}
		Hence \eqref{eq:proof-F-L-aux-1}--\eqref{eq:proof-F-L-aux-2} show that
		\begin{equation}
			y^{i}(x)_{j} = (\Delta_{\eta,\kappa,\alpha^{i}}x)_{j}, \quad \forall j \in \nu.
		\end{equation}
		We conclude that Equation \eqref{eq:proof-show-Delta-y-i} holds and $\mc{F} = L_{\chi}$ with $\chi=(\eta,\kappa,\omega)$.

		We finish the proof by demonstrating that the equation $\mc{F} = L_{\chi}$ is compatible with $\mc{F}$ being locally Laplacian, i.e. the assumptions of the Proposition are not violated. By virtue of the definitions \eqref{eq:ass-g} and \eqref{eq:def-alphai}, we find $\alpha^{i}_{ij} =\II$ for all edges $i>j$, such that
			\begin{equation}
				y^{i}(x)_{i}
				= - \sum_{j\in\nu}\frac{\kappa_{ij}}{Z_{i}}(\alpha^{i}_{ij} x_{j} - x_{i})
				= - \sum_{j \in \nu} \frac{\kappa_{ij}}{Z_{j}}(x_{j} - x_{i})
				= (\Delta_{\eta,\kappa} x)_{i}, \quad \forall i \in \nu,
			\end{equation}
		which is Equation \eqref{eq:def-y-diag}.
	\end{proof}

	\begin{remark}[associated gauges and synchronizability]
		Suppose that $\mc{F}$ is a locally Laplacian operator and the associated gauges $\xi^{ij}$ are all equal to the unit element, that is $\xi^{ij}_{k} = \II$ for all $i,j,k \in \nu$. Then condition \eqref{eq:intergrability} is directly satisfied and hence $\mc{F}$ is equal to some operator $L_{\chi}$.
		In particular, then the voltage $\omega$ is \textit{synchronizable}. This notion is defined and discussed in detail in Section \ref{sec:transform}.
	\end{remark}

	Another way of characterizing the gauged Laplacian is to analyze its \textit{lifting properties}. For a given graph $\eta$ with edge set $\veps$, we can construct a family of related graphs, called the \textit{tree blow-ups of $\eta$}. We now iteratively construct such a tree blow-up of $\eta$, which we denote by $\tilde \eta$, with vertex set $\tilde \nu$ and edge set $\tilde \veps$.

	Start with a spanning tree $\tilde\eta \subset \eta$. For every vertex $i\in\nu$, we denote with $\tilde{i}$ the corresponding vertex of $\tilde\nu$. These two vertex sets are in one-to-one correspondence at this stage, since the spanning tree contains all the vertices of $\eta$. The edge set $\tilde \veps$ is a subset of $\veps$. 

	Define an orientation for the edges of $\tilde\eta$ by setting $\tilde{i} > \tilde{j}$ if and only if $i>j$. Then, for every edge $i>j$ in $\veps \setminus \tilde{\veps}$, add a new vertex $I$ to $\tilde\nu$ and a new edge to $\tilde{\veps}$ from $\tilde{j}$ to the newly created vertex $I$ with orientation $I>\tilde{j}$. The resulting graph $\tilde\eta$ is a tree blow-up of $\eta$.

	Any tree blow-up $\tilde\eta$ has the following properties.
	\begin{enumerate}
		\item $\tilde\eta$ is a tree.

		\item There is a bijection between the edge sets $\tilde\veps$ and $\veps$.

 		\item There is a surjective mapping $\pi:\tilde\nu\to\nu$, with $\tilde{i} \in \pi^{-1}(i)$ for all $i\in\nu$.

		\item $\tilde{\eta}$ inherits a symmetric edge weight function $\tilde{\kappa}$ by setting $\tilde{\kappa}_{IJ}\coloneq\kappa_{\pi(I)\pi(J)}$ for all $I,J \in \tilde{\nu}$.

		\item  We denote by $\tilde{\Gamma} = (\tilde{\Xi} \times \tilde{X}) / \tilde{\Xi}$ the space of gauge invariant node features over the graph $\tilde{\eta}$, where $\tilde{X} = \R^{\tilde{N} \times d}$ denotes the $d$ node features over $\tilde{\eta},\tilde{N} = |\tilde{\nu}|$ and $\tilde{\Xi} = (\SO(d))^{\tilde{N}}$. There are mappings  $\sigma_1 : \Gamma \to \tilde{\Gamma}$ and $\sigma_2 : \tilde{\Gamma} \to \Gamma$, defined by
		\begin{equation}\label{eq:def-sigma-1}
			\sigma_1(\phi)_{I}
			= \phi_{\pi(I)}, \qquad \forall I \in \tilde{\nu}
		\end{equation}
		and
		\begin{equation}\label{eq:def-sigma-2}
			\sigma_2(\tilde{\phi})_i
			= \sum_{\substack{I \in \tilde{\nu} \\  \pi(I) = i}} \tilde{\phi}_I, \qquad \forall i \in \nu.
		\end{equation}
	\end{enumerate}

	\begin{proposition}[tree blow-up lifting]\label{prop:blow-up-tree}
		Let $\mc{F}:\Gamma \to \Gamma$ denote a gauge invariant transformation over the graph $\eta$, with tree blow-up $\tilde{\eta}$. Then
		\begin{equation}
			\mc{F} = L_{\chi},
		\end{equation}
		for some gauge invariant voltage graph $\chi = (\eta,\kappa,\omega)$ if and only if there is a gauge $\tilde\psi \in \tilde \Xi$ such that the diagram
		\begin{equation}\label{eq:tree-lift}
			\begin{tikzcd}
				&\tilde X \ar[r,"\tilde{\Delta}_{\tilde\eta,\tilde\kappa}"] &\tilde X \ar[dr,"\Sigma_{\tilde{\xi}}^{-1}"] &\\ 
 				\tilde{\Gamma}\ar[ru,"\Sigma_{\tilde \xi}"] & & & \tilde{\Gamma} \ar[dl,"\sigma_{2}"] \\
				&\Gamma \ar[ul,"\sigma_{1}"] \ar[r,"\mc{F}"] & \Gamma&
			\end{tikzcd},
		\end{equation}
		commutes, where $\tilde{\Delta}_{\tilde\eta,\tilde\kappa} : \tilde X \to \tilde X$ is a weighted graph Laplacian on $\tilde \eta$ defined by 
		\begin{equation}\label{eq:def-delta-td}
			(\tilde{\Delta}_{\tilde\eta,\tilde\kappa}\tilde\x)_{I}
			= -\sum_{J\in \tilde\nu} \frac{\tilde\kappa_{IJ}}{Z_{\pi(I)}} (\tilde\x_{J} - \tilde\x_{I}) \qquad \forall \tilde x \in \tilde X, I \in \tilde \nu,
		\end{equation}
		and $\Sigma_{\xi}$ was defined in in the proof of Proposition \ref{prop:g-inv-trafos}.
	\end{proposition}

	\begin{proof}

		We first assume that $\mc{F} = L_{\chi}$ holds and show that the diagram \eqref{eq:tree-lift} commutes for some $\tilde \xi$. We pick a gauge $\xi \in \Xi$ and write $\omega = [\xi,\a]$. Associated to $\xi$, we construct another gauge $\tilde\zeta\in\tilde\Xi$ by setting
		\begin{equation}\label{eq:proof-zeta-sigma-xi}
			\tilde\zeta_{I} \coloneqq \xi_{\pi(I)}.
		\end{equation}

		Next, we observe that there is a bijection between the set of voltages $\mc A$ on $\eta$ and voltages $\tilde{\mc A}$ on $\tilde \eta$, since the edge sets of the two graphs are correspond one-to-one to each other. For every voltage $\alpha \in \mc{A}$ there is an associated voltage $\tilde \alpha \in \tilde{\mc A}$ defined by
		\begin{equation}
			\tilde \alpha_{IJ} = \alpha_{ij},
		\end{equation} 
		where $I>J \in \tilde \veps$ is the edge uniquely corresponding to the edge $i>j \in \veps$. We infer that there is a gauge $\tilde{\psi} \in \tilde\Xi$ such that
		\begin{equation}\label{eq:proof-sync}
			\alpha_{ij} = \tilde \alpha_{IJ} = \tilde \psi_{I}^{-1} \tilde \psi_{J}, \qquad \forall I>J \in \tilde \veps,
		\end{equation}
		which follows from the fact that $\tilde\eta$ is a tree and Corollary \ref{cor:tree-gauge}.
		Then there is a gauge $\tilde\xi \in \tilde\Xi$ that satisfies
		\begin{equation}\label{eq:proof-blow-up-gauge-transform}
			\tilde\zeta\tilde\psi^{-1}=\tilde\xi.
		\end{equation}

		We now show that $\tilde \xi \in \tilde \Xi$ yields the desired commutative diagram. The relation to be shown reads as follows
		\begin{equation}
			\mc{F} \overset{!}{=} \sigma_{2} \circ \Sigma_{\tilde\xi}^{-1} \circ \tilde{\Delta}_{\tilde\eta,\tilde\kappa} \circ \Sigma_{\tilde \xi} \circ \sigma_{1}.
		\end{equation} 
		We further break down this relation for a section $\phi = [\xi,x]$, by evaluating the right hand side
		\begin{subequations}\label{eq:proof-eq-comm-diagram}
		\begin{align}
			[\xi,x]
			&\xrightarrow{\sigma_{1}} [\tilde\zeta, \tilde x] = [\tilde\zeta\tilde\psi^{-1},\tilde\psi\tilde\x] = [\tilde\xi, \tilde\psi\tilde x] \\
			&\xrightarrow{\Sigma_{\tilde\xi}} \tilde \psi \tilde x 			\xrightarrow{\tilde{\Delta}_{\tilde\eta,\tilde\kappa}} \tilde\Delta_{\tilde\eta,\tilde\kappa}\tilde \psi \tilde x \\
			&\xrightarrow{\Sigma_{\tilde\xi}^{-1}} [\tilde \xi,\tilde{\Delta}_{\tilde\eta,\tilde\kappa}\tilde \psi \tilde x]
			=
			[\tilde{\zeta}, \tilde{\psi}^{-1} \tilde{\Delta}_{\tilde\eta,\tilde\kappa}\tilde \psi \tilde x]
			\\
			&\xrightarrow{\sigma_{2}}
			\sigma_2([\tilde{\zeta}, \tilde{\psi}^{-1} \tilde{\Delta}_{\tilde\eta,\tilde\kappa}\tilde \psi \tilde x]),
		\end{align}
		\end{subequations}
		where $\tilde{x}_I = x_{\pi(I)},\forall I \in \tilde \nu$. If we write $\mc{F}([\xi,x]) = [\xi,y(x)]$ and use the definition \eqref{eq:def-sigma-2} of $\sigma_2$, the commutativity of the diagram \eqref{eq:tree-lift} is equivalent to the relation
		\begin{subequations}\label{eq:converse}
		\begin{align}
			y(x)_i
			\overset{!}&{=} \sum_{\substack{I\in \tilde\nu \\ \pi(I) = i}} \tilde{\psi}^{-1}_{I} (\tilde{\Delta}_{\tilde\eta,\tilde\kappa} \tilde \psi \tilde x)_{I}, \qquad \forall i \in \nu \\
			\overset{\eqref{eq:def-delta-td}}&{=} -\sum_{\substack{I\in \tilde\nu \\ \pi(I) = i}} \tilde{\psi}_I^{-1}\sum_{J\in \tilde\nu} \frac{\tilde\kappa_{IJ}}{Z_{\pi(I)}} (\tilde{\psi}_{J} \tilde\x_{J} -\tilde{\psi}_I \tilde\x_{I})  \\
			&= -\sum_{\substack{I\in \tilde\nu \\ \pi(I) = i}} \sum_{J\in \tilde\nu} \frac{\tilde\kappa_{IJ}}{Z_{\pi(I)}} (\tilde{\psi}_I^{-1}\tilde{\psi}_{J} \tilde\x_{J} -\tilde\x_{I}),  
		\end{align}
		\end{subequations}
		which is valid due to the following argument.
		
		By assumption, $\mc{F} = L_\chi$ holds, from which $y(x) = \Delta_{\eta,\kappa,\alpha}x$ follows. It is therefore enough to show
		\begin{equation}\label{eq:chi-L-vs-tree-L}
			\sum_{j \in \nu} \frac{\kappa_{ij}}{Z_{i}}(\alpha_{ij}x_{j} - x_{i})
			\overset{!}{=}
			\sum_{\substack{I\in \tilde\nu \\ \pi(I) = i}} \sum_{J\in \tilde\nu} \frac{\tilde\kappa_{IJ}}{Z_{\pi(I)}} (\tilde{\psi}_I^{-1}\tilde{\psi}_{J} \tilde\x_{J} -\tilde\x_{I}),
		\end{equation}
		for all $i \in \nu$. 
		Taking into account $\tilde\x_{I} = x_{\pi(I)}$ and $\tilde\kappa_{IJ} = \kappa_{\pi(I)\pi(J)}$, the right hand side of \eqref{eq:chi-L-vs-tree-L} simplifies to
		\begin{equation}
			\sum_{\substack{I\in \tilde\nu \\ \pi(I) = i}} \sum_{J\in \tilde\nu} \frac{\tilde\kappa_{IJ}}{Z_{\pi(I)}} (\tilde{\psi}_I^{-1}\tilde{\psi}_{J} \tilde\x_{J} -\tilde\x_{I})
			=
			\sum_{\substack{I\in \tilde\nu \\ \pi(I) = i}} \sum_{J \in \tilde \nu}
			\frac{\kappa_{i\pi(J)}}{Z_{i}}
			(\tilde\psi_{I}^{-1}\tilde\psi_{J}\x_{\pi(J)} - \x_i), \quad \forall i \in \nu.
		\end{equation}
		Recall that for every edge $i>j\in \veps$, there is a unique edge $I>J\in \tilde\veps$ which also appears in this summation, and that for every such edge the relation $\alpha_{ij} = \tilde\psi_{I}^{-1}\tilde\psi_{J}$ holds. Furthermore, the sum above only visits vertices $J$ adjacent to $I$ due to the support of $\tilde\kappa$. This argument shows that \eqref{eq:chi-L-vs-tree-L} holds, after rearranging terms, which proves the claim.\\[2mm]

		As for the converse statement, suppose the gauge $\tilde \xi \in \tilde \Xi$ is given. The goal is to prove the relation $\mc{F} = L_\chi$ for some $\chi$. We start by defining a gauge $\xi \in \Xi$ via
		\begin{equation}
			\xi_i \coloneqq \tilde\xi_{\tilde i}.
		\end{equation}
		This entails the definition of two more gauges, $\tilde \zeta, \tilde \psi \in \tilde \Xi$ by
		\begin{equation}
			\tilde \zeta_{I} \coloneqq \xi_{\pi(i)}, \qquad
			\tilde\zeta\tilde\psi^{-1} = \tilde \xi.
		\end{equation}
		Next, evaluate $\mc{F}$ with respect to the gauge $\xi$ by writing $\mc{F}([\xi,x]) = [\xi,y(x)]$. The commutativity of \eqref{eq:tree-lift} then yields
		\begin{equation}
			y(x)_i = 
			\sum_{\substack{I\in \tilde\nu \\ \pi(I) = i}} \tilde{\psi}^{-1}_{I} (\tilde{\Delta}_{\tilde\eta,\tilde\kappa} \tilde \psi \tilde x)_{I}, \qquad \forall i \in \nu
		\end{equation}
		Using the above arguments for proving the first claim, one then has
		\begin{equation}
			y(x)_i = 
			\sum_{j \in \nu} \frac{\kappa_{ij}}{Z_i}(\alpha_{ij}x_j - x_i), \qquad \forall i \in \nu,
		\end{equation} 
		where
		\begin{equation}
			\alpha_{ij} \coloneqq \tilde\psi_{I}^{-1} \tilde \psi_{J},
		\end{equation}
		and $I>J \in \tilde \veps$ is the edge corresponding to $i>j \in \veps$. This shows that $y(x) = \Delta_{\eta,\kappa,\alpha} x$ and $\mc{F} = L_\chi$, where $\chi = (\eta,\kappa,[\xi,\alpha])$ and proves the second part of the equivalence.
	\end{proof}

	We summarize our findings.
	\begin{theorem}[axioms for the gauged Laplacian]\label{thm:axioms}
		The gauged Laplacian $L_{\chi} : \Gamma \to \Gamma$ is the unique gauge invariant transformation that
		\begin{enumerate}
			\item is locally Laplacian and satisfies the integrability condition from Proposition \ref{prop:loc-lpl};
			\item satisfies the tree blow-up lifting property from Proposition \ref{prop:blow-up-tree}.
		\end{enumerate}
		The gauged heat kernel $K_{\chi}$ is distinguished by the fact that its generator is the unique operator which has these properties.
	\end{theorem}

\subsection{GHK as a Machine Learning Device}

	We review in this subsection the role of the gauged heat kernel in machine learning.

	The connection between heat kernel methods and machine learning methods has been fruitfully applied in the context of \textit{graph neural networks} \cite{kipfSemiSupervisedClassificationGraph2017a,defferrardConvolutionalNeuralNetworks2017,bronstein2017geometric}, where heat kernels on graphs were recognized as basic elements of generalized wavelet filtering architectures \cite{chamberlain2021grand,chamberlain2021beltrami}.

	The main building block of generalized \textit{message passing} network architectures is the following. Given input data $x$ are repeatedly transformed by some form of aggregation transformation $\mc{K}$, concatenated with non-linear activation functions. The connection to heat kernel methods results from identifying heat kernels $K = e^{-t\hgen}$ as candidates for the aggregation operator $\mc{K}$. The further step towards machine learning entails a parametrization of a suitable subspace of potential heat kernels and learning a suitable heat kernel from given data by optimizing some loss function.

	In the case of the gauged heat kernel, we can parametrize all gauge invariant Laplace operators, and thereby also all heat kernels over a given graph $\eta$, in terms of weight functions $\kappa$ and voltages $\omega$.
	Thus, the connection between the gauged heat kernel and machine learning is to interpret $K_{\chi}$ as a parametrized data transformation architecture, whose parameters are given by graph weights $\kappa$ and voltages $\a$, and to infer optimal weight from data.
	This approach was taken in \cite{bodnarNeuralSheafDiffusion2023,bambergerBundleNeuralNetworks2024} and \cite{cassel2025bundle}, where such parameters were estimated in the context of node classification on graphs. The authors report that the gauged heat kernel $K_{\vg}$ has favorable properties for solving such tasks.

	A motivation to pursue connections between gauged heat kernels and machine learning is to perform learning over geometrically meaningful parameter domains, such as the space of voltage graphs $\chi$ in present case. In \cite{bodnarNeuralSheafDiffusion2023}, voltage graphs were related to \textit{sheaf structures}, whereas \cite{gaoGeometrySynchronizationProblems2019,bambergerBundleNeuralNetworks2024} related them to discrete vector and principal bundles. The gauged heat kernel also resembles discretized \textit{kernel field transforms} for coordinate independent convolutional neural networks \cite{weilerCoordinateIndependentConvolutional2021}. The connection is that the gauged heat kernel may be viewed as discretization of a kernel field transform, with a trivial kernel field.

	We provide additional geometric interpretation of voltage graphs in terms of \textit{lattice gauge theory} \cite{Gattringer:2010zz}, see the Remarks \ref{rem:general-groups}, \ref{rem:graph-voltage-related}, \ref{rem:laplacian-relations} and Appendix \ref{sec:app}. This relates the space of graph voltages to parallel transport operators and gauge fields on principal fiber bundles.

	Building on intuitions gained from the gauge theory perspective, we introduce below novel analysis methods for the gauged heat kernel and voltage graphs with regard to their applications in data science. Specifically, we focus in Sections \ref{sec:transform} and \ref{sec:ym} on the confluence of \textit{synchronization} properties of voltage graphs, heat kernel transformations and Yang Mills curvature analysis. Especially the Yang-Mills \textit{energies} are promising in the context of machine learning, providing a scalar quantity for controlling the transformation properties of corresponding heat kernels.

	Overall, this approach highlights symmetry properties inherent to the gauged heat kernel and its embedding into the coordinate/gauge invariant deep learning formalism.

\section{Transformation Properties of the GHK}\label{sec:transform}

\subsection{GHK as a Projection Operator}
		We assume that $V$ is finite dimensional, which is the case for node feature data over finite graphs, and take a closer look to the operators introduced by Definition \ref{def:ginv-k}.

	\begin{proposition}[heat kernel converges to projection operator \cite{hansenSpectralTheoryCellular2019a,hansenOpinionDynamicsDiscourse2021}]\label{prop:projection}
		Let $K^t$ by a heat kernel with generator $\hgen$. Then $K^t$ is self-adjoint and is positive semi-definite for all scales $t > 0$. If $\ker L$ denotes the nullspace of the operator $\hgen: V \to V$ with the corresponding orthogonal projection $\o{proj}$. Then one has
		\begin{equation}
			\| K^t - \o{proj} \| \leq e^{- \mu t},
		\end{equation}
		where $\mu$ is the smallest non-zero eigenvalue of $\hgen$. Thus, the heat kernel family converges to projection operator $\o{proj}$ in the operator norm on $V$.
	\end{proposition}

	\begin{proof}
		Since $\hgen$ is self-adjoint and positive semi-definite and $V$ is finite dimensional, we can diagonalize $\hgen$,
		\begin{equation}
			\hgen = \sum_{r = 1}^{\dim V} \lambda_r \varphi_r \varphi^\top_r, \qquad
			\hgen v = \sum_r \lambda_r \la \varphi_r,v \ra \varphi_r,
		\end{equation}
		where $\lambda_r \geq 0$ and $\varphi_r$ are the eigenvectors with duals $\varphi_r^{\top}$.
		Then we have
		\begin{equation}
			K^t = \sum_r e^{-\lambda_r t} \varphi_r \varphi^\top_r,
		\end{equation}
		and the properties follow directly.
	\end{proof}
	Proposition \ref{prop:projection} says that $K^t$ defines an asymptotic (in the sense of infinite time) homotopy between the identity $K^0$ and the projection operator $\rm{proj}$: $V$ transforms over time into $\ker L_{\chi}$ under $K^t$. The nullspace $\ker \hgen_{\vg}$ thus defines an important analysis tool for heat kernel methods, in that if one can characterize $\ker \hgen_{\vg}$, then one has an impactful characterization of the behavior of $K^t$.

	\subsection{Nullspace of the Gauged Laplacian}
		Motivated by Proposition \ref{prop:projection}, we now focus on and provide a complete characterization of the nullspace of the generator $\hgen_{\vg}$.

		\begin{lemma}[{local characterization of the nullspace \cite[Proposition 2.3]{gaoGeometrySynchronizationProblems2019}}]\label{lem:loc-desc}
			Let $\phi = [\gt,\x ]$ denote a section in the nullspace of $\hgen_{\vg}$, i.e.~$\hgen_{\vg} \phi = 0$. Then one has
			\begin{equation}\label{eq:alpha-nullspace}
				\x_i = \a_{ij} \x_{j}, \quad \forall i, j \in \nu, \quad i > j, \quad \omega = [\gt,\a].
			\end{equation}
			This property is independent of the choice of gauge $\xi \in \Xi$.
	  \end{lemma}

		\begin{proof}
			By Lemma \ref{lem:dip}, we have
			\begin{equation}
			  0 = \lla \phi, \hgen_{\vg} \phi \rra_{Z} = - \sum_{i ,j \in \nu} \kappa_{ij} \| \a_{ij} \x_{j} - \x_{i} \|^{2}.
			\end{equation}
			This implies \eqref{eq:alpha-nullspace} since all terms in the sum are non-negative and $\kappa_{ij} \geq 0$.
		\end{proof}

		\begin{corollary}[independence of graph weights]
			The nullspace $\ker \hgen_{\vg}$ only depends on the graph voltage $\omega$ and not the graph weights $\kappa$.
		\end{corollary}

		Lemma \ref{lem:loc-desc} says that every node feature $\x_i$ relates to its adjacent feature vectors $\x_j$ through a matrix multiplication. Keeping track of these relations over the entire graph will be relevant for what follows.

		\begin{definition}[{net voltages \cite[Section 2.1]{grossTopologicalGraphTheory1987}}]\label{def:net-voltages}
			Let $\gamma = (i_0, \dots , i_m)$ denote a (not necessarily oriented) path on the graph $\eta$, that is, the  vertices $i_{k}$ are all pairwise disjoint and $i_{k+1} \sim i_{k}$ for all $k \in \{0,\dots,m - 1\}$.
			We then denote by
			\begin{equation}\label{eq:def-Gamma-alpha-gamma}
				\Pi_{\a}(\gamma) = \a_{i_{n},i_{n-1}} \a_{i_{n -1}, i_{n-2}} \dots \a_{i_1 i_0} \in \SO(d),
			\end{equation}
			the \textit{net voltage} of $\a \in \mc{A}$ along the path $\gamma$.
		\end{definition}

		\begin{remark}[parallel transporters are net voltages]
			The matrices $\Pi_{\a}(\gamma)$ are also known as \textit{parallel transport matrices} in the language of geometric graph theory and lattice gauge theory \cite{Gattringer:2010zz,gaoGeometrySynchronizationProblems2019,bodnarNeuralSheafDiffusion2023}.
		\end{remark}

		\begin{proposition}[{\cite[Proposition 4]{bodnarNeuralSheafDiffusion2023}}]
			A section $\phi = [\gt,\x]$ is an element of $\ker \hgen_{\vg}$ with $\vg = (\eta,\kappa,[\xi,\a])$ if and only if for all paths $\gamma = (i_0, \dots , i_m)$, we have the relation
			\begin{equation}
				\x_{i_{m}} = \Pi_{\a}(\gamma) \x_{i_{0}}.
			\end{equation}
		\end{proposition}

		\begin{proof}
			This is immediate from the definition \eqref{eq:def-Gamma-alpha-gamma} of $\Pi_{\a}(\gamma)$ and Lemma \ref{lem:loc-desc} by transitively extending the relation \eqref{eq:alpha-nullspace} along the path.
		\end{proof}

		This property causes a stark contrast between the kernel of $L_\chi$ and the kernel of many other graph Laplacians \cite{hein2007graph}.

		For instance, see Figure \ref{fig:examples for graphs} which depicts two simple voltage graphs and the dimensions of the respective kernels of $L_\chi$. Notably, the kernels can be characterized without the specification of a weight function $\kappa$; the preceding proposition shows that the kernel of $L_\chi$ is \textit{independent} of $\kappa$. 			 
		The examples illustrate that the nullspace of the gauged Laplacian is qualitatively different to the kernel of basic graph Laplacians. While the kernel of many graph Laplacians is one-dimensional and generated by the \textit{constant} node signal, the kernel of the gauged Laplacians depends on the choice of voltage $\alpha$. The shown voltage graphs are examples where the dimension of the kernel is neither 1-dimensional, nor contains the constant node signal (the elements of the kernel are constant however \textit{up to gauge transformations}, as shown in Theorem \ref{thm:spanning-isom}).

		\begin{remark}[relation to the continous holonomy principle]
			This last proposition is to be compared to the \textit{holonomy principle} \cite[Theorem 4.9]{baumEichfeldtheorieEinfuehrungDifferentialgeometrie2014}, which says that the parallel sections of a vector bundle are in one-to-one correspondence to the \textit{holonomy invariant vectors}. The last result can be thought of as a discrete version of this statement.
		\end{remark}

		\begin{figure}[t]
			\begin{subfigure}[c]{0.3\textwidth}
				\includegraphics[width = \textwidth]{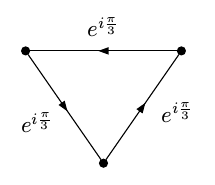}
			\end{subfigure}
			\hfill
			\begin{subfigure}[c]{0.6\textwidth}
				\includegraphics[width = \textwidth]{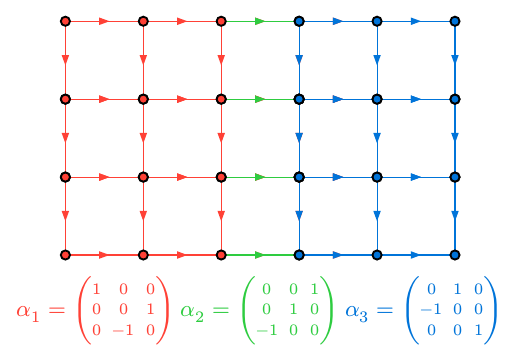}
			\end{subfigure}
			\caption{
				Examples of non-synchronizable voltages: a graph with $\dim \ker L_{\chi} = 0$ (left panel) and $\dim \ker L_{\vg} = 1$ (right panel) for the groups $\U(1)$ and $\SO(3)$ respectively.
			}
			\label{fig:examples for graphs}
		\end{figure}

		\begin{figure}[b]
			\centering
			\includegraphics[width=1\textwidth]{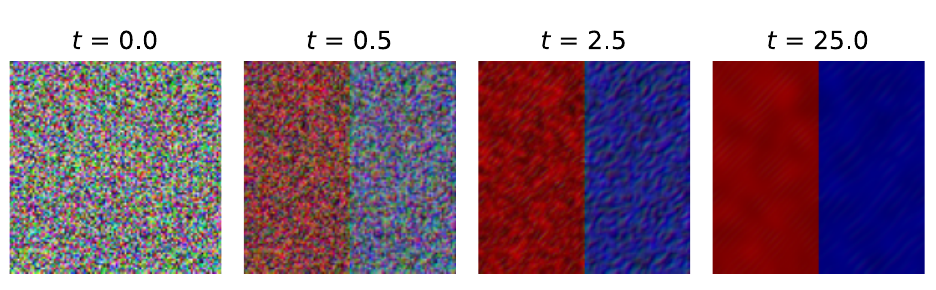}
			\caption{
				Visualization of heat kernel transformation $K^{t}_{\chi}$applied to the data shown by the left-most panel, for various time steps $t=\{0,0.5,2.5,25\}$, computed using a simple Euler scheme. The $\R^{3}$ node feature vectors are visualized in terms of colors. The given $\SO(3)$-voltage $\alpha$ configuration is depicted in the right panel of Figure \ref{fig:examples for graphs}. The heat kernel transformation causes the projection of the initial configuration over time to the one-dimensional nullspace of $L_{\chi}$. This contrast with the typical long-time transformations of classical graph Laplacians to constant output and also highlights a non-dissipative tendency of the gauged heat kernel: the channel-wise means of the initial data are not preserved during diffusion.
			}
		\end{figure}

		\begin{corollary}[{\cite[Lemma 6]{bodnarNeuralSheafDiffusion2023}}]\label{cor:L-dim-ub}
		The dimension of the nullspace $\ker \hgen_\chi \subset \Gamma$ is bounded by $d$
			\begin{equation}
				\dim \ker \hgen_\chi \leq d.
			\end{equation}
		\end{corollary}

		\begin{proof}
			Pick any vertex $i \in \nu$. The node feature $\x_{i} \in \R^d$ fixes the node features at all other vertices of the graph by Lemma \ref{lem:loc-desc}. This means that the dimension of $\ker\hgen_\chi$ cannot exceed the degrees of freedom of the node feature $\x_{i}$ which is given by $d$.
		\end{proof}

	\subsection{Nullspace Analysis via Voltage Synchronizability}
		We continue the analysis of the nullspace $\ker \hgen_{\vg}$ in order to provide a more concrete description.

		\begin{definition}[synchronization for voltages]\label{def:synchronization}
			Let $\a \in \mc{A}$ be a graph voltage. We call $\a$ \textit{synchronizable} if there is a group element $\gt \in \Xi$ such that
			\begin{equation}
				\a_{ij} = \gt_i^{-1} \gt_{j}, \quad \forall i > j \in \veps.
			\end{equation}
			Then the transformation $\gt$ is called the \textit{synchronization} for $\a$. We say $\a$ is \textit{synchronizable on a subgraph} $\eta' \subseteq \eta$ if the above relation holds for the edges of $\eta'$ induced by $\eta$.
		\end{definition}
		If a voltage can be synchronized, then its information is \textit{compressible} in the sense that the information $\alpha_{ij}$ that is distributed over all edges of the of the graph, can now be represented purely in terms of the information provided by the nodes.

			A voltage $\a \in \mc{A}$ is synchronizable if and only if there is a gauge transformation $\xi \in \Xi$ such that $\xi \a = \mbf{1}_{\mc{A}}$, where $\mbf{1}_{\mc{A}}$ denotes the voltage assigning the identity matrix to all edges. This leads to the proposition below.
		\begin{proposition}[gauge invariant synchronizations]\label{prop:char-sync}
			Let $\omega = [\gt,\a] = [\gt',\a']$ denote a gauge invariant graph voltage with two graph voltages $\a,\a'$ chosen as representatives. Then $\a'$ is synchronizable iff $\a$ is synchronizable. This means that `synchronizable' is a gauge invariant statement. The voltage $\omega$ is synchronizable iff there is an element $\gt'' \in \Xi$ such that
			\begin{equation}
				\omega = [\gt'',\mbf{1}_{\mc{A}}],
			\end{equation}
		\end{proposition}

		\begin{lemma}[{trees are synchronizable \cite[Section 3.2.1]{Gattringer:2010zz}\cite[Remark 6.2]{jiangGaugeTheoryGraphs2023}}]\label{lem:tree-synchronization}
			Let $\theta \subseteq \eta$ denote a tree and let $\a$ denote a voltage on the subgraph $\theta$. Then $\a$ is synchronizable on $\theta$.
		\end{lemma}

		\begin{proof}
			We construct a synchronization $\gt$. Identify an arbitrary node $i_{0}$ as root of the tree and set $\gt_{i_0} = g_0$ for an arbitrary group element $g_{0} \in \SO(d)$. This will fix $\gt$ on all other nodes of the tree $\theta$ as follows. Let $i_1,\dots,i_l$ denote the children of $i_0$ and set
			\begin{equation}
				\gt_{i_{j}} = \a_{i_{j},i_0} g_0, \quad \forall j = 1, \dots, l.
			\end{equation}
			which ensures $\a_{i_{j} i_{0}} = \gt_{i_{j}} \gt_{i_0}^{-1}$. Then apply the same step to the subtrees starting at the nodes $i_1, \dots , i_{l}$ and iterate these steps until the leaves. This fixes $\gt$ on all the nodes of $\theta$ and defines a synchronization for $\a$ on $\theta$.

			An equivalent way of defining $\gt$ is to fix first $\gt_{i_{0}} = \II$ and then set $\xi_{j} = \Pi_{\a}(\gamma_{i_{0}j})$ according to Definition \ref{def:net-voltages}, where $\gamma_{i_{0}j}$ is the unique path in the tree $\eta$ that connects the two nodes $i_{0},j \in \theta$. This results in the same synchronization as above for $g_{0} = \II$.
		\end{proof}
		\begin{remark}[exact inference on acyclic graphs]
		The statement of Lemma \ref{lem:tree-synchronization} resembles various situations involving \textit{exact} inference on \textit{acyclic} graphs, via some form of `message passing'; see \cite{Aji:2000aa} for a basic reference.
		\end{remark}
		\begin{corollary}[tree gauges]\label{cor:tree-gauge}
			Let $\omega$ be a \cc and $\theta \subset \eta$ be a tree. Then there is a representation $\omega = [\gt,\a]$ such that
			\begin{equation}
				\a_{ij} = \II, \qquad \text{for all edges $i > j$ in $\theta$.}
			\end{equation}
		\end{corollary}

		\begin{figure}
			\includegraphics[width=0.45\textwidth]{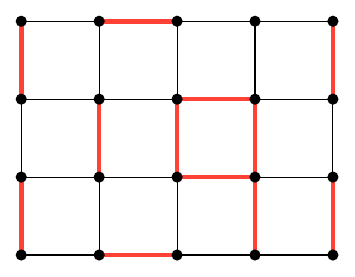}
			\caption{The edges $\veps_1$ of the complement graph of a spanning tree $\eta_{0}$ of the grid graph $\eta$, marked with red.}\label{fig:spanning}
		\end{figure}

		\begin{theorem}[spanning tree isomorphism for the nullspace]\label{thm:spanning-isom}
			Let $\eta_0$ denote a spanning tree of $\eta$ and denote by $\veps_1$ the edges of $\eta$ which are not present in the spanning tree (see Figure \ref{fig:spanning}) and let $\omega$ be a voltage on $\eta$.  Then there is a representation $\omega = [\gt,\a]$ and an induced isomorphism
			\begin{equation}
				\ker \hgen_{\vg} \cong
				\begin{cases}
					\R^{d} & \text{if }\veps_1 = \emptyset, \\
					\bigcap_{ i > j \in \veps_1}
					\E_1(\a_{ij}) &\text{else},
				\end{cases} 
			\end{equation}
			where $\E_1(\cdot)$ denotes the eigenspace of the matrix to the eigenvalue 1.
		\end{theorem}

		\begin{proof}
			Choose the representation $[\gt,\a]$ such that
			\begin{equation}\label{eq:proof-choice-alpha-1}
			\a_{ij} = \II
			\end{equation}
			for all edges $i > j$ in $\eta_{0}$. This is possible by Corollary \ref{cor:tree-gauge}. We compute the nullspace of $L_{\chi}$ -- recall its definition \eqref{eq:def-gauged-Laplacian} -- with respect to this representation.

			 Assume $\x$ lies in the kernel of $\Lpl_{\eta,\kappa,\a}$. For any node $i \in \nu$, there is a neighbor $j$ in the spanning tree $\eta_{0}$. By Lemma \ref{lem:loc-desc}, this implies
			\begin{equation}
				\x_{i} = \a_{ij} \x_{j} = \II \x_{j} = \x_{j}.
			\end{equation}
			In view of \eqref{eq:proof-choice-alpha-1}, this means that for any node $i \in \nu$, the feature vector $\x_{i}$ must be equal to $\x_{j}$ for every neighbor $j$ of $i$ in $\eta_{0}$.
			By transitivity this implies that $\x_k = \x_i$ for all other nodes $k \in \nu$ connected to $i$ via a path inside $\eta_{0}$.
			Since $\eta_0$ is a spanning tree, any two nodes are connected by a path in $\eta_{0}$. This means that $\x_i = \x_0$ for some $\x_{0} \in \R^d$, for all $i \in \nu$.

			There are additional requirements imposed by the edges $i > j \in \veps_{1}$. For all these edges we have
			\begin{equation}
				\x_0 = \a_{ij} \x_0, \quad \forall i > j \in \veps_{1},
			\end{equation}
			which is the eigenvalue equation for $\a_{ij}$ with eigenvalue 1. This means that the vector $\x_0$ lies in the subspace
			\begin{equation}
				\x_0 \in \bigcap_{i > j \in \veps_1} \E_1 (\a_{ij}).
			\end{equation}

			The concrete isomorphism then looks as follows. Pick any node $k \in \nu$ and define
			\begin{equation}
				\ker \hgen_{\vg} \to \R^d, \quad \phi = [\gt,\x ] \mapsto \x_k.
			\end{equation}
			This mapping is independent of the chosen node $k$ since the node features are constant across the vertices inside the kernel, when we choose the gauge $\xi$. In words: transform $\phi$ to the right gauge $\gt$ and pick the $k$-th row of the feature matrix $x$. This linear mapping is injective as shown by the arguments above and its image is $\R^{d}$ if $\veps_1 = \emptyset$ and $\bigcap_{i > j \in \veps_1} E_1(\a_{ij})$ else.
		\end{proof}

			The following corollary was also noted in \cite[Proposition 2.1]{gaoGeometrySynchronizationProblems2019} and \cite[Lemma 6]{bodnarNeuralSheafDiffusion2023}.
		
		\begin{corollary}\label{cor:synchronizable}
			$\omega$ is synchronizable on $\eta$ if and only if $\dim \ker \hgen_{\vg} = d$.
		\end{corollary}

		\begin{proof}
			Assume $\dim \ker \hgen_{\vg} = d$ and let $\veps_1$ denote again the complement of the set of edges of a spanning tree $\eta_0 \subset \eta$, with tree gauge representation $\omega=[\gt,\a]$. 
			This implies that $\alpha_{ij} = \II$ on all edges of spanning tree $i > j$ in $\eta_{0}$. 
			If $\veps_{1} = 0$, then $\omega$ is synchronizable by Proposition \ref{lem:tree-synchronization}. Else, we have by assumption,
			\begin{equation}
				\dim \bigcap_{i > j \in \veps_{1}} \E_{1}(\a_{ij}) = d.
			\end{equation}
			Since the dimension of the intersection of the eigenspaces is $d$, we infer that $E_{1}(\a_{ij}) = \R^{d}$ for all edges in $\eta_1$, which implies $\a_{ij}=\II$ for all $i>j$ in $\eta_{1}$. We conclude that $\a_{ij}=\II$ on all edges $i>j$ in the entire graph $\eta$ and hence that $\alpha$ is synchronizable by Proposition \ref{prop:char-sync}. 
			The converse statement follows by stating the arguments in reverse order.
		\end{proof}

		For the next corollary, we recall our assumption that the graph $\eta$ is connected, i.e. has only one connected component, and $N = |\nu|$. In order to state the next result, we note that we can interpret gauges $\xi \in \Xi$ as linear mappings $X \to X$ by the group action of $\Xi$ on $X$. We furthermore identify $X = \R^{N \times d}$ with the tensor product $\R^{N}\otimes\R^{d}$.

		\begin{corollary}[tensor decomposition via synchronizability]\label{cor:tensor-decomp}
			Let $\omega = [\gt,\a]$ denote a graph voltage and $\Lpl_{\eta,\kappa,\alpha}$ a representation of $L_{\vg}$ with respect to the gauge $\gt$, according to Definition \ref{def:ginv-k}. Then $\omega$ is synchronizable if and only if there is a gauge $\xi \in \Xi$, a matrix $R \in \o{SO}(N)$ and a diagonal matrix $\Lambda \in \R^{N \times N}$ such that
			\begin{equation}\label{eq:decomp}
				\Delta_{\eta,\kappa,\alpha} = \xi^{-1} (R \Lambda R^{\top} \otimes \II) \xi,
			\end{equation}
			as operators $X \to X$, and there is a unique $i \in \{1,\dots,N\}$ with $\lambda_i \coloneqq \Lambda_{ii} = 0$ and $\lambda_j \coloneqq \Lambda_{jj} > 0$ for all $j \neq i$.
		\end{corollary}

		\begin{proof}
			First assume that $\Delta_{\eta,\kappa,\alpha}$ admits a decomposition as defined in \eqref{eq:decomp}. We then prove that $\dim \ker \Delta_{\eta,\kappa,\alpha} = d$, which shows synchronizability by Corollary \ref{cor:synchronizable}.
			A linear isomorphism $\R^{d} \to \ker \Delta_{\eta,\kappa,\alpha}$ can be specified as follows. Denote by $r_i$ the $i$-th column of the matrix $R$. Then we consider the mapping
			\begin{equation}\label{eq:proof-decomp-isom}
				\R^{d} \to \ker \Delta_{\eta,\kappa,\alpha}\; ,
				\quad
				v \mapsto \xi^{-1}(r_{i} \otimes v).
			\end{equation}
			This mapping is linear and injective. It is well-defined in the sense that it maps into the kernel of $\Delta_{\eta,\kappa,\alpha}$
			\begin{subequations}
			\begin{align}
				\Delta_{\eta,\kappa,\alpha}\xi^{-1}(r_{i}\otimes v)
				&=
				\xi^{-1} (R \Lambda R^{\top} \otimes \II) \xi\xi^{-1}(r_{i}\otimes v) \\
				&=
				\lambda_i \xi^{-1}(r_i \otimes v) \\
				\overset{\lambda_i = 0}&{=} 0,
			\end{align}
			\end{subequations}
			where the second equation holds since $R^{\top} r_i$ is a vector with 1 in the $i$-th entry and 0 else, and $R \in \o{SO}(N)$. It remains to show surjectivity. Note that $\Delta_{\eta,\kappa,\alpha}$ and $R \Lambda R^{\top}\otimes \II$ are similar as linear operators. Thus, they have isomorphic kernels, with isomorphism given by
			\begin{equation}
				\ker \Delta_{\eta,\kappa,\alpha}\to \ker (R \Lambda R^{\top}\otimes \II), \quad
				x \mapsto \xi x.
			\end{equation}
			We first show surjectivity for the mapping
			\begin{equation}\label{eq:simple-iso}
				\R^{d} \to \ker (R \Lambda R^{\top}\otimes \II), \quad
				v \mapsto r_{i} \otimes v,
			\end{equation}
			from which surjectivity of \eqref{eq:proof-decomp-isom} follows. The mapping \eqref{eq:simple-iso} is surjective due to
			\begin{equation}
				\ker(R \Lambda R^{\top} \otimes \II) = \ker (R \Lambda R^{\top} )\otimes \R^{d},
			\end{equation}
			and the fact that $\ker (R \Lambda R^{\top})$ is the span of the vector $r_i$ by the assumptions of the Corollary. We conclude that \eqref{eq:proof-decomp-isom} is an isomorphism and hence that $\dim \ker \Delta_{\eta,\kappa,\alpha} = d$.

			Conversely, assume that $\a$ is synchronizable with synchronization $\gt$. Then $\a_{ij} = \gt_{i}\gt_{j}^{-1}$ by Definition \eqref{def:synchronization} and furthermore by \eqref{eq:def-lalpha},
			\begin{equation}
				\big(\Lpl_{\eta,\kappa,\a}(\x)\big)_i = - \gt_i \sum_{i,j \in \nu} \frac{\kappa_{ij}}{Z_i} ( \gt_j^{-1} \x_j - \gt_{i}^{-1} \x_i ).
			\end{equation}
			This yields the explicit expression
			\begin{equation}\label{eq:similariy-delta-k}
				\Lpl_{\eta,\kappa,\a} = \gt (\Delta_{\eta,\kappa} \otimes \II) \gt^{-1} :X \to X,
			\end{equation}
			where $\Delta_{\eta,\kappa}$ denotes the weighted graph Laplacian on $\eta$ with weights $\kappa$ and random walk normalization given by Equation \eqref{eq:def-y-diag}.
			The Laplacian $\Delta_{\eta,\kappa}=R\Lambda R^{\top}$ can be diagonalized with a matrix $R  \in \rm{O}(N)$ and eigenvalues $\lambda_{k}=\Lambda_{kk} \geq 0$, where exactly one of the eigenvalues is 0 due to the assumption that the graph $\eta$ is connected. Thus
			\begin{equation}
				\Lpl_{\eta,\kappa,\a} = \gt (R \Lambda R^\top \otimes \II) \gt^{-1}.
			\end{equation}
		\end{proof}

		\begin{corollary}[spectra of synchronizable heat kernels]
			Let $\omega = [\gt,\a]$ denote a synchronizable graph voltage.
			Then the spectrum of $\hgen_{\chi}$ is equal to the spectrum of the weighted graph Laplacian $\Delta_{\eta,\kappa}$. Furthermore, the eigenspaces $E_{\lambda}(\hgen_{\vg})$ are isomorphic to $E_{\lambda}(\Delta_{\eta,\kappa}) \otimes \R^{d}$.
		\end{corollary}
		\begin{proof}
			Let $\Delta_{\eta,\kappa,\alpha}$ denote a representation of $\hgen_{\chi}$. From  Equation \eqref{eq:similariy-delta-k} in the proof of Corollary \ref{cor:tensor-decomp}, we deduce that $\Delta_{\eta,\kappa,\alpha}$ and $\Delta_{\eta,\kappa} \otimes \II$ are \textit{similar} as linear operators and thus share a spectrum and have isomorphic eigenspaces. The claim then follows from the observation that the spectrum of $\Delta_{\eta,\kappa} \otimes \II$ is equal to the spectrum of $\Delta_{\eta,\kappa}$ and that the eigenspaces of $\Delta_{\eta,\kappa} \otimes \II$ are the tensor product of the eigenspaces of  $\Delta_{\eta,\kappa}$ and $\R^{d}$.
		\end{proof}

%% file: tex2/yangmills.tex
\section{Yang Mills Curvature Analysis}
\label{sec:ym}

Both Theorem \ref{thm:spanning-isom} and Corollary \ref{cor:tensor-decomp} show that the synchronizability of a voltage $\omega$ is central to the transformation properties of the heat kernel $K_{\vg}$. In this section, we analyze the synchronizability of graph voltages $\omega$ independently of heat kernels.
We provide here an analysis of synchronizability in terms of \textit{Yang-Mills} energies, which are an important tool in the theory of smooth vector bundles. Our exposition closely relates to the work \cite{gaoGeometrySynchronizationProblems2019} in that we also link synchronization properties to holonomy properties of graph voltages. While \cite{gaoGeometrySynchronizationProblems2019} analyzes synchronizability in terms of \textit{synchronization bundles}, we choose here geometric gauge theory as the vantage point which allows us to extend the analysis \textit{beyond flat voltages}.

\subsection{Holonomy for Graph Voltages}
    We define holonomy for a discrete \cc $\a$ on a graph $\eta$.

	\begin{definition}[loops of a graph $\eta$]
		A \textit{loop} of $\eta$ is an ordered sequence $l = (i_0,\dots,i_m)$ of vertices $i_{j} \in \nu$ such that $i_{m}=i_{0}$, $i_{j} \sim i_{j+1}$ for all $j \in \{0,\dots,m - 1\}$ and $i_{j} \neq i_{j'}$ if $j \neq j'$. The vertex $i_{0}$ is called the \textit{base point} of $l$. We denote by $\Lps_{+}(\eta)$ the set of loops of $\eta$.
	\end{definition}

	We define an equivalence relation on the space of loops as follows. Two loops $l,l'$ are \textit{equivalent} if the ordered sequence $l'$ can be obtained from the sequence $l$ by applying a cyclic permutation and, if necessary, a permutation that reverses the order.

	\begin{definition}[cycles of $\eta$]
		A \textit{cycle} of $\eta$ is an equivalence class of loops $\ell$ under the relation defined above. We denote by $\Lps(\eta)$ the set of cycles of $\eta$.
	\end{definition}

	\begin{definition}[holonomy]\label{def:holonomy}
		The \textit{holonomy} of a voltage $\a \in \mc{A}$ along a loop $l\in \Lps_{+}(\eta)$ with $l = (i_{0},\dots,i_{m})$ is defined as
		\begin{equation}
			\Hol_{\a}(l) \coloneqq \prod_{k=1}^m \a_{i_{k - 1} i_{k}}.
		\end{equation}
		We define the \textit{scalar holonomy} as
		\begin{equation}
			\mfk{h}_\a(l) \coloneqq \tr(2 \II - \Hol_{\a}(l) - \Hol_{\a}(l)^{\top}).
		\end{equation}
	\end{definition}
	The holonomy $\Hol_{\a} : \Lps_+(\eta) \to \SO(d)$ is well-defined for any loop $l$, yet depends on the choice of a base point and orientation. The scalar holonomy $\mfk{h}_\a\colon \Lps(\eta) \to\mathbb{R}_{+}$, on the other hand, neither depends on the base point nor on the orientation of a loop $l$, hence is well-defined for any cycle $\ell$.

	\begin{proposition}[scalar holonomy does not depend on base point]\label{prop:scalar-hol-basepoint}
		If $l,l' \in \Lps_+(\eta)$ denote two equivalent oriented loops, then
		\begin{equation}
			\hol_{\a}(l) = \hol_{\a}(l').
		\end{equation}
		The scalar holonomy is well-defined as a function $\hol_{\a}: \Lps(\eta) \to \R_+$
	\end{proposition}
	\begin{proof}
		The two $\Hol_{\a}(l)$ and $\Hol_{\a}(l')$ can be obtained from each other by composing conjugation and inversion operations, since the same loop is traversed, but with different ordering and maybe different orientation.
		A different ordering results in a conjugation operation on $\Hol_{\a}(l)$. Reversing an orientation is equivalent to inverting $\Hol_{\a}(l)$.
		As for the scalar holonomy, we symmetrize $\Hol_{\a}(l)$ and take the trace. Due to symmetrization, the scalar holonomy becomes invariant under inversion, which amounts to transposition of $\Hol_{\a}(l) \in \SO(d)$. Due to taking the trace, $\hol_{\a}(l)$ is invariant with respect to conjugation, which results from changing the base point.

		The scalar holonomy maps to $\R_{+}$, since the matrices $\Hol_{\a}(l)$ lie in $\SO(d)$ such that $\Hol_{\a}(l) + \Hol_{\a}(l)^{\top}$ has its spectral range in $[-2,2]$. Consequently, the matrix $2 \II - \Hol_{\a}(l) - \Hol_{\a}(l)^{\top}$ is symmetric and positive semi-definite with spectral range $[0,4]$, which the trace maps to the interval $[0,4d]$.
	\end{proof}

	\begin{proposition}[scalar holonomy completely characterizes holonomy]\label{prop:scalar-hol-hol}
		We have
		\begin{equation}
			\Hol_{\a}(l) = \II \quad\iff\quad
			\hol_{\a}(l) = 0 .
		\end{equation}
	\end{proposition}

	\begin{proof}
		\renewcommand\leo\sigma
		If $\Hol_{\a}(l) = \II$, then $\hol_{\a}(l) = 2d - d - d = 0$.

		If we have $\hol_{\a}(l) = 0$, then the positive-semidefinite symmetric matrix $2 - \Hol_{\a}(l) - \Hol_{\a}(l)^{\top}$ has trace 0, which implies that the matrix itself is 0. Therefore, multiplying the equation $2 = \Hol_{\a}(l) + \Hol_{\a}(l)^{\top}$ from the left by $\Hol_{\a}(l)$ and rearranging gives
		\begin{equation}\label{eq:hol-poly}
			\Hol_{\a}(l)^{2} - 2\Hol_{\a}(l) + \II = 0.
		\end{equation}
		Any eigenvalue $\leo$ of $\Hol_{\a}(l)$ therefore satisfies the polynomial equation $\leo^{2} - 2 \leo + 1 = (\leo - 1)^{2} = 0$, that is $\leo = 1$ must hold. Thus, $\Hol_{\a}(l)$ is a diagonalizable matrix with all eigenvalues equal to 1 and hence $\Hol_{\a}(l) = \II$.
	\end{proof}

	\begin{proposition}[gauge transform of holonomy]\label{prop:ginv-hol}
		If $\gt \in \Xi$ denotes a gauge transformation, then we have the property
		\begin{equation}\label{eq:Hol-gauge}
			\Hol_{\gt \a} (l) = \gt_{i_0}^{-1} \Hol_{\a}(l) \gt_{i_{0}},
		\end{equation}
		where $i_0$ is the base point for $l \in \Lps_{+}(\eta)$ and
		\begin{equation}\label{eq:hol-gauge}
			\hol_{\a}(l) = \hol_{\gt \a}(l).
		\end{equation}
	\end{proposition}

	\begin{proof}
		Equation \eqref{eq:Hol-gauge} follows from  Definition \ref{def:holonomy} and the transformation property $(\gt \a)_{i j} = \gt_i \a_{i j} \gt_{j}^{-1}$. Equation \eqref{eq:hol-gauge} holds since the trace is invariant under conjugation.
	\end{proof}

\subsection{Holonomy as an Obstruction}
	\begin{proposition}[synchronizable voltages have trivial holonomy]\label{prop:triv-hol-sync}
		If $\a$ is synchronizable then $\Hol_\a(l) = \II$, for every $l \in \Lps_{+}(\eta)$.
	\end{proposition}

	\begin{proof}
		A synchronizable voltage $\a$ can be written as $\a_{i j} = \gt_{i} \gt_{j}^{-1}$ for some $\gt \in \Xi$. From Definition \ref{def:holonomy} the equation $\Hol_{\a}(l) = \II$ follows, for every loop $l \in \Lps_{+}(\eta)$.
	\end{proof}
	As a consequence of the preceding statement, a loop with non-trivial holonomy $\Hol_{\alpha}$ defines an obstruction for $\a$ to be synchronizable.

	\begin{theorem}[{\cite[Corollary 2.1]{gaoGeometrySynchronizationProblems2019}}]
	\label{thm:holonomy-obstruction}
		A voltage $\a$ is synchronizable if and only if $\Hol_\a(l) = \II$ for all loops $l \in \Lps_+(\eta)$.
	\end{theorem}

	We first prove the following lemma.
	\begin{lemma}[path-independent net voltage]\label{lem:path-indep-pt}
		Let $\a$ denote a graph voltage with $\Hol_{\a}(l) = \II$ (trivial holonomy) for all loops $l \in \Lps_{+}(\eta)$. Then the net voltage on $\eta$ is path-independent: If there are two paths $\gamma, \gamma'$ in the graph which start and end at the same vertices, then $\Pi_{\a}(\gamma) = \Pi_\a(\gamma') \in \SO(d)$.
	\end{lemma}
	\begin{proof}
		Let $\gamma,\gamma'$ denote two such paths as in the statement of the lemma with transport matrices $\Pi_{\a}(\gamma)$ and $\Pi_{\a}(\gamma')$. Then consider the loop $l$ formed by composing the two paths $\gamma$ and $\gamma'$, which is possible because they start and end at the same points. We can orient the paths such that we can write $\Hol_{\a}(l) = \Pi_{\a}(\gamma')^{\top} \Pi_{\a}(\gamma) = \II$, from which we conclude that $\Pi_{\a}(\gamma') = \Pi_\a(\gamma)$.
	\end{proof}

	\begin{proof}[Proof of Theorem \ref{thm:holonomy-obstruction}]
		Since $\Hol_{\a}(l)=\II$ for all loops, the net voltage $\Pi$ along the graph is path-independent, by Lemma \ref{lem:path-indep-pt}.

		The reasoning below is based on the following facts.
		\begin{enumerate}
			\item Let $\eta_0$ denote a spanning tree of the graph $\eta$.
			\item Synchronize $\a$ along the tree $\eta_0$, which is possible by Lemma \ref{cor:tree-gauge}. This defines a synchronization $\gt$ on all graph vertices.
			\item This synchronization $\gt$ also synchronizes the edges of $\eta \setminus \eta_0$ since net voltage is path-independent.
		\end{enumerate}
		The spanning tree $\eta_{0}$ contains all  vertices of the graph such that the partial synchronization found on the tree $\eta_{0}$ defines a synchronization matrix $\xi_{i}$, for each vertex $i \in \nu$.
		The claim is that these matrices satisfy
		\begin{equation}\label{eq:hol-obstr-alpha}
		\alpha_{ij} = \xi_{i} \xi_{j}^{-1}
		\end{equation}
		also for the edges $i > j$ in the complement $\eta\setminus \eta_{0}$ of the spanning tree.

		Let $i> j$ be such an edge. Both endpoints $i$ and $j$ are connected to the root node $i_{0}$ of the spanning tree via unique paths $\gamma_{i}$ and $\gamma_{j}$. Hence we can form a closed loop $l$ by composing the paths $\gamma_{i}$ and $\gamma_{j}$ and the edge $i > j$. By the trivial holonomy assumption, we have
		\begin{equation}
			\Hol_{\a}(l) = \Pi_{\a}(\gamma_{i})\Pi_{\a}(\gamma_{j})^{-1}\alpha_{ij}  = \II,
		\end{equation}
		which implies
		\begin{equation}
			\alpha_{ij} = \Pi_{\a}(\gamma_{i})^{-1} \Pi_{\a}(\gamma_{j}).
		\end{equation}
		But by the construction of the tree gauge in Lemma \ref{lem:tree-synchronization}, we know that $\xi_{i} = \Pi_{\a}(\gamma_{i})^{-1}$ holds for all $i \in \nu$. Comparing this with the above equation shows that $\a_{ij} = \xi_{i} \xi_{j}^{-1}$ and verifies \eqref{eq:hol-obstr-alpha}.
	\end{proof}

\subsection{Yang-Mills Energy}
	We now characterize the Yang-Mills energy as a \textit{local} measure of holonomy. It has the advantage that it can be computed by only considering the 3-cliques which form a subspace of the space of all loops $\Lps(\eta)$, which is often too large to be checked for holonomy. In the general geometric context, the Yang-Mills energy measures the \textit{curvature} of connections on principal bundles. See, e.g., \cite[Section 10]{Petersen:2016aa} for a brief overview of curvature and holonomy and \cite[Sections 12,13] {Berger:2003aa} for an extended overview including the topological viewpoint.

	\begin{definition}[3-cliques (triangles)]
		The \textit{3-cliques $\Lps^3(\eta)$} of $\eta$ is the set of cycles in $\Lps(\eta)$ of length 3.
	\end{definition}

	\begin{definition}[Yang-Mills energy, flat voltages]\label{def:ym-flat}
		The \textit{discrete Yang-Mills energy} is the function
		\begin{equation}
			\mc{YM} : \Omega \to \R_{+}, \quad \omega = [\gt,\a] \mapsto \sum_{\ell \in \Lps^{3}(\eta)} \hol_{\a}(\ell).
		\end{equation}
		A voltage $\omega$ is called \textit{flat} if $\mc{YM}(\omega) = 0$.
	\end{definition}
	This functional is well-defined in that it does not depend on the choice of representative $[\xi,\alpha]$ for $\omega$, by Proposition \ref{prop:ginv-hol}.
	By the non-negativity of $\hol$, a voltage $\omega = [\xi,\alpha]$ is flat if and only if $\hol_{\alpha}(\delta) = 0$ for all 3-cliques $\delta \in \mc{L}^{3}(\eta)$.

	\begin{proposition}[synchronizable voltages are flat]\label{prop:synch-flat}
		Let $\omega$ be a synchronizable voltage, then $\mc{YM}(\omega) = 0$.
	\end{proposition}
	\begin{proof}
		The statement follows from Proposition \ref{prop:triv-hol-sync}.
	\end{proof}

	\begin{remark}[motivation from discrete gauge theory]
		Definition \ref{def:ym-flat} of the Yang-Mills energy is motivated by \textit{lattice gauge theory} \cite{Gattringer:2010zz} and \textit{gauge theory on graphs} \cite{baezSpinNetworkStates1996,jiangGaugeTheoryGraphs2023}. In lattice gauge theory one usually works with cubical meshes where the smallest cycles have length 4, from which a slight mismatch originates between our set-up and the convention of lattice gauge theory.
	\end{remark}

	\begin{figure}[t]
		\includegraphics[width=0.65\textwidth]{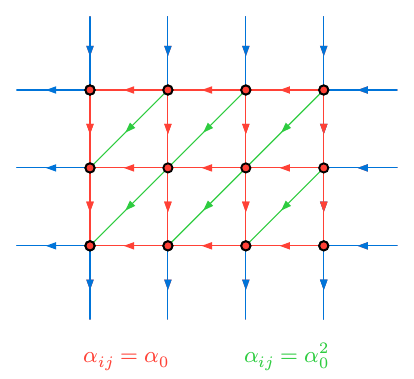}
		\caption{Regular 2 dimensional torus grid graph with flat but non-synchronizable voltage. Periodic boundary edges are marked with blue.}\label{fig:flat-non-triv}
	\end{figure}

	\begin{example}[flat but non-synchronizable voltage]
		Consider the regular 2 dimensional torus grid graph with additional diagonal edges as depicted by Figure \ref{fig:flat-non-triv}. This graph contains \textit{non-contractible} cycles, which are cycles which wrap across the periodic boundaries.

		Consider the orientation of the edges depicted in Figure \ref{fig:flat-non-triv} and define a voltage $\a$ that evaluates to a fixed matrix $\a_0$ for non-diagonal grid edges and to $\a_{0}^{2}$ on all diagonal edges, see Figure \ref{fig:flat-non-triv}. Then $\a$ is flat, that is $\mc{YM}(\a) = 0$, but there are still loops $l$ such that $\Hol_\a(l) \neq \II$ and hence $\a$ is not synchronizable, in general.
		For instance, following a loop $l$ along the straight vertical path across the periodic boundary of the depicted graph, we obtain $\Hol_{\alpha}(l) = \alpha_{0}^{3}$, which need not be equal to $\II$.
	\end{example}

	This example demonstrates that flatness is a \textit{local} measure of obstruction for \ccs to be synchronizable. In other words, even if $\a$ is flat, there may still exist a loop along which there is non-vanishing curvature such that $\a$ cannot be synchronizable.

		\begin{remark}[general theory of flat bundles]
			The notion of flatness is well known in the context of vector bundles theory, see e.g. \cite{milnorExistenceConnectionCurvature1958} for an early account. Notably, flat $G$-bundles can be classified in terms of equivalence classes of $G$-representations of the fundamental group (the non-contractible loops) of the underlying base manifold \cite{baumEichfeldtheorieEinfuehrungDifferentialgeometrie2014}.

			Flatness of voltages is connected to flatness of principal fiber bundles via lattice gauge theory; a flat principal bundle induces a flat voltage through discretization \cite{rotheLatticeGaugeTheories2012}.

			It was shown in \cite{gaoGeometrySynchronizationProblems2019} that a flat voltage $\a$ enables to construct a flat principal bundle over geometric realizations of the clique complex of $\eta$. Synchronizability of the voltage is equivalent to the triviality of the corresponding principal bundle.
		\end{remark}

	The above statements show that the Yang-Mills energy can only capture information about the flatness of voltages, whereas \textit{topological} effects are ignored, so far. We now turn to a modified version of the Yang-Mills energy, augmented by topological information, such that it also reflects topological properties of the underlying graph.

\subsection{Topological Obstructions to Synchronizability}
	The Yang-Mills energy is a \textit{local measure of non-synchronizability} for voltages.
	`Locality' here means that we only consider the holonomy across the smallest loops, the 3-cliques. Accordingly, the Yang-Mills energy can only account for obstructions that arises from cycles which are boundaries
		of subgraphs which can be built from triangles, cf.~Corollary \ref{cor:flat-voltag-null}.
	There are, of course, more complex cycles in graphs, in general.
Therefore, we focus below on \textit{non-local} obstructions \textit{not} captured by the Yang-Mills energy. The main result of this section, the more powerful \textit{extended} Yang-Mills energy, is stated as Theorem \ref{thm:eym}.

	\subsubsection{Graph Homology}
		In order to characterize non-local obstructions, we recall some facts about simplicial homology on graphs \cite{jiangStatisticalRankingCombinatorial2011a,grigor2015cohomology}; see \cite{Hatcher:2002aa} for the general background.

		We denote by $C_{*}(\eta)$ the set of induced subgraphs of $\eta$. This set becomes an abelian group under the \textit{symmetric difference} operation, denoted here by $+$, where $0 \in C_*(\eta)$ is the empty set. According to the symmetric difference operation, every element of $C_{*}(\eta)$ has order 2, i.e. it holds $x + x = 0$ for all $x \in C_*(\eta)$, which implies that $C_*(\eta)$ is also a vector space over the field $\mathbb{Z}/2\mathbb{Z}$.

		Next, let $C_0(\eta)$ denote the subgroup of $C_{*}(\eta)$ generated by subgraphs consisting only of vertices of $\eta$, by $C_1(\eta)$ the subgroup generated by edges of $\eta$ and by $C_2(\eta)$ the subgroup generated by 3-cliques of $\eta$. The group $C_1(\eta)$ contains all subgraphs of $\eta$ that do not contain isolated vertices. Elements of $\Sigma \in C_{2}(\eta)$ can be represented as finite sums $\Sigma = \sum_{r} \delta_{r}$ where $\delta_{r} \in \mc{L}^{3}(\eta)$ are 3-cliques. Both $C_1(\eta)$ and $C_2(\eta)$ are  vector subspaces of $C_*(\eta)$ and the 3-cliques are a basis for $C_2(\eta)$, such that representations of the form $\Sigma = \sum_{r} \delta_{r}$ are unique.

		We define the linear maps $\partial_1 : C_1(\eta) \to C_0(\eta)$ and $\partial_2 : C_2(\eta) \to C_1(\eta)$: $\partial_1$ sends a subgraph $\sigma \in C_1(\eta)$ to its vertices which are contained in an odd number of edges and $\partial_2$ is defined on 3-cliques $\mc{L}^{3}(\eta)$ by sending a triangle to the sum of its boundary edges. The mapping can then be linearly extended to $C_{2}(\eta)$ by
		\begin{equation}
		\partial_{2} : C_{2}(\eta) \to C_{1}(\eta), \quad
			\Sigma = \sum_{r} \delta_{r} \mapsto  \sum_{r} \partial_{1} \delta_{r}.
		\end{equation}
		Equivalently, $\partial_{2}$ sends $\Sigma \in C_{2}(\eta)$ to the sum of its edges which are contained in an odd number of triangles.
		One can check that $\partial_1 \circ \partial_2 = 0$ and $\ker \partial_1$ are the Eulerian subgraphs of $\eta$, while $\o{im} \partial_2$ are the Eulerian subgraphs that can be represented as the boundary of subgraphs $\Sigma \in C_{2}(\eta)$ additively built from triangles.

		The \textit{first homology group} of $\eta$ is defined as the quotient group
		\begin{equation}
			\Hom_1(\eta) = \ker \partial_1 / \o{im} \partial_2.
		\end{equation}
		This group contains the Eulerian subgraphs modulo boundaries of subgraphs $\Sigma \in C_{2}(\eta)$. It is also a vector space over the field $\mathbb{Z}/2\mathbb{Z}$ with dimension denoted by
		\begin{equation}
			h_{1}(\eta) \coloneqq \dim H_{1}(\eta).
		\end{equation}

		The space $\Hom_{1}(\eta)$ admits a \textit{cycle basis} \cite{diestelGraphTheory2025}, i.e.~a vector space basis consisting of simple cycles $\ell_k \in \Lps(\eta), k \in \{1,\dots,h_{1}(\eta)\}$.

		There is a canonical projection
		\begin{equation}
		\Lps(\eta) \to \Hom_1(\eta),\qquad \ell \mapsto [\ell],
		\end{equation}
		and two cycles $\ell,\ell'$ are called \textit{homologous} if $[\ell'] = [\ell]$.

	\subsubsection{Holonomy of Flat Voltages}

		\begin{figure}
			\begin{minipage}{.3\textwidth}
				\includegraphics[width=\textwidth]{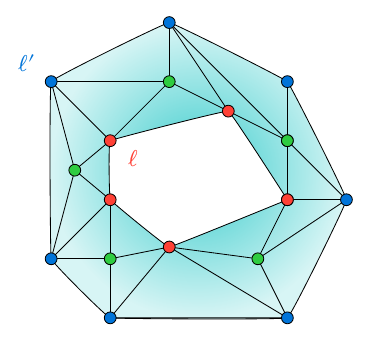}
			\end{minipage}
			\begin{minipage}{.3\textwidth}
				\includegraphics[width=\textwidth]{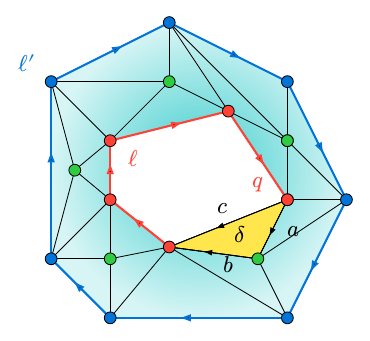}
			\end{minipage}
			\begin{minipage}{.3\textwidth}
				\includegraphics[width=\textwidth]{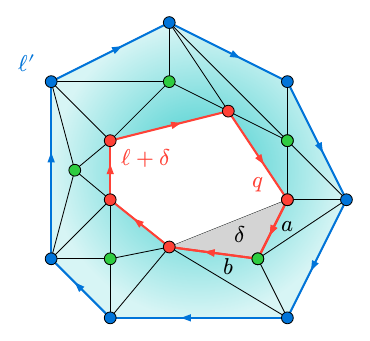}
			\end{minipage}
			\caption{Illustration for proof of Lemma \ref{lem:scalar-hol-homologous}}\label{fig:proof-homotopy}
		\end{figure}

Recall Definition \ref{def:ym-flat} and the corresponding Proposition \ref{prop:synch-flat}.

		\begin{lemma}[flat \ccs do not distinguish homologous loops]\label{lem:scalar-hol-homologous}
			If $\a$ is a \textbf{flat} voltage and $\ell,\ell' \in \Lps(\eta)$ are two homologous cycles, then
			\begin{equation}
				\hol_{\a}(\ell) = \hol_{\a}(\ell').
			\end{equation}
		\end{lemma}

		\begin{proof}
			\newcommand\xx\Sigma%
			\newcommand\yy\Sigma%
				Since the two cycles are homologous, we may choose $\yy \in C_2(\eta)$  such that $\partial_2 \yy = \ell' + \ell \in C_1(\eta)$.
				Next we choose two oriented loops $l,l'$ which represent $\ell,\ell'$, respectively.

				We may write $\yy = \sum_{r} \delta_r$, for some $\delta_r \in \Lps^3(\eta)$  and determine
				 oriented loops in $\Lps_{+}(\eta)$ which represent the cycles $\delta_r$. For simplicity,  we denote these oriented loops again by $\delta_r$.

				Let $\delta$ be such an oriented triangle of $\yy$ which shares an edge with $\ell$, and with edge matrices $a,b,c$ satisfying the relation
				\begin{equation}\label{eq:triangle}
					c = b a.
				\end{equation}
				This equation is a rearrangement of $\Hol_{\a}(\delta)=\II$, which is a consequence of the assumed flatness of $\a$.

				Choose the orientation for $l$ such that
				\begin{equation}
					\Hol_{\a}(l) = c q, \quad q \in \SO(d),
				\end{equation}
				where $q$ is the net voltage along the segments of $l$ not contained in $\delta$, see Figure \ref{fig:proof-homotopy}. Applying \eqref{eq:triangle} yields
				\begin{equation}
					\Hol_{\a}(l) = b a q = \Hol_{\a}(l + \delta),
				\end{equation}
				where $l + \delta$ is the loop $l$ expanded by $\delta$.

			By iterating these steps over the triangles $\delta_i$ we obtain
			\begin{equation}
				\Hol_{\a}(l) = \Hol_{\a}\Big(l + \sum_{r} \delta_r \Big) = \Hol_{\a}(l'),
			\end{equation}
			where all the orientations are chosen compatibly.

			The above steps rely on specific choices for the orientations of $l,l'$. Yet the scalar holonomy is invariant under changes of orientation. Hence we conclude that
			\begin{equation}
				\hol_{\a}(\ell) = \hol_{\a}(\ell'),
			\end{equation}
			which completes the proof.
		\end{proof}

			A loop $\ell$ is called \textit{null-homologous} if $[\ell] = [0] \in H_{1}(\eta)$.
			\begin{corollary}[flat voltage are unobstructed on null-homologous loops]\label{cor:flat-voltag-null}
				Let $\a$ be a flat voltage and $\ell$ a null-homologous loop. Then $\hol_\a(\ell) = 0$.
			\end{corollary}
			\begin{proof}
				For the null-homologous loop $\ell$, there is a triangle $\delta \in \mc{L}^{3}(\eta)$ such that $[\ell] = [\delta]$. Thus,
				\begin{equation}
					\hol_{\alpha}(\ell) = \hol_{\alpha}(\delta),
				\end{equation}
				by Lemma \ref{lem:scalar-hol-homologous}. Flatness of $\a$ implies $\hol_{\alpha}(\delta) = 0$.
			\end{proof}

		\begin{figure}
			\includegraphics[width=0.7\textwidth]{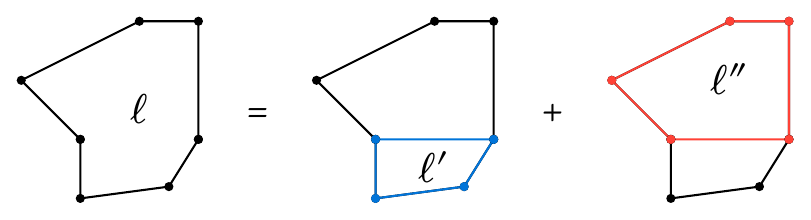}
			\caption{Decomposition of the loop $\ell$ (cf.~Lemma \ref{lem:scalar-hol-flat-loop})}\label{fig:loop-decomp}
		\end{figure}

		\begin{figure}[b]
			\centering
			\includegraphics[width=.35\textwidth]{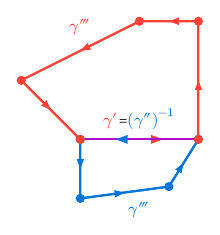}
			\caption{Decomposition of cycles $\ell'$ and $\ell''$ into partially overlapping arcs.}\label{fig:gammapp}
		\end{figure}

		\begin{lemma}[scalar holonomy and flat loops]\label{lem:scalar-hol-flat-loop}
			Let $\ell$ be a cycle such that there is a decomposition $\ell = \ell' + \ell''$ with two cycles $\ell',\ell'' \in \Lps(\eta)$ (depicted in Figure \ref{fig:loop-decomp}) and $\hol_{\a}(\ell'') = 0$. Then one has
			\begin{equation}\label{eq:lem-scalar-hol-flat-loop}
				\hol_{\a}(\ell) = \hol_\a(\ell').
			\end{equation}
		\end{lemma}

		\begin{proof}
			First find representing loops $l,l',l'' \in \Lps_{+}(\eta)$ for $\ell,\ell',\ell''$ with orientations and common base point $i_{0}$, as depicted by Figure \ref{fig:loop-decomp}. Then decompose the loop $l'$ into two parts $\gamma,\gamma'$ and the loop $l''$ into two parts $\gamma'',\gamma'''$, where $\gamma''$ is $\gamma'$ in reverse order, as depicted in Figure \ref{fig:gammapp}.
			Then we can write
			\begin{equation}
				\Hol_{\a}(l) = \Pi(\gamma''')\Pi(\gamma), \quad \Hol_{\a}(l') = \Pi(\gamma')\Pi(\gamma), \quad \Hol_{\a}(l'') = \Pi(\gamma''')\Pi(\gamma'').
			\end{equation}
			and thus have
			\begin{equation}\label{eq:proof-scalar-hol-flat-loop-aux}
				\Hol_{\a}(l) = \Pi(\gamma''')\Pi(\gamma) = \Pi(\gamma''')\Pi(\gamma'')\Pi(\gamma')\Pi(\gamma) = \Hol_{\a}(l'') \Hol_{\a}(l'),
			\end{equation}
			since $\Pi(\gamma'') = \big(\Pi(\gamma')\big)^{-1}$. The assumption $\hol_{\a}(\ell'') = 0$ implies $\Hol_{\a}(l'') = \II$ by Proposition \ref{prop:scalar-hol-hol} and consequently $\Hol_{\a}(l) = \Hol_{\a}(l')$ by \eqref{eq:proof-scalar-hol-flat-loop-aux}. Turning to scalar holonomy yields the equation $\hol_{\a}(l) = \hol_{\a}(l')$  that does not depend on the chosen orientations anymore, by virtue of Proposition \ref{prop:scalar-hol-basepoint}, and therefore proves \eqref{eq:lem-scalar-hol-flat-loop}.
		\end{proof}

	\subsubsection{Extended Yang-Mills Energy}
		\begin{theorem}[extended Yang-Mills energy]\label{thm:eym}
			Let $\{\ell_{1},\dotsc,\ell_{h_{1}(\eta)}\} \in \Lps(\eta)$ denote a cycle basis of $\Hom_1(\eta)$. The function
			\begin{equation}
				\mc{YM}_{\o{e}} : \Omega \to \R_+, \qquad \mc{YM}_{\o{e}}(\omega) =
				\mc{YM}(\omega) + \sum_{k=1}^{h_{1}(\eta)} \hol_{\a}(\ell_{k})
			\end{equation}
			establies the equivalence of the relations
			\begin{equation}
				\omega \in \Omega \text{  is synchronizable} \quad\iff\quad \mc{YM}_{\o{e}}(\omega) = 0.
			\end{equation}
		\end{theorem}

		\begin{proof}
			If $\omega$ synchronizable, then $\hol_{\a}(\ell)= 0$ for all loops $\ell \in \mc{L}(\eta)$. Then $\mc{YM}(\omega)=0$ by Proposition \ref{prop:synch-flat} and consequently $\mc{YM}_{\o{e}}(\omega) = 0$.

			It remains to show that $\mc{YM}_{\o{e}}(\omega) = 0$ implies that $\omega$ is synchronizable. Thus, we assume that $\mc{YM}_{\o{e}}(\omega) = 0$ holds. Recall that the scalar holonomy is non-negative on every loop $\ell \in \mc{L}(\eta)$, such that both terms of the extended Yang-Mills energy are non-negative. Therefore, $\mc{YM}_{\rm{e}}= 0$ implies that $\omega$ is flat and $\hol_{\a}(\ell_k) = 0$ for all $k$. The claim is that $\hol_{\a}(\ell) = 0$ holds for all loops of the graph $\eta$.

			First, we remark that there are basis loops $\ell_{k_j}, k_{j} \in \{1,\dots,h_{1}(\eta)\} \}$ such that $[\ell] = \sum_{j} [\ell_{k_j}] = \left[\sum_{j} \ell_{k_{j}}\right]$, because the $\ell_k$ form a basis for $\Hom_1(\eta)$. Next observe that one can write
			\begin{equation}
			\ell = \sum_j \ell_{k_{j}} + \sum_{r} \delta_{r},
			\end{equation}
			where the $\delta_{r}$ are triangles. Now we distribute the triangles $\delta_{r}$ of the second sum into the first sum, by adding the triangles $\delta_{r}$ to the cycles $\ell_{k_{j}}$ until all triangles are absorbed into one of the cycles. This yields a new set of cycles, denoted by $\tilde{\ell}_{k_{j}}$, and
			\begin{equation}
				\ell = \sum_{j} \tilde{\ell}_{k_{j}}.
			\end{equation}
			The relation between $\tilde{\ell}_{k_{j}},\ell$ and $\ell_{k_{i}}$ is depicted in Figure \ref{fig:aux-loops}.

Now one has $[\tilde{\ell}_{k_{j}}]=[\ell_{k_{j}}]$ since one cycle results from the other by adding triangles. This implies $\hol_{\a}(\tilde{\ell}_{k_{j}}) = \hol_{\a}(\ell_{k_{j}}) = 0$, by Lemma \ref{lem:scalar-hol-homologous}. Applying repeatedly Lemma \ref{lem:scalar-hol-flat-loop} yields
			\begin{equation}
				\hol_{\a}(\ell) = \hol_{\a}\Big(\sum_j \tilde{\ell}_{k_{j}}\Big) = 0.
			\end{equation}
			This shows $\omega$ has trivial holonomy by Proposition \ref{prop:scalar-hol-hol}  and thus is synchronizable by Theorem \ref{thm:holonomy-obstruction}.
		\end{proof}
		\begin{figure}
			\includegraphics[width=\textwidth]{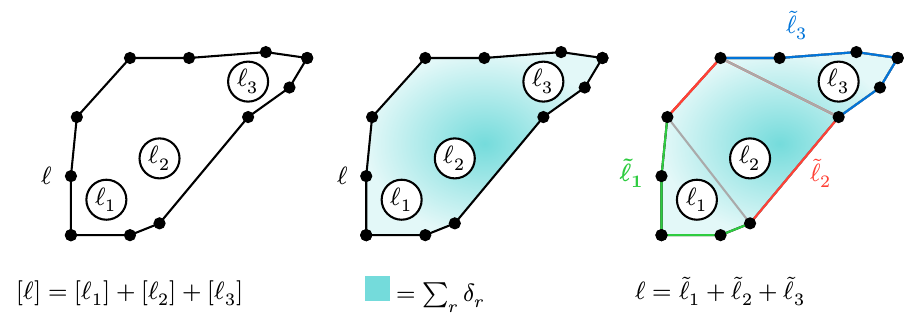}
			\caption{Illustration of the proof of Theorem \ref{thm:eym}: the relationship between the basis cycles $\ell_{k_{j}}$, the main cycle $\ell$ and the auxiliary cycles $\tilde{\ell}_{k_{j}}$.}\label{fig:aux-loops}
		\end{figure}

%% file: tex2/general.tex
\section{From Smooth Bundles to Discrete Ones}\label{sec:app}
\newcommand{\g}{\mathfrak{g}}

In this section we clarify how the smooth theory of vector bundles \cite{hamiltonMathematicalGaugeTheory2017,weilerCoordinateIndependentConvolutional2021,gerkenGeometricDeepLearning2021} relates to the discrete framework presented in Sections \ref{sec:bundles} and \ref{sec:hks}.

Consider the situation depicted by Figure \ref{fig:embed-graph}: 
$M$ is be a smooth, compact manifold without boundary. A given graph $\eta$ is embedded into $M$, with embedded vertices $p_i \in M, i \in \nu$ and connected by smooth arcs $\gamma_{ij} : [0,1] \to M$ that represent the edges of the embedded graph. We denote the embedding by $\iota : \eta \to M$.
Assume further that $M$ is equipped with a Riemannian metric $\g$. We obtain a positive weight function $\kappa : \nu \times \nu \to \R_{\geq 0}$ on $\eta$, defined by
\begin{equation}\label{eq:metric-weights}
	\kappa_{ij} = \o{len}_{\g}(\gamma_{ij})^{-2} \in \R_{\geq 0},
	\quad \forall i \sim j \in \veps,
\end{equation}
where $\o{len}_{\g}(\gamma_{ij})$ is the length of the path $\gamma_{ij}$ as measured by the metric $\g$\footnote{
	The inverse squared distance is chosen here, so as to comply with the result in Proposition \ref{prop:smooth-L}, but different choices are conceivable at this point.
}.
Thus, we consider the edge weights $\kappa$ as discrete representatives of the metric $\g$ on the graph $\eta$.

\begin{remark}[graph embeddings exist]\label{rem:graph-embedding-exists}
	It is generally possible to find such embeddings for any given graph. In fact, for any graph $\eta$, there is a manifold $M$ \textbf{of dimension 2}, such that $\eta$ can be embedded into $M$ \cite{grossTopologicalGraphTheory1987}.
\end{remark}

\begin{definition}[weighted graphs from embeddings]
	Let $\eta$ be a given graph and $\iota : \eta \to M$ a graph embedding, where $M$ is a Riemannian manifold with metric $\g$. Then $\iota$ induces a symmetric positive weight function $\kappa$ on $\eta$ via Equation \eqref{eq:metric-weights}, denoted with abuse by 
\begin{equation}\label{eq:def-kappa}
\kappa \coloneqq \iota^{*} \g.
\end{equation}
\end{definition}
We point out that the literal mathematical interpretation of \eqref{eq:def-kappa} does \textit{not} make sense. The function $\iota^{*}\g$ is not a Riemannian metric on the zero dimensional manifold $\nu$, but a positive weight function on the graph $\eta$. Yet we stick to the established pullback notation \eqref{eq:def-kappa} because it clearly indicates the role of the weight function $\kappa$.

\begin{figure}[t]
	\includegraphics[width=0.9\textwidth]{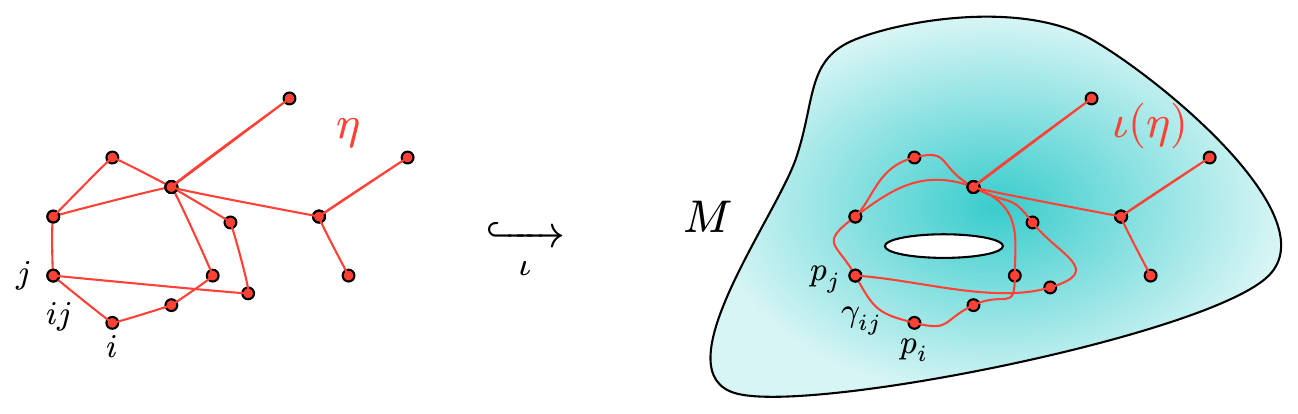}
	\caption{A smooth Riemannian manifold and an embedded graph $\eta$.}\label{fig:embed-graph}
\end{figure}

\begin{subsection}{Discrete Bundles from Smooth Bundles}
	We now describe the relationship between smooth associated vector bundles $E \to M$ and the discrete bundle sections $\Gamma$ introduced by Definition \ref{def:discrete-bundles}. This is a crucial step for connecting our discrete approach and the machinery of smooth vector bundles.
	
	The following result shows that discrete vector bundles arise indeed from smooth associated vector bundles through a pullback mapping.

	\begin{proposition}[discrete pullback bundles]\label{prop:discrete-pullback-bundles}
		Let $P \to M$ be a principal $\SO(d)$ bundle over $M$ and $E \coloneqq P \times_{\rho} \R^d$ the associated vector bundle for the representation $\rho$ of \, $\SO(d)$ on $\R^d$, under which a $\SO(d)$-rotation matrix acts on an $\R^d$-vector by matrix-vector multiplication. We denote by $E_p$ the fiber of $E$ at the point $p$.
		Assume that the graph $\eta$ is embedded into $M$, with embedding $\iota : \eta \hookrightarrow M$.
		Then one has an isomorphism
		\begin{equation}\label{eq:iota-star}
			\iota^{*}E \coloneqq \mathop{\dot{\bigcup}}_{i \in \nu} E_{p_i} \cong \nu \times Y,
		\end{equation}
		where $Y \coloneqq (\SO(d) \times \R^{d}) / \SO(d)$. The set $\iota^{*}E$ is a vector bundle over the zero dimensional base manifold $\nu$ and its smooth sections are given by
		\begin{equation}
			\Gamma(\iota^{*}E) \cong Y^{N}, \qquad N = |\nu|.
		\end{equation}
	\end{proposition}

	\begin{proof}
		We note that the fibers are given by
		\begin{equation}
			E_{p_{i}} = (P_{p_{i}} \times \R^d) / \SO(d),
		\end{equation}
		where $P_{p_{i}}$ is the fiber of the bundle $P$ at $p_i$.
		Now choose for every $i \in \nu$ an isomorphism $P_i \cong \SO(d)$. This is possible for all the $P_{p_i}$ \textit{simultaneously}, since the set $\nu$ is \textit{finite}. These choices induce an isomorphism $E_{p_{i}} \cong Y$ for all $i \in \nu$. Collecting these isomorphisms yields
		\begin{equation}
			\mathop{\dot{\bigcup}}_{i \in \nu} E_{p_{i}} \cong \nu \times Y.
		\end{equation}
		The sections of the bundle $\iota^{*}E$ can then be calculated as
		\begin{equation}
			\Gamma(\iota^{*}E) \cong
			\left\{ \phi : \nu \to \nu \times Y \mid \phi \in C^{\infty}, \; \exists v : \nu \to Y \text{ s.t. } \phi = \mathrm{id} \times v \right\}.
		\end{equation}
		Since $\nu$ is zero-dimensional, the smoothness condition $\phi \in C^{\infty}$ is void and we obtain
		\begin{equation}
			\Gamma(\iota^{*}E) \cong \left\{ v : \nu \to Y \right\} = Y^N, \quad N = |\nu|.
		\end{equation}
	\end{proof}

	\begin{corollary}
		Let $\iota : \eta \hookrightarrow M$ be an embedded graph and $E \to M$ be a $\SO(d)$-associated vector bundle. Then there is an induced mapping $\iota^{*} : \Gamma(E) \to \Gamma$ by pulling back sections along $\iota$, after a choice of trivialization $\iota^{*}E \cong \nu \times Y$.
	\end{corollary}

	\begin{remark}[geometric situation]
	We explicate Proposition \ref{prop:discrete-pullback-bundles} from the geometric viewpoint. 
		We regard the vertex set $\nu$ as a zero dimensional manifold and the embedding $\iota$ as a smooth embedding of manifolds $\nu \hookrightarrow M$ (where the smoothness condition here is void since $\nu$ is zero-dimensional).
		Accordingly, we define a vector bundle over $\nu$ through pulling back $E$ along $\iota$, resulting in the vector bundle $\iota^{*}E \to \nu$ over the base space $\nu$.
	\end{remark}
	We showed that the sections of the pullback bundle $\iota^*E$ are isomorphic to $Y^N$, after a choice of trivialization. On the other hand, $Y^N = \Gamma$ also represents the \textit{discrete vector bundle} of our choice for gauge invariant data modelling, see \eqref{eq:gamma}. This shows that discrete vector bundles can be obtained from smooth ones $E$ in terms of the sections of $\iota^*E$. 

	\begin{remark}[topological obstructions]
		If the set $\nu$ would be a manifold of dimension $\geq 1$ and we were interested in smoothly varying isomorphisms $P_{p_i} \cong \SO(d)$, then we could encounter topological obstructions.
		The absence of obstructions relies on the fact that the set $\nu$ is discrete such that every bundle over it is topologically trivial -- cf.~\eqref{eq:iota-star}, whereas bundles over $M$ need not be trivial. 
		
		We refer to Section \ref{sec:ym}, however, where we show that the limitation of bundle triviality over graphs can be overcome, by working with suitably chosen voltage graphs. While the sections of vector bundles over $\nu$ cannot encode topological information about the graph $\eta$ or the manifold $M$, we show that the space of graph voltages $\mc{A}$ \textit{depends} on the topology of the underlying graph, see Theorems \ref{thm:holonomy-obstruction} and \ref{thm:eym}. Furthermore, if the embedding $\iota: \eta \hookrightarrow M$ (cf.~Figure \ref{fig:embed-graph}) is such that the topology of $M$ is compatible with the topology of the graph $\eta$, then the space of graph voltages $\mc{A}$ over $\eta$ indeed reflects some topological properties of $M$.
	\end{remark}
\end{subsection}

\begin{subsection}{Graph Voltages from Parallel Transport}
	We further detail the connection between graph voltages and parallel transport operators on associated vector bundles. 

	\begin{figure}[t]
		\includegraphics[width=0.55\textwidth]{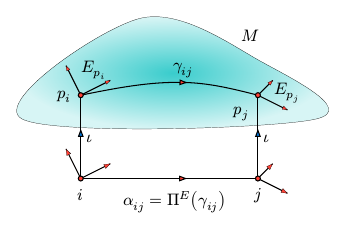}
		\caption{Relationship between parallel transport and graph voltages.}
		\label{fig:parallel-voltage}
	\end{figure}

	We first recall few facts about parallel transport on vector bundles \cite{baumEichfeldtheorieEinfuehrungDifferentialgeometrie2014,hamiltonMathematicalGaugeTheory2017}. Due to the $\SO(d)$-structure of the bundle $E$, we also have a \textit{bundle metric} on $E$, which is a scalar product $\la \cdot , \cdot \ra_{p}$ on every fiber $E_{p}$ with associated norm $\n{\cdot}_p$, for every $p \in M$. A bundle metric promotes the fibers of $E$ to metric spaces. A \textit{parallel transport operator} along a given curve $\gamma : [0,1] \to M$ on the associated vector bundle $E$ is a linear operator $\Pi^{E}(\gamma) : E_{\gamma(0)} \to E_{\gamma(1)}$ with the following properties.
	\begin{enumerate}
		\item $\Pi^{E}(\gamma)$ does not depend on the parametrization of the curve $\gamma$.

		\item $\Pi^{E}(\gamma)$ is an \textit{orientation preserving isometry}.

		\item If $\gamma^{-1}$ denotes the curve in reverse orientation, then $\Pi^{E}(\gamma^{-1}) : E_{\gamma(1)} \to E_{\gamma(0)}$ is the inverse linear mapping to $\Pi^{E}(\gamma)$, that is $\Pi^{E}(\gamma^{-1}) = \Pi^{E}(\gamma)^{-1}$.
	\end{enumerate}
	That is, parallel transport is an identification of fibers of a given bundle along a given curve $\gamma$, attached to the endpoints of the curve; see Figure \ref{fig:parallel-voltage}. Parallel transport operators of this kind exist and can be constructed in terms of the \textit{horizontal lifting construction} on the principal bundle $P$, see \cite[Sections 5.8, 5.9]{hamiltonMathematicalGaugeTheory2017}.

	\begin{proposition}\label{prop:gaute-to-volt}
		Let $P$ be an principal $\SO(d)$ bundle over $M$ with embedded graph $\iota : \eta \to M$, and let $E = P \times_{\rho} \R^{d}$ be the associated vector bundle with respect to the standard representation of $\SO(d)$. Furthermore, let $\Pi^E$ denote a parallel transport operator on the bundle $E$ and $\zeta: \iota^* E \cong \nu \times Y$ a choice of trivialization. Then $\Pi^E$ induces a gauge invariant graph voltage $\omega=[\mbf{1}_\Xi,\alpha]$ via the commutative diagram
		\begin{equation}
			\begin{tikzcd}
				& E_{p_i} \ar[dl,"\zeta_i"]\ar[r,"\Pi^E(\gamma_{ij})"] & E_{p_j} \ar[dr,"\zeta_j"] & \\
				Y \ar[dr,"\zeta_{0}"] & & & Y \ar[dl,"\zeta_{0}"] \\
			  &	\R^d \ar[r,"\alpha_{ij}"] & \R^d &
			\end{tikzcd}, 
		\end{equation}
		where $\zeta_0 : Y \to \R^d, [g,x] \mapsto gx$ is the canonical isomorphism from $Y$ to $\R^d$ and $\mbf{1}_\Xi$ is the identity element of of the group $\Xi = \SO(d)^N$. 
		The gauge invariant voltage is also denoted by 
		\begin{equation}\label{eq:omega-via-E}
			\omega \coloneq \iota^{*}\Pi^{E}.
		\end{equation}
	\end{proposition}

	\begin{proof}
		For all nodes $i \in \nu$, we can find isomorphisms $\zeta_i : E_{p_i} \to Y$. Then the linear isomorphisms $\Pi^E(\gamma_{ij}) : E_{p_i} \to E_{p_j}$ define linear operators $Y \to Y$.
		In conjunction with $\zeta_0$, these isomorphisms induce a linear mapping $\alpha_{ij} : \R^d \to \R^d$ representing $\Pi^E(\gamma_{ij})$. 
		Note that the maps $\zeta_{0}$ and $\zeta_i$ are linear orientation-preserving isometries, as is $\Pi^{E}(\gamma_{ij})$. This implies that $\alpha_{ij}$ is an orientation-preserving isometry, i.e. $\alpha_{ij} \in \SO(d)$.
		In addition, for the reverse paths $\gamma_{ji} = \gamma_{ij}^{-1}$, one has $\Pi^E(\gamma_{ji}) = \Pi^E(\gamma_{ij})^{-1}$ and consequently
		\begin{equation}
			\alpha_{ji} = \alpha_{ij}^{-1}.
		\end{equation}

		Isometries $\alpha_{ij}$ are represented by the graph voltage $\alpha \in \mc{A}$ and we can construct a corresponding a gauge invariant voltage by setting $\omega = [\mbf{1}_\Xi,\alpha]$.
	\end{proof}

	\begin{remark}[relationship between gauge fields and parallel transport]
		The common way to parametrize parallel transport operators in gauge theory is to specify \textit{gauge fields}, which is the term in physics used for the notion of \textit{connections} from differential geometry \cite{hamiltonMathematicalGaugeTheory2017}. A gauge field $B$ is a Lie algebra valued one-form, with additional properties. The relationship to parallel transport operators $\Pi^{E}$ is given by \cite{Gattringer:2010zz}
		\begin{equation}
			\Pi^{E}(\gamma) = \mc{P} \exp \left(\int_{\gamma} B \right),
		\end{equation}
		where $\mc{P} \exp$ denotes the \textit{path ordered exponential}, $\int_{\gamma} B$ integrates the pullback of the gauge field along $\gamma$, and the full expression is the solution to the parallel transport ODE associated to $B$. This fact is explained in detail in \cite[Section 5.10]{hamiltonMathematicalGaugeTheory2017}.
	\end{remark}

\end{subsection}

\subsection{Discrete Laplacians from Covariant Derivatives}
	We extend the relationship between the discrete and continuous theories by stating an approximation result that connects respective energy functions. We assume in this subsection that we are given an embedded graph $\iota : \eta \to M$, an $\SO(d)$-associated bundle $E \to M$ with parallel transport operator $\Pi^{E}$ and a Riemannian metric $\g$ on $M$. 

In order to state our result, we recall few more notions that are established in the smooth case. Recall the bundle metric $\la\cdot,\cdot\ra_p$ on $E$. This bundle metric also defines an inner product on the $E$-valued one-forms $\Omega^{1}(M,E)$, together with the Riemannian metric $\g$, through
\begin{equation}\label{eq:sigma-times-0}
\la \sigma,\sigma'\ra_{p} 
:= \sum_{a,b=1}^{n} \g^{a b}(p)\la \sigma_{p}(\partial_{a}),\sigma'_{p}(\partial_{b})\ra_{p},\quad
\sigma, \sigma' \in \Omega^{1}(M,E),
\end{equation}
with $\sigma_{p}, \sigma'_{p}\in\mathrm{Hom}(T_{p}M, E_{p})$ and basis vectors $\{\partial_{a}\}_{a=1}^{n}\subset T_{p}M$ and $\g^{ab}(p)$ the inverse matrix to $\g_{ab}(p)$, which is the representing matrix of the metric $\g$ in the basis $\{\partial_{a}\}_{a=1}^{n}$.

This expression can be simplified in terms of \textit{local orthonormal frame fields}. A local frame field at a point $p \in M$ is a set of vector fields $e_{A}, A \in \{1,\dots, n\}$, all defined on an open neighborhood $U$ of $p$, with component functions $e^{a}_{A},\, a\in\{1,\dotsc,n\}$ such that
\begin{equation}
	\sum_{a,b = 1}^{n} \g_{ab}(p) e_{A}^{a}(p)  e^{b}_{B}(p) = \delta_{AB},\quad \forall A, B\in\{1,\dotsc,n\},\quad\forall p\in U.
\end{equation}
For any given point $p$, we can find a small neighborhood around it, such that a local orthonormal frame field exists. The vectors $e_{A}$ form an orthonormal frame for the bilinear form $\g_{ab}(p)$ at every $p \in U$. We can infer, that the associated Gramian matrix is the inverse bilinear form 
\begin{equation}\label{eq:inverse-frame}
	\sum_{A}^{n}  e^{a}_{A}(p) e^{b}_{A}(p) = \g^{ab}(p), \quad \forall a,b \in \{1,\dots,n\}, p \in U.
\end{equation}

This enables to rewrite \eqref{eq:sigma-times-0} in the form
\begin{subequations}\label{eq:sigma-times}
\begin{align}
\la\sigma,\sigma'\ra_{p}
&=
\sum_{a,b=1}^{n} \g^{ab}(p) \la \sigma_{p}(\partial_{a}), \sigma_{p}(\partial_{b})\ra_p \\ 
\overset{\eqref{eq:inverse-frame}}&{=}
\sum_{a,b,A}^{n} e^a_{A} e^b_{A} \la \sigma_{p}(\partial_a), \sigma_{p}(\partial_{b}) \ra_p \\
&=\sum_{a,b,A=1}^{n} \la \sigma_{p}(e^{a}_{A}\partial_a) , \sigma_{p}(e^{b}_{A} \partial_b) \ra_p \\
&=\sum_{A=1}^{n} \Big\la \sigma_{p}\Big(\sum_{a=1}^n e^{a}_{A}\partial_a\Big) , \sigma_{p}\Big(\sum_{b=1}^n e^{b}_{A} \partial_b\Big) \Big\ra_p \\
&=\sum_{A=1}^{n} \la\sigma_{p}(e_A),\sigma'_{p}(e_A)\ra_{p},\quad\forall p\in U.
\end{align}
\end{subequations}
where the third step follows by linearity of $\sigma_p$ and the last step from the definition of the component functions $e_A = \sum_{a=1}^n e^a_A \partial_A$.

Given a parallel transport operator $\Pi^{E}$ on $E$, one defines a \textit{covariant derivative for bundle sections} as follows \cite{hamiltonMathematicalGaugeTheory2017}. In order to compute the derivative $\nabla_{v}\Phi(p)$ for $v \in T_{p}M$ and $\Phi \in \Gamma(E)$, first take any curve $\gamma : [0,\veps) \to M$ with $\gamma(0) = p, \dot{\gamma}(0) = v$. Then define the function
	\begin{equation}
		S_{\Phi} : [0,\veps) \to E_{p}, \quad S_{\Phi}(t) = \Pi^{E}(\gamma|_{[0,t)})^{-1} \Phi(\gamma(t)) \in E_{p},
	\end{equation}
	where $\gamma|_{[0,t)}$ is the curve $\gamma$ restricted to the interval $[0,t)$. Then define
	\begin{equation}\label{eq:cov-der}
		\nabla_{v} \Phi(p) \coloneqq \lim_{t\to 0^{+}} \frac{S_{\Phi}(t) - S_{\Phi}(0)}{t} \in E_{p}.
	\end{equation}
	This defines a distinguished mapping
	\begin{equation}
		\nabla : \Gamma(E) \to \Omega^{1}(M,E), \quad \Phi \mapsto \nabla \Phi.
	\end{equation}
	We associate the \textit{Laplacian} $\Delta_{E}$ with the covariant derivative $\nabla$ by 
\begin{equation}
	\Delta_{E} : \Gamma(E) \to \Gamma(E),\qquad
	\Delta_{E} = \nabla^{*} \circ \nabla, 
\end{equation}
where $\nabla^{*} : \Omega^{1}(M,E) \to \Gamma(E)$ is the adjoint operator to $\nabla$ with respect to scalar products
	\begin{equation}
		\lla \Phi , \Phi' \rra_{E} \coloneq \int_{M} \la \Phi(p),\Phi'(p)\ra_{p} \sqrt{\det \g(p)} \dd^{n}p, \quad \Phi,\Phi' \in \Gamma(E)
	\end{equation}
	and
	\begin{equation}
		\lla \sigma,\sigma' \rra_{E} \coloneq \int_{M} \la \sigma(p),\sigma'(p) \ra _{p} \sqrt{\det \g(p)} \dd^{n}p, \quad \sigma,\sigma' \in \Omega^{1}(M,E)
	\end{equation}
	respectively. The adjoint $\nabla^*$ is defined by
	\begin{equation}
		\lla \nabla \Phi , \sigma \rra_{E} = \lla \Phi , \nabla^{*} \sigma \rra_{E}, \quad \forall \Phi \in \Gamma(E) ,\quad  \sigma \in \Omega^{1}(M,E)
	\end{equation}
	implying
	\begin{equation}
		\lla \Phi,\Delta_{E} \Phi \rra_{E} = \lla \Phi, \nabla^{*} \nabla \Phi \rra_{E} = \lla \nabla \Phi , \nabla \Phi \rra_{E} \geq 0
	\end{equation}
	assuming the manifold $M$ is compact and without boundary.

	\begin{figure}
		\includegraphics[width=0.45\textwidth]{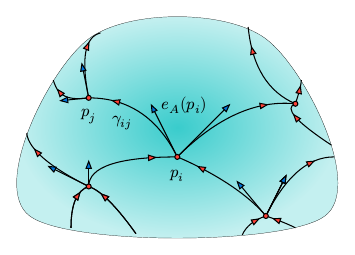}
		\caption{
		Illustration of the conditions for Proposition \ref{prop:smooth-L}. Blue arrows indicate the local frame.
		}
		\label{fig:frame}
	\end{figure}

	\begin{proposition}[discrete and smooth Dirichlet energy \cite{Gattringer:2010zz}]
	\label{prop:smooth-L}
		Let $\iota : \eta \to M$ be an embedded graph and $E \to M$ an $\SO(d)$-associated vector bundle with parallel transport operator $\Pi^{E}$. Assume that the manifold $M$ is endowed with the Riemannian metric $\g$ and that the following criteria are fulfilled.
		\begin{enumerate}
			\item $\eta$ is is a \textit{regular} graph with equal degree $2\dim M$ at every node $i \in \nu$, oriented such that every node has the equal number of in- and out-edges.

			\item $M$ is covered by open sets $U_i, i \in \nu$, such that $\cup_{i \in \nu} U_i = M$ and $\iota(i) = p_i \in U_i$ and $\iota(j) = p_j \not\in U_i$ for all $j \neq i$.

			\item Let $\{\beta_i\}_{i=1}^{n}$ denote a partition of unity with respect to this covering. We have the volume normalization condition
			\begin{equation}\label{eq:norm}
				\int_{U_i} \beta_{i} \sqrt{\det \g} = \o{vol}_{\g}(U_{i}) = 1,\quad\forall i\in\nu.
			\end{equation}

			\item Let $e_A$ denote a local orthonormal frame for the metric $\g$, defined in all points $p_i, i \in \nu$. For all $i \in \nu$ and $A \in \{1,\dots,n\}$, there is a unique neighbor $j \in \nu$ and $j > i$ with path $\gamma_{ij} : p_i \to p_j$ such that $\dot \gamma_{ij}(0) = e_A(p_i)$; see Figure \ref{fig:frame} for an illustration.
		\end{enumerate}
		Then, for any section $\Phi \in \Gamma(E)$, one has
		\begin{equation}
		    \lla \Phi, \Delta_{E} \Phi \rra_{E} = \lla \phi, \hgen_{\chi} \phi \rra_Z + \mc{O}(C) + \mc{O}(C'),
		\end{equation}
		where $\lla \cdot,\cdot \rra_Z$ was defined in \eqref{eq:def-bracket-z}, \eqref{eq:def-zi} here with $\kappa_{ij}=\lambda_{ij}^{-2}, \lambda_{ij} = \o{len}_{\g}(\gamma_{ij})$, and (recall \eqref{eq:def-kappa} and \eqref{eq:omega-via-E})
		\begin{equation}
			\chi = (\eta,\iota^{*}\g,\iota^{*}\Pi^{E}), \quad \phi = \iota^{*} \Phi,
		\end{equation}
		after a choice of trivialization $\zeta: \iota^*E \to \nu \times Y$. The constants $C,C'$ are given by
		\begin{equation}
			C = \max_{i \in \nu} \sup_{p \in U_i} \left| \n{\nabla \Phi(p)}_{p}^2 - \n{\nabla \Phi(p_i)}_{p_{i}}^2 \right|, \qquad 
			C' = \max_{\substack{i,j \in \nu \\ i \sim j}} \lambda_{ij}^{2}.
		\end{equation}
	\end{proposition}

	\begin{remark}[proper discretization of $M$]
		Assumption (i) specifies the uniform degree of $\eta$ that matches the dimension of $M$, for a suitable discretization.

		Assumption (iii) is a volume normalization condition that can be easily satisfied by a conformal rescaling of the metric with a suitable function.

		Assumption (iv) entails that the paths $\gamma_{ij}$ must be compatible with the metric $\g$ in the sense that they must be generated by a local frame vector field.

		In summary, the stated assumptions enforce compatibility relations between the embedded graph $\iota(\eta)$ and $M$ in the metric and topological sense.
	\end{remark}

	\begin{proof}[Proof of Proposition \ref{prop:smooth-L}]
		We first evaluate the integral over the manifold
		\begin{equation}\label{eq:proof-bundle-metric-aux}
			\begin{aligned}
				\lla \Phi, \Delta_{E} \Phi \rra_{E}
				&=
				\int_M \n{\nabla \Phi(p)}_p^2 \sqrt{\det \g(p)} \; d^{n}p
				\\
				&=
				\sum_{i \in \nu} \int_{U_i} \beta_i(p) \n{\nabla \Phi(p)}_p^2 \sqrt{\det \g(p)} \; d^{n}p
				\\
				&=
				\sum_{i \in \nu} \n{\nabla \Phi(p_i)}_{p_{i}}^2 \underbrace{\int_{U_i} \beta_i(p) \sqrt{\det \g(p)} \; d^{n} p}_{\overset{\eqref{eq:norm}}{=}1} +  \; \mc{O}(C) \\
				&=
				\sum_{i \in \nu} \n{\nabla \Phi(p_i)}_{p_{i}}^2 +  \mc{O}(C),
			\end{aligned}
		\end{equation}
		where
		\begin{equation}
			C = \max_{i \in \nu} \sup_{p \in U_i} \left| \n{\nabla \Phi(p)}_{p}^2 - \n{\nabla \Phi(p_i)}_{p_{i}}^2 \right|.
		\end{equation}
		To proceed, we employ the definition of the bundle metric
		\begin{equation}\label{eq:proof-nabla-Phi-pi-aux}
			\n{\nabla \Phi(p_i)}_{p_{i}}^2
			=
			\la \nabla \Phi(p_i), \nabla \Phi(p_i) \ra_{p_{i}}
			\overset{\eqref{eq:sigma-times}}{=}
			\sum_{A=1}^{n} \la \nabla_A \Phi(p_i) , \nabla_A \Phi(p_i) \ra_{p_{i}},
			\quad \nabla_{A} \coloneq \nabla_{e_{A}}
		\end{equation}
		and approximate the covariant derivatives, as defined in \eqref{eq:cov-der}, by choosing an explicit curve $\gamma$ and a finite time $t > 0$. For some given $A$ and $p_{i}$, there is a unique pair of indices $i,j$ such that $\gamma_{ij}(0) = p_{i}$, $\gamma_{ij}(\lambda_{ij}) = p_{j}$ and $\dot{\gamma}_{ij}(0) = e_{A}(p_{i})$, by assumption, where $\lambda_{ij} = \o{len}(\gamma_{ij})$.  
		
		Approximating the derivative then yields
		\begin{equation}
			\nabla_A \Phi (p_i) =
			\frac{\Pi^{E}(\gamma_{ij})^{-1}\Phi(p_{j}) - \Phi(p_i)}{\lambda_{ij}}
			+ \mc{O}(\lambda_{ij}).
		\end{equation}
		Now we can choose a trivialization $\iota^*E \cong \nu \times Y$ which generates isometries $E_{p_i} \cong \R^d$, under which $\Phi(p_i) \in E_{p_i}$ is sent to some $x_i \in \R^d$ for $i \in \nu$. Furthermore, the operator $\Pi^E(\gamma_{ij})$ is identified with a matrix $\alpha_{ij} \in \SO(d)$ under these isometries, which allows us to evaluate    the squared norm of the first term of the above derivative approximation to
		\begin{equation}
			\left\|
			\frac{\Pi^{E}(\gamma_{ij})^{-1} \Phi(p_{j}) - \Phi(p_i)}{\lambda_{ij}}
			\right\|^{2}_{p_i}
			=
			\frac{1}{\lambda_{ij}^{2}}
			\left\|
			\alpha_{ij}x_{i} - x_{j}
			\right\|^{2}.
		\end{equation}
		Substituting into \eqref{eq:proof-nabla-Phi-pi-aux} gives
		\begin{equation}
			\n{\nabla \Phi(p_i)}_{p_{i}}^2
			= \sum_{j\in \nu, j > i}\frac{\n{\alpha_{ij} x_j - x_i}^2}{\lambda_{ij}^2} + \mc{O}(\lambda_{ij}^{2})
		\end{equation}
		and consequently
		\begin{equation}
			\sum_{i \in \nu} \n{\nabla \Phi(p_i)}_{p_{i}}^2 +  \mc{O}(C) =
			\sum_{i \in \nu} \sum_{j > i}
			\frac{\n{\alpha_{ij} x_j - x_i}^2}{\lambda_{ij}^2} + \mc{O}(C') +   \mc{O}(C)
		\end{equation}
		where
		\begin{equation}
			C' = \max_{\substack{i,j \in \nu \\ i \sim j}} \lambda_{ij}^{2}.
		\end{equation}
		Substituting into \eqref{eq:proof-bundle-metric-aux} yields the desired result
		\begin{equation}
			\begin{aligned}
			\lla \Phi , \Delta_{E} \Phi \rra_{E}
			&=
			\sum_{i \in \nu} \sum_{j > i}\frac{\n{\alpha_{ij} x_j - x_i}^2}{\lambda_{ij}^2} + \mc{O}(C') + \mc{O}(C) \\
			\overset{\text{Lemma }\ref{lem:dip}}&{=}
			\lla \phi, \hgen_{\vg} \phi \rra_{Z} + \mc{O}(C') + \mc{O}(C),
			\end{aligned}
		\end{equation}
		with $\phi = [\xi,x] = \iota^{*} \Phi$ and $\omega = [\xi,\alpha] = \iota^{*} \Pi^{E}$ and $\kappa = \iota^{*}\g$.
	\end{proof}

%% file: ym.bbl
\newcommand{\etalchar}[1]{$^{#1}$}
\begin{thebibliography}{BnBDD{\etalchar{+}}21}

\bibitem[AB83]{atiyah1983yang}
Michael~Francis Atiyah and Raoul Bott.
\newblock {The Yang-Mills Equations over Riemann Surfaces}.
\newblock {\em Philosophical Transactions of the Royal Society of London.
  Series A, Mathematical and Physical Sciences}, 308(1505):523--615, 1983.

\bibitem[AGLM93]{Alvarez:1993aa}
L.~Alvarez, F.~Guichard, P.~L. Lions, and J.~M. Morel.
\newblock {Axioms and Fundamental Equations of Image Processing}.
\newblock {\em Arch. Rat. Mech. Anal.}, 123:199--257, 1993.

\bibitem[AK06]{Aubert:2006aa}
G.~Aubert and P.~Kornprobst.
\newblock {\em {Mathematical Problems in Image Processing: Partial Differential
  Equations and the Calculus of Variations}}.
\newblock Springer, 2nd edition, 2006.

\bibitem[AM00]{Aji:2000aa}
S.~M. Aji and R.~J. McEliece.
\newblock {The Generalized Distributive Law}.
\newblock {\em IEEE Trans. Information Theory}, 46(2):325--343, 2000.

\bibitem[AMM{\etalchar{+}}18]{Abbe:2018ab}
E.~Abbe, L.~Massoulie, A.~Montanari, A.~Sly, and N.~Srivastava.
\newblock {Group Sychronization on Grids}.
\newblock {\em Math. Stat. Learn.}, 1:227--256, 2018.

\bibitem[AMOS19]{Arridge:2019wu}
S.~Arridge, P.~Maass, O.~\"Oktem, and C.-B. Sch\"{o}nlieb.
\newblock {Solving Inverse Problems Using Data-Driven Models}.
\newblock {\em Acta Numerica}, 28:1--174, 2019.

\bibitem[AMS23]{aronssonGeometricalAspectsLattice2023a}
Jimmy Aronsson, David~I. M{\"u}ller, and Daniel Schuh.
\newblock Geometrical aspects of lattice gauge equivariant convolutional neural
  networks.
\newblock (arXiv:2303.11448), March 2023.

\bibitem[Aro22]{aronssonHomogeneousVectorBundles2022}
Jimmy Aronsson.
\newblock Homogeneous vector bundles and {{G-equivariant}} convolutional neural
  networks.
\newblock {\em Sampling Theory, Signal Processing, and Data Analysis},
  20(2):10, July 2022.

\bibitem[Avr15]{Avramidi:2015aa}
I.~G. Avramidi.
\newblock {\em {Heat Kernel Method and its Applications}}.
\newblock Birkh\"{a}user, 2015.

\bibitem[Bae96]{baezSpinNetworkStates1996}
John~C. Baez.
\newblock Spin {{Network States}} in {{Gauge Theory}}.
\newblock {\em Advances in Mathematics}, 117(2):253--272, February 1996.

\bibitem[Bat10]{batardCliffordBundlesCommon2010}
Thomas Batard.
\newblock Clifford {{Bundles}}: {{A Common Framework}} for {{Image}}, {{ Vector
  Field}}, and {{Orthonormal Frame Field Regularization}}.
\newblock {\em SIAM Journal on Imaging Sciences}, 3(3):670--701, January 2010.

\bibitem[Bat11]{batardHeatEquationsVector2011}
Thomas Batard.
\newblock Heat {{Equations}} on {{Vector Bundles}}---{{Application}} to {{
  Color Image Regularization}}.
\newblock {\em Journal of Mathematical Imaging and Vision}, 41(1):59--85,
  September 2011.

\bibitem[Bat25]{batardHighOrderDIPVBTVImage2025}
Thomas Batard.
\newblock High-{{Order DIP-VBTV}}: {{An Image Restoration Model Combining}} a
  {{Deep Image Prior}} and a {{High-Order Total Variation}} on {{ Vector
  Bundles}}.
\newblock {\em Journal of Mathematical Imaging and Vision}, 67(2):24, April
  2025.

\bibitem[Bau14]{baumEichfeldtheorieEinfuehrungDifferentialgeometrie2014}
Helga Baum.
\newblock {\em {Eichfeldtheorie: Eine Einf{\"u}hrung in die
  Differentialgeometrie auf Faserb{\"u}ndeln}}.
\newblock Springer Berlin Heidelberg, Berlin, Heidelberg, 2014.

\bibitem[BB14]{batardCovariantDerivativesTheir2014}
Thomas Batard and Marcelo Bertalm{\'\i}o.
\newblock On {{Covariant Derivatives}} and {{Their Applications}} to {{Image
  Regularization}}.
\newblock {\em SIAM Journal on Imaging Sciences}, 7(4):2393--2422, January
  2014.

\bibitem[BBDB24]{bambergerBundleNeuralNetworks2024}
Jacob Bamberger, Federico Barbero, Xiaowen Dong, and Michael Bronstein.
\newblock Bundle {{Neural Networks}} for message diffusion on graphs.
\newblock (arXiv:2405.15540), May 2024.

\bibitem[BBL{\etalchar{+}}17]{bronstein2017geometric}
Michael~M Bronstein, Joan Bruna, Yann LeCun, Arthur Szlam, and Pierre
  Vandergheynst.
\newblock Geometric deep learning: going beyond euclidean data.
\newblock {\em IEEE Signal Processing Magazine}, 34(4):18--42, 2017.

\bibitem[BCI{\etalchar{+}}20]{Battiston:2020aa}
F.~Battiston, G.~Cencetti, I.~Iacopini, V.~Latora, M.~Lucas, A.~Patania, J.-G.
  Young, and G.~Petri.
\newblock {Networks Beyond Pairwise Interactions: Structures and Dynamics}.
\newblock {\em Physics Reports}, 874:1--92, 2020.

\bibitem[BDGC{\etalchar{+}}23]{bodnarNeuralSheafDiffusion2023}
Cristian Bodnar, Francesco Di~Giovanni, Benjamin~Paul Chamberlain, Pietro
  Li{\`o}, and Michael~M. Bronstein.
\newblock Neural {{Sheaf Diffusion}}: {{A Topological Perspective}} on {{
  Heterophily}} and {{Oversmoothing}} in {{GNNs}}.
\newblock (arXiv:2202.04579), January 2023.

\bibitem[Ber03]{Berger:2003aa}
M.~Berger.
\newblock {\em {A Panoramic View of Riemannian Geometry}}.
\newblock Springer, 2003.

\bibitem[BGV96]{Berline:1996aa}
N.~Berline, E.~Getzler, and M.~Vergne.
\newblock {\em {Heat Kernels and Dirac Operators}}.
\newblock Springer, 2nd edition, 1996.

\bibitem[BnBDD{\etalchar{+}}21]{Boguna:2021aa}
M.~Bogu\~{n}a, I.~Bonamassa, M.~De~Domenico, S.~Havlin, D.~Krioukov, and
  M.~{\'A}. Serrano.
\newblock {Network Geometry}.
\newblock {\em Nature Reviews Physics}, 3:114--135, 2021.

\bibitem[BS14]{batardClassGeneralizedLaplacians2014}
T.~Batard and N.~Sochen.
\newblock {A Class of Generalized Laplacians on Vector Bundles Devoted to
  Multi-Channel Image Processing}.
\newblock {\em Journal of Mathematical Imaging and Vision}, 48(3):517--543,
  March 2014.

\bibitem[BYF{\etalchar{+}}19]{budninskiyParallelTransportUnfolding2019}
Max Budninskiy, Gloria Yin, Leman Feng, Yiying Tong, and Mathieu Desbrun.
\newblock Parallel {{Transport Unfolding}}: {{A Connection-Based Manifold
  Learning Approach}}.
\newblock {\em SIAM Journal on Applied Algebra and Geometry}, 3(2):266--291,
  January 2019.

\bibitem[CGB23]{Carletti:2023aa}
T.~Carletti, L.~Giambagli, and G.~Bianconi.
\newblock {Global Topological Synchronization on Simplicial and Cell
  Complexes}.
\newblock {\em {Physical Review Letters}}, 130(187401), 2023.

\bibitem[Cha84]{chavelEigenvaluesRiemannianGeometry1984}
Isaac Chavel.
\newblock {\em Eigenvalues in {{Riemannian}} Geometry}.
\newblock Number 115 in Pure and Applied Mathematics. Academic Press, Orlando,
  1984.

\bibitem[Chu97]{Chung:1997aa}
F.R.K. Chung.
\newblock {\em {Spectral Graph Theory}}.
\newblock Amer. Math. Soc., 1997.

\bibitem[CMO{\etalchar{+}}24]{cloningerRandomWalksConductance2024}
Alexander Cloninger, Gal Mishne, Andreas Oslandsbotn, Sawyer~J. Robertson,
  Zhengchao Wan, and Yusu Wang.
\newblock Random {{Walks}}, {{Conductance}}, and {{Resistance}} for the {{
  Connection Graph Laplacian}}.
\newblock {\em SIAM Journal on Matrix Analysis and Applications},
  45(3):1541--1572, September 2024.

\bibitem[CRE{\etalchar{+}}21]{chamberlain2021beltrami}
Benjamin Chamberlain, James Rowbottom, Davide Eynard, Francesco Di~Giovanni,
  Xiaowen Dong, and Michael Bronstein.
\newblock Beltrami flow and neural diffusion on graphs.
\newblock {\em Advances in Neural Information Processing Systems},
  34:1594--1609, 2021.

\bibitem[CRG{\etalchar{+}}21]{chamberlain2021grand}
Ben Chamberlain, James Rowbottom, Maria~I Gorinova, Michael Bronstein, Stefan
  Webb, and Emanuele Rossi.
\newblock Grand: Graph neural diffusion.
\newblock In {\em International conference on machine learning}, pages
  1407--1418. PMLR, 2021.

\bibitem[CSAS25]{cassel2025bundle}
Jonas Cassel, Fabio Schlindwein, Peter Albers, and Christoph Schn{\"o}rr.
\newblock Bundle scale spaces and local gauge symmetries for graph networks.
\newblock In {\em International Conference on Scale Space and Variational
  Methods in Computer Vision}, pages 245--257. Springer, 2025.

\bibitem[CSTY23]{Handbook:2023aa}
K.~Chen, C.-B. Sch\"{o}lieb, X.-C. Tai, and L.~Younes, editors.
\newblock {\em {Handbook of Mathematical Models and Algorithms in Computer
  Vision and Imaging}}.
\newblock Springer, 2023.

\bibitem[Dav89]{Davies:1989aa}
E.~B. Davies.
\newblock {\em {Heat Kernels and Spectral Theory}}.
\newblock Cambridge University Press, 1989.

\bibitem[DBV17]{defferrardConvolutionalNeuralNetworks2017}
Micha{\"e}l Defferrard, Xavier Bresson, and Pierre Vandergheynst.
\newblock Convolutional {{Neural Networks}} on {{Graphs}} with {{Fast Localized
  Spectral Filtering}}.
\newblock (arXiv:1606.09375), February 2017.

\bibitem[Die25]{diestelGraphTheory2025}
Reinhard Diestel.
\newblock {\em Graph {{Theory}}}, volume 173 of {\em Graduate {{Texts}} in
  {{Mathematics}}}.
\newblock Springer, Berlin, Heidelberg, 2025.

\bibitem[Eva25]{evansSignedNetworksGroup2025}
Tim~S. Evans.
\newblock From {{Signed Networks}} to {{Group Graphs}}, May 2025.

\bibitem[FIMS22]{favoniLatticeGaugeEquivariant2022}
Matteo Favoni, Andreas Ipp, David~I. M{\"u}ller, and Daniel Schuh.
\newblock Lattice gauge equivariant convolutional neural networks.
\newblock {\em Physical Review Letters}, 128(3):032003, January 2022.

\bibitem[GAC{\etalchar{+}}23]{gerkenGeometricDeepLearning2021}
J.~E. Gerken, J.~Aronsson, O.~Carlsson, H.~Linander, F.~Ohlsson, C.~Petersson,
  and D.~Persson.
\newblock {Geometric Deep Learning and Equivariant Neural Networks}.
\newblock {\em Artificial Intelligence Review}, 2023.

\bibitem[Gao21]{gaoDiffusionGeometryFibre2021}
Tingran Gao.
\newblock The diffusion geometry of fibre bundles: {{Horizontal}} diffusion
  maps.
\newblock {\em Applied and Computational Harmonic Analysis}, 50:147--215,
  January 2021.

\bibitem[Gar15]{Garrity:2015aa}
T.~A. Garrity.
\newblock {\em {Electricity and Magnetism for Mathematicians}}.
\newblock {Cambridge University Press}, 2015.

\bibitem[GBM21]{gaoGeometrySynchronizationProblems2019}
T.~Gao, J.~Brodzki, and S.~Mukherjee.
\newblock {The Geometry of Synchronization Problems and Learning Group Actions
  }.
\newblock {\em Discrete \& Computational Geometry}, 65:150--211, 2021.

\bibitem[GL10]{Gattringer:2010zz}
Christof Gattringer and Christian~B. Lang.
\newblock {\em Quantum chromodynamics on the lattice}, volume 788.
\newblock Springer, Berlin, 2010.

\bibitem[GLMY15]{grigor2015cohomology}
Alexander Grigor'yan, Yong Lin, Yuri Muranov, and Shing-Tung Yau.
\newblock Cohomology of digraphs and (undirected) graphs.
\newblock {\em Asian Journal of Mathematics}, 19(5):887--932, 2015.

\bibitem[Gri09]{Grigoryan:2009aa}
A.~Grigor'yan.
\newblock {\em {Heat Kernel and Analysis on Manifolds}}.
\newblock {American Mathematical Society}, 2009.

\bibitem[Gri11]{Grigoryan:2011aa}
A.~Grigoryan.
\newblock {Analysis on Graphs}.
\newblock Lecture Notes, University of Bielefeld, Germany, 2011.

\bibitem[Gro11]{groetschLinearInverseProblems2011}
Charles Groetsch.
\newblock Linear {{Inverse Problems}}.
\newblock In {\em Handbook of {{Mathematical Methods}} in {{Imaging}}}, pages
  3--41. Springer, New York, NY, 2011.

\bibitem[GT87]{grossTopologicalGraphTheory1987}
Jonathan~L. Gross and Thomas~W. Tucker.
\newblock {\em Topological Graph Theory}.
\newblock Willey Interscience Series in Discrete Mathematics and Optimization.
  J.Wiley \& sons, New York Chichester Brisbane [etc], 1987.

\bibitem[HAL07]{hein2007graph}
Matthias Hein, Jean-Yves Audibert, and Ulrike~von Luxburg.
\newblock Graph laplacians and their convergence on random neighborhood graphs.
\newblock {\em Journal of Machine Learning Research}, 8(6), 2007.

\bibitem[Ham17]{hamiltonMathematicalGaugeTheory2017}
M.~J.~D. Hamilton.
\newblock {\em {Mathematical Gauge Theory}}.
\newblock Springer, 2017.

\bibitem[Hat02]{Hatcher:2002aa}
A.~Hatcher.
\newblock {\em {Algebraic Topology}}.
\newblock Cambridge Univ. Press, 2002.

\bibitem[Hat03]{hatcher2003vector}
Allen Hatcher.
\newblock {Vector Bundles and K-theory}.
\newblock {\em http://www. math. cornell. edu/\~{} hatcher}, 2003.

\bibitem[HG19]{hansenSpectralTheoryCellular2019a}
Jakob Hansen and Robert Ghrist.
\newblock Toward a spectral theory of cellular sheaves.
\newblock {\em Journal of Applied and Computational Topology}, 3(4):315--358,
  December 2019.

\bibitem[HG21]{hansenOpinionDynamicsDiscourse2021}
Jakob Hansen and Robert Ghrist.
\newblock Opinion {{Dynamics}} on {{Discourse Sheaves}}.
\newblock {\em SIAM Journal on Applied Mathematics}, 81(5):2033--2060, January
  2021.

\bibitem[HLW20]{huroyanSolvingJigsawPuzzles2020}
Vahan Huroyan, Gilad Lerman, and Hau-Tieng Wu.
\newblock Solving {{Jigsaw Puzzles}} by the {{Graph Connection Laplacian}}.
\newblock {\em SIAM Journal on Imaging Sciences}, 13(4):1717--1753, January
  2020.

\bibitem[JGX{\etalchar{+}}25]{Jia:2025aa}
J.~Jia, J.~Gao, B.~Xue, J.~Wang, Q.~Cai, Q.~Chen, X.~Zhao, P.~Jiang, and
  K.~Gai.
\newblock {From Principles to Applications: A Comprehensive Survey of Discrete
  Tokenizers in Generation, Comprehension, Recommendation, and Information
  Retrieval}.
\newblock {\em arXiv:2502.12448}, 2025.

\bibitem[Jia23]{jiangGaugeTheoryGraphs2023}
Shuhan Jiang.
\newblock Gauge theory on graphs.
\newblock (arXiv:2211.17195), May 2023.

\bibitem[JLYY11]{jiangStatisticalRankingCombinatorial2011a}
Xiaoye Jiang, Lek-Heng Lim, Yuan Yao, and Yinyu Ye.
\newblock Statistical ranking and combinatorial {{Hodge}} theory.
\newblock {\em Mathematical Programming}, 127(1):203--244, March 2011.

\bibitem[Jos17]{Jost:2017aa}
J{\"u}rgen Jost.
\newblock {\em Riemannian {{Geometry}} and {{Geometric Analysis}}}.
\newblock Universitext. Springer International Publishing, Cham, 2017.

\bibitem[JW22]{jennerSteerablePartialDifferential2022}
Erik Jenner and Maurice Weiler.
\newblock Steerable {{Partial Differential Operators}} for {{Equivariant Neural
  Networks}}, April 2022.

\bibitem[KW17]{kipfSemiSupervisedClassificationGraph2017a}
Thomas~N. Kipf and Max Welling.
\newblock Semi-{{Supervised Classification}} with {{Graph Convolutional
  Networks}}.
\newblock In {\em International {{Conference}} on {{Learning
  Representations}}}, February 2017.

\bibitem[Lin94]{Lindeberg:1994aa}
T.~Lindeberg.
\newblock {Scale-space Theory: A Basic Tool for Analyzing Structures at
  Different Scales}.
\newblock {\em Journal of Applied Analysis}, 21(1--2):225--270, 1994.

\bibitem[Lin22]{Lindeberg:2022aa}
T.~Lindeberg.
\newblock {Scale-covariant and Scale-invariant Gaussian Derivative Networks}.
\newblock {\em Journal of Mathematical Imaging and Vision}, 64(3):223--242,
  2022.

\bibitem[LMW05]{leokDiscreteTheoryConnections2005}
Melvin Leok, Jerrold~E. Marsden, and Alan~D. Weinstein.
\newblock A {{Discrete Theory}} of {{Connections}} on {{Principal Bundles}}.
\newblock {\em arXiv:math/0508338}, August 2005.

\bibitem[Mil58]{milnorExistenceConnectionCurvature1958}
John Milnor.
\newblock On the existence of a connection with curvature zero.
\newblock {\em Commentarii Mathematici Helvetici}, 32(1):215--223, December
  1958.

\bibitem[Nak03]{nakaharaGeometryTopologyPhysics2003}
Mikio Nakahara.
\newblock {\em Geometry, {{Topology}} and {{Physics}}, {{Second Edition}}}.
\newblock Taylor \&amp; Francis, June 2003.

\bibitem[Ohs13]{ohsawaSymmetryReductionOptimal2013}
Tomoki Ohsawa.
\newblock Symmetry {{Reduction}} of {{Optimal Control Systems}} and {{
  Principal Connections}}.
\newblock {\em SIAM Journal on Control and Optimization}, 51(1):96--120,
  January 2013.

\bibitem[Pet16]{Petersen:2016aa}
Peter Petersen.
\newblock {\em {Riemannian Geometry}}.
\newblock Springer, 3rd edition, 2016.

\bibitem[PS95]{peskinIntroductionQuantumField1995a}
Michael~E. Peskin and Daniel~V. Schroeder.
\newblock {\em An {{Introduction}} to {{Quantum Field Theory}}}.
\newblock January 1995.

\bibitem[Rot12]{rotheLatticeGaugeTheories2012}
Heinz~J. Rothe.
\newblock {\em Lattice {{Gauge Theories}} : {{An Introduction}} ({{Fourth
  Edition} })}, volume~43.
\newblock World Scientific Publishing Company, 2012.

\bibitem[SS12]{shkolnisky2012viewing}
Yoel Shkolnisky and Amit Singer.
\newblock Viewing direction estimation in cryo-em using synchronization.
\newblock {\em SIAM journal on imaging sciences}, 5(3):1088--1110, 2012.

\bibitem[Ste99]{steenrod1999topology}
Norman Steenrod.
\newblock {\em The topology of fibre bundles}, volume~14.
\newblock Princeton university press, 1999.

\bibitem[STLC24]{Siino:2024aa}
M.~Siino, I.~Tinnirello, and M.~La~Cascia.
\newblock {Is Text Preprocessing Still Worth the Time? A Comparative Survey on
  the Influence of Popular Preprocessing Methods on Transformers and
  Traditional Classifiers}.
\newblock {\em Information Systems}, 121(102342), 2024.

\bibitem[Sus24]{sushch2DDiscreteYang2024}
Volodymyr Sushch.
\newblock {{2D Discrete Yang}}--{{Mills Equations}} on the {{Torus}}.
\newblock {\em Symmetry}, 16(7):823, July 2024.

\bibitem[SW12]{Singer:2012aa}
A.~Singer and H.~T. Wu.
\newblock {Vector Diffusion Maps and the Connection Laplacian}.
\newblock {\em Communications on Pure and Applied Mathematics},
  65(8):1067--1144, 2012.

\bibitem[SZSH11]{singerViewingAngleClassification2011}
A.~Singer, Z.~Zhao, Y.~Shkolnisky, and R.~Hadani.
\newblock Viewing {{Angle Classification}} of {{Cryo-Electron Microscopy Images
  Using Eigenvectors}}.
\newblock {\em SIAM Journal on Imaging Sciences}, 4(2):723--759, January 2011.

\bibitem[Tar05]{Tarantola:2005aa}
A.~Tarantola.
\newblock {\em {Inverse Problem Theory and Methods for Model Parameter
  Estimation}}.
\newblock SIAM, 2005.

\bibitem[TKSL24]{tianMatrixweightedNetworksModeling2024}
Yu~Tian, Sadamori Kojaku, Hiroki Sayama, and Renaud Lambiotte.
\newblock Matrix-weighted networks for modeling multidimensional dynamics,
  October 2024.

\bibitem[TL24]{tianStructuralBalanceRandom2024}
Yu~Tian and Renaud Lambiotte.
\newblock Structural {{Balance}} and {{Random Walks}} on {{Complex Networks}}
  with {{Complex Weights}}.
\newblock {\em SIAM Journal on Mathematics of Data Science}, 6(2):372--399,
  June 2024.

\bibitem[Ton18]{tong2018gauge}
David Tong.
\newblock Gauge theory.
\newblock {\em Lecture notes, DAMTP Cambridge}, 10(8):74, 2018.

\bibitem[vL07]{Luxburg:2007aa}
U.~von Luxburg.
\newblock {A Tutorial on Spectral Clustering}.
\newblock {\em Statistics and Computing}, 17(4):395--416, 2007.

\bibitem[VSP{\etalchar{+}}17]{vaswani2017attention}
Ashish Vaswani, Noam Shazeer, Niki Parmar, Jakob Uszkoreit, Llion Jones,
  Aidan~N Gomez, {\L}ukasz Kaiser, and Illia Polosukhin.
\newblock Attention is all you need.
\newblock {\em Advances in neural information processing systems}, 30, 2017.

\bibitem[Wei98]{Weickert:1998aa}
J.~Weickert.
\newblock {\em {Anisotropic Diffusion in Image Processing}}.
\newblock B.G. Teubner Verlag, 1998.

\bibitem[WFVW21]{weilerCoordinateIndependentConvolutional2021}
Maurice Weiler, Patrick Forr{\'e}, Erik Verlinde, and Max Welling.
\newblock Coordinate {{Independent Convolutional Networks}} -- {{Isometry}} and
  {{Gauge Equivariant Convolutions}} on {{Riemannian Manifolds}}.
\newblock (arXiv:2106.06020), June 2021.

\bibitem[Wip23]{wipfSymmetrienPhysikGruppen2023}
Andreas Wipf.
\newblock {\em {Symmetrien in der Physik: Gruppen- und Darstellungstheorie mit
  Anwendungen}}.
\newblock Springer Berlin Heidelberg, Berlin, Heidelberg, 2023.

\bibitem[Zas82]{Zaslavsky:1982aa}
T.~Zaslavsky.
\newblock {Signed Graphs}.
\newblock {\em Discrete Applied Mathematics}, 4:47--74, 1982.

\end{thebibliography}
